\documentclass[a4paper]{article}
\usepackage{RR}
\usepackage{hyperref}

\usepackage[english]{babel}
\usepackage[latin1]{inputenc}
\usepackage[T1]{fontenc}
\usepackage{float}
\usepackage{amsmath}
\usepackage{amsfonts}
\usepackage{amssymb}
\usepackage{amsthm}
\usepackage{stmaryrd}
\usepackage{url}
\usepackage{epsfig}
\usepackage{natbib}

\def\var{\mbox{Var}}

\psfigdriver{dvips}

\usepackage{mathrsfs}

\usepackage{geometry}
\newfloat{figure}{H}{lof}
\newfloat{Table}{Hptb}{lot} 

\newtheorem{thrm}{Theorem}[section]
\newtheorem{prte}[thrm]{Proposition}
\newtheorem{lemma}[thrm]{Lemma}
\newtheorem{cor}[thrm]{Corollary}
 \newtheorem{defi}{Definition}[section]

\RRdate{Aout 2008}
\RRNo{6616}
\RRversion{2}
\RRdater{Avril  2009}

\RRauthor{Nicolas Verzelen
\thanks{Laboratoire de Math\'ematiques UMR 8628, Universit\'e Paris-Sud, 91405 Osay}
\thanks{INRIA Saclay, Projet SELECT, Universit\'e Paris-Sud, 91405 Osay}}
\authorhead{Verzelen}
\RRtitle{S\'election de mod\`eles en grande dimension pour des design gaussiens}
\RRetitle{High-dimensional Gaussian model selection on a Gaussian design}
\titlehead{Model selection on a Gaussian design}
\RRresume{We consider the problem of estimating the conditional mean of a real Gaussian variable $\nolinebreak Y=\sum_{i=1}^p\nolinebreak\theta_iX_i+\nolinebreak \epsilon$ where the vector of the covariates $(X_i)_{1\leq i\leq p}$ follows a joint Gaussian distribution. This issue often occurs when one aims  at estimating the graph or the distribution of a Gaussian graphical model. We introduce a general model selection procedure which is based on the minimization of a penalized least squares type criterion. It handles a variety of problems such as ordered and complete variable selection, allows to incorporate some prior knowledge on the model and applies when the number of covariates $p$ is larger than the number of observations $n$. Moreover, it is shown to achieve a non-asymptotic oracle inequality independently of the correlation structure of the covariates. We also exhibit various minimax rates of estimation in the considered framework and hence derive adaptivity properties of our procedure.}
\RRabstract{
We consider the problem of estimating the conditional mean of a real Gaussian variable $\nolinebreak Y=\sum_{i=1}^p\nolinebreak\theta_iX_i+\nolinebreak \epsilon$ where the vector of the covariates $(X_i)_{1\leq i\leq p}$ follows a joint Gaussian distribution. This issue often occurs when one aims  at estimating the graph or the distribution of a Gaussian graphical model. We introduce a general model selection procedure which is based on the minimization of a penalized least-squares type criterion. It handles a variety of problems such as ordered and complete variable selection, allows to incorporate some prior knowledge on the model and applies when the number of covariates $p$ is larger than the number of observations $n$. Moreover, it is shown to achieve a non-asymptotic oracle inequality independently of the correlation structure of the covariates. We also exhibit various minimax rates of estimation in the considered framework and hence derive adaptiveness properties of our procedure.}

\RRmotcle{S\'election de mod\`eles, r\'egression lin\'eaire, in\'egalit\'es oracles, mod\`eles graphiques gaussiens, vitesse  minimax d'estimation}
\RRkeyword{ Model selection, Linear regression, oracle inequalities, Gaussian graphical models,  minimax
rate of estimation}
\RRprojets{Select}
\RRtheme{\THCog} 
 \RCSaclay 

\begin{document}
\makeRR   

\section{Introduction}

\subsection{Regression model}

We consider the following regression model
\begin{eqnarray}\label{modele_regression}
Y & = & X\theta + \epsilon\ ,
\end{eqnarray}
where $\theta$ is an unknown vector of $\mathbb{R}^{p}$. The row vector 
$X:=(X_i)_{1\leq i\leq p}$ follows a real zero mean Gaussian distribution with non singular covariance 
matrix $\Sigma$ and $\epsilon$ is a real zero mean Gaussian random variable 
independent of $X$ with variance $\sigma^2$. The variance of $\epsilon$ 
corresponds to the conditional variance of $Y$ given $X$, $\var(Y|X)$. In the sequel, the parameters $\theta$, $\Sigma$, and $\sigma^2$ are considered as unknown. \\ 

Suppose we are given $n$ i.i.d. replications of the vector $(Y,X)$. We respectively write ${\bf Y}$ and ${\bf 
X}$ for the vector of  $n$ observations of $Y$ and the $n\times p$ matrix of observations of $X$. In the present work, we propose a new procedure to estimate the vector $\theta$, when the matrix $\Sigma$ and the variance $\sigma^2$ are both unknown. This corresponds to estimating the conditional expectation of the variable $Y$ given the random vector $X$.
Besides, we want to handle the difficult case of high-dimensional data, i.e. the number of covariates $p$ is possibly much larger than $n$. This estimation problem is equivalent to building a suitable predictor of $Y$ given the covariates $(X_i)_{1\leq i\leq p}$. Classically, we shall use the mean-squared prediction error to assess the quality of our estimation. For any $(\theta_1,\theta_2)\in\mathbb{R}^p$, it is defined by
 \begin{eqnarray}\label{definition_perte}
l(\theta_1,\theta_2) := \mathbb{E}\left[\left(X\theta_1-X\theta_2\right)^2\right]\ .
\end{eqnarray}

\subsection{Applications to Gaussian graphical models (GGM)}

Estimation in the regression model (\ref{modele_regression}) is mainly motivated by the study of Gaussian graphical models (GGM). Let $Z$ be a Gaussian random vector indexed by the elements of a finite set $\Gamma$. The vector $Z$ is a GGM with respect to an undirected graph $\mathcal{G}=(\Gamma,E)$ if for any couple $(i,j)$ which is not contained in the edge set $E$, $Z_i$ and $Z_j$ are independent, given the remaining variables. See Lauritzen \cite{lauritzen96} for definitions and main properties of GGM. Estimating the neighborhood of a given point $i\in \Gamma$ is equivalent to estimating the support of the regression of $Z_i$  with respect to the covariates $(Z_j)_{j\in \Gamma\setminus\{i\}}$. Meinshausen and B\"uhlmann \cite{meinshausen06} have taken this point of view in order to estimate the graph of a GGM. Similarly, we can apply the model selection procedure we shall introduce in this paper to estimate the support of the regression  and therefore the graph $\mathcal{G}$ of a GGM.

Interest in these models has grown since they allow the description of dependence structure of high-dimensional data. As such, they are widely used in spatial statistics \citep{cressie,rue} or probabilistic expert systems \citep{cowell}. More recently, they have been applied to the analysis of microarray data. The challenge is to infer the network regulating the expression of the genes using only a small sample of data, see for instance Sch\"afer and Strimmer \cite{schafer05}, or Wille \emph{et al.} \cite{wille04}.

This has motivated the search for new estimation procedures to handle the linear regression model (\ref{modele_regression}) with Gaussian random design. Finally, let us mention that the model (\ref{modele_regression}) is also of interest when estimating the distribution of directed graphical models or more generally the joint distribution of a large Gaussian random vector. Estimating the joint distribution of a Gaussian vector $(Z_i)_{1\leq i\leq p}$ indeed amounts to estimating the conditional expectations and variance of $Z_i$ given $(Z_j)_{1\leq j\leq i-1}$ for any $1\leq i\leq p$.

\subsection{General oracle inequalities}

Estimation of high-dimensional Gaussian linear models has now attracted a lot of attention. Various procedures have been proposed to perform the estimation of $\theta$ when $p>n$. The challenge at hand it to design estimators that are both computationally feasible and are proved to be efficient. The Lasso estimator has been introduced by Tibshirani \cite{tib96}.
Meinshausen and B\"uhlmann \cite{meinshausen06} have shown that this estimator is consistent under a neighborhood stability condition. These convergence results were refined in the works of Zhao and Yu \cite{Zhao06}, Bunea \emph{et al.} \cite{bunea07}, Bickel \emph{et al.} \cite{bickeltsy08}, or Cand\`es and Plan \cite{CandesPlan07} in a slightly different framework. Cand\`es and Tao \cite{candes08} have also introduced the Dantzig-selector procedure which performs similarly as $l_1$ penalization methods. In the more specific context of GGM, B\"uhlmann and Kalisch \cite{buhlmann07} have analyzed the PC algorithm and have proven its consistency when the GGM follows a faithfulness assumption. All these methods share an attractive computational efficiency and most of them are proven to converge at the optimal rate when the covariates are nearly independent. However, they also share two main drawbacks. First, the $l_1$ estimators are known to behave poorly when the covariates are highly correlated and even for some covariance structures with small correlation (see e.g. \citep{CandesPlan07}). Similarly, the PC algorithm is not consistent if the faithfulness assumption is not fulfilled. Second, these procedures do not allow to integrate some biological or physical prior knowledge. Let us provide two examples. Biologists sometimes have a strong preconception of the underlying biological network thanks to previous experimentations. For instance, Sachs \emph{et al.} \cite{Sachs05}) have produced  multivariate flow cytometry data in order to study a human T cell signaling pathway. Since this pathway has important medical implications, it was already extensively studied and a network is conventionally accepted (see \cite{Sachs05}). For this particular example, it could be more interesting to check whether some interactions were forgotten or some unnecessary interactions were added in the model than performing a complete graph estimation. Moreover, the covariates have in some situations a temporal or spatial interpretation. In such a case, it is natural to introduce an \emph{order} between the covariates, by assuming that a covariate which is \emph{close} (in space or time) to the response $Y$ is more likely to be significant. Hence, an ordered variable selection method is here possibly more relevant than the complete variable selection methods previously mentioned.\\

Let us emphasize the main differences of our estimation setting with related studies in the literature. Birg\'e and Massart \cite{birge2001} consider model selection in a fixed design setting with known variance. Bunea \emph{et al.} \cite{tsybakov_agregation_07}  also suppose that the variance is known. Yet, they consider a random design setting, but they assume that the 
regression functions are bounded (Assumption A.2 in their paper) which is not the case here. Moreover, they obtain risk bounds with respect to the  empirical norm $\|{\bf X}(\widehat{\theta}-\theta)\|^2_n$ and not the integrated loss $l(.,.)$. Here, $\|.\|_n$ refers to the canonical norm in $\mathbb{R}^n$ reweighted by $\sqrt{n}$. As mentioned earlier, our objective is to infer the conditional expectation of $Y$ given $X$. Hence, it is more significant to assess the risk with respect to the loss $l(.,.)$. Baraud \emph{et al.} \cite{baraud08} consider fixed design regression but do not assume that the variance is known.

Our objective is twofold. First, we introduce a general model selection procedure that is very flexible and allows to integrate any prior knowledge on the regression.  We prove non-asymptotic oracle inequalities that hold without any assumption on the correlation structure between the covariates. Second, we obtain non-asymptotic rates of estimation for our model $(\ref{modele_regression})$ that help us to derive adaptive properties for our criterion. \\

In the sequel, a \emph{model} $m$ stands for a subset  of $\{1,\ldots , p\}$. We note $d_m$ the size of $m$ whereas the linear space $S_m$ refers to the set of vectors $\theta\in\mathbb{R}^p$ whose components outside $m$ equal zero. If $d_m$ is smaller than $n$, then we define $\widehat{\theta}_m$ as the least-square estimator of $\theta$ over $S_m$. In the sequel, $\Pi_m$ stands for the projection of $\mathbb{R}^n$ into the space generated by $({\bf X}_i)_{i\in m}$.
Hence, we have the relation ${\bf X}\widehat{\theta}_m=\Pi_m{\bf Y}$. 
Since the covariance matrix $\Sigma$ is non singular, observe that almost surely the rank of $\Pi_m$ is $d_m$.
Given a collection $\mathcal{M}$ of models, our purpose is to select a model $\widehat{m}\in\mathcal{M}$ that exhibits a risk as small as possible with respect to the prediction loss function $l(.,.)$ defined in (\ref{definition_perte}). The model $m^*$  that minimizes the risks $\mathbb{E}[l(\widehat{\theta}_m,\theta)]$ over the whole collection $\mathcal{M}$ is called an oracle. Hence, we want to perform as well as the oracle $\widehat{\theta}_{m^*}$. However, we do not have access to $m^*$ as it requires the knowledge of the true vector $\theta$. A classical method to estimate a \emph{good} model $\widehat{m}$ is achieved through \emph{penalization} with respect to the complexity of models. In the sequel, we shall select the model $\widehat{m}$ as
\begin{eqnarray}\label{methode_penalisation} 
  \widehat{m}:= \arg\min_{m\in\mathcal{M}}\text{Crit}(m):=\arg\min_{m\in\mathcal{M}}\|{\bf Y} -\Pi_m {\bf Y}\|_n^2\left[1+pen(m)\right]\ ,
\end{eqnarray}
where $pen(.)$ is a positive function defined on $\mathcal{M}$. Besides, we recall that $\|.\|_n$ refers to the canonical norm in $\mathbb{R}^n$ reweighted by $\sqrt{n}$. Observe that $Crit(m)$ is the sum of the least-square error $\|{\bf Y} -\Pi_m {\bf Y}\|_n^2$ and a penalty term $pen(m)$ rescaled by the least-square error in order to come up with the fact that the conditional variance  $\sigma^2$ is unknown. We precise in Section \ref{section_procedure} the heuristics underlying this model selection criterion. Baraud \emph{et al.} \cite{baraud08} have extensively studied this penalization method in the fixed design Gaussian regression framework with unknown variance. In their introduction, they explain how one may retrieve classical criteria like  AIC \cite{akaike}, BIC \cite{schwarz78}, and FPE \cite{akaike69} by choosing a suitable penalty function $pen(.)$.\\

This model selection procedure is really flexible through  the choices of the collection $\mathcal{M}$ and of the penalty function $pen(.)$. Indeed, we may perform complete variable selection by taking the collection of subsets of $\{1,\ldots ,p\}$ whose is smaller than some integer $d$. Otherwise, by taking a nested collection of models, one performs ordered variable selection. We give more details in Sections \ref{section_procedure} and \ref{section_oracle}. If one has some prior idea on the true model $m$, then one could only consider the collection of models that are close in some sense to $m$. Moreover, one may also give a Bayesian flavor to the penalty function $pen(.)$ and hence specify some prior knowledge on the model.  \\

First, we state a non-asymptotic oracle inequality when the complexity of the collection $\mathcal{M}$ is small and for penalty functions $pen(m)$ that are larger than $Kd_m/(n-d_m)$ with $K>1$. Then, we prove that the FPE criterion of Akaike \cite{akaike69} which corresponds to the choice $K=2$ achieves an asymptotic exact oracle inequality for the special case of ordered variable selection. For the sake of completeness, we prove that choosing $K$ smaller than one yields to terrible performances.  \\

In Section \ref{Section_oracle_general}, we consider general collection of models $\mathcal{M}$. By introducing new penalties that take into account the complexity of $\mathcal{M}$  as in \citep{massart_pente}, we are able to state a non-asymptotic oracle inequality. In particular, we consider the problem of complete variable selection. In Section \ref{section_prior}, we define penalties based on a prior distribution on $\mathcal{M}$. We then derive the corresponding risk bounds.\\

Interestingly, these rates of convergence do not depend on the covariance matrix $\Sigma$ of the covariates, whereas known results on the Lasso or the Dantzig selector rely on some assumptions on $\Sigma$, as discussed in Section \ref{Section_oracle_general}.  We illustrate in Section \ref{section_simulation} on simulated examples that for some covariance matrices $\Sigma$ the Lasso performs poorly whereas our methods still behaves well. Besides, our penalization method does not require the knowledge of the conditional variance $\sigma^2$. In contrast, the Lasso and the Dantzig selector are constructed for known variance. Since $\sigma^2$ is unknown, one either has to estimate it or has to use a cross-validation method in order to calibrate the penalty. In both cases, there is some room for improvements for the practical calibration of these estimators.

However, our model selection procedure suffers from a computational cost that depends linearly on the size of the collection $\mathcal{M}$.
For instance, the complete variable selection problem is NP-hard. This makes it intractable when $p$ becomes too large (i.e. more than 50). In contrast, our criterion applies for arbitrary $p$ when considering ordered variable selection since the size of $\mathcal{M}$ is linear with $n$. We shall mention in the discussion some possible extensions that we hope can cope with the computational issues.\\

In a simultaneous and independent work to ours, Giraud \cite{giraud08} applies an analogous procedure to estimate the graph of a GGM. Using slightly different techniques, he obtains non-asymptotic results that are complementary to ours. However, he performs an unnecessary thresholding to derive an upper bound of the risk.
Moreover, he does not consider the case of nested collections of models as we do in Section \ref{ordered_selection}. Finally, he does not derive minimax rates of estimation.

\subsection{Minimax rates of estimation}

In order to assess the optimality of our procedure, we investigate in Section \ref{Section_minimax} the minimax rates of estimation for ordered and complete variable selection. For ordered variable selection, we compute the minimax rate of estimation over ellipsoids which is analogous to the rate obtained in the fixed design framework. We derive that our penalized estimator is adaptive to the collection of ellipsoids independently of the covariance matrix $\Sigma$. For complete variable selection, we prove that the minimax rates of estimator of vectors $\theta$ with at most $k$ non-zero components is of order $\frac{k\log p}{n}$ when the covariates are independent. This is again coherent with the situation observed in the fixed design setting. Then, the estimator $\widetilde{\theta}$ defined for complete variable selection problem is shown to be adaptive to any sparse vector $\theta$. Moreover, it seems that the minimax rates may become faster when the matrix $\Sigma$ is far from identity. We investigate this phenomenon in Section \ref{section_sparse}. All these minimax rates of estimation are, to our knowledge, new in the Gaussian random design regression. Tsybakov \cite{tsybakov_minimax} has derived minimax rates of estimation in a general random design regression setup, but his results do not apply in our setting as explained in Section \ref{section_sparse}.

\subsection{Organization of the paper and some notations}

In Section \ref{section_procedure}, we precise our estimation procedure and explain the heuristics underlying the penalization method. The main results are stated in Section \ref{section_oracle}. In Section \ref{Section_minimax}, we derive the different minimax rates of estimation and assess the adaptivity of the penalized estimator $\widehat{\theta}_{\widehat{m}}$. We perform a simulation study and compare the behaviour of our estimator with Lasso and adaptive Lasso in Section \ref{section_simulation}. Section \ref{section_discussion} contains a final discussion and some extensions, whereas the proofs are postponed to Section \ref{section_proofs}.\\

Throughout the paper, $\|.\|_n^2$ stands for the square of the canonical norm in $\mathbb{R}^n$ reweighted by $n$. For any vector $Z$ of size $n$, we recall that $\Pi_mZ$ denotes the orthogonal projection of $Z$ onto the space generated by $({\bf X}_i)_{i\in m}$. The notation $X_m$ stands for $(X_i)_{i\in m}$ and ${\bf X}_m$ represents the $n\times d_m$ matrix of the $n$ observations of $X_m$. For the sake of simplicity, we write $\widetilde{\theta}$ for the penalized estimator $\widehat{\theta}_{\widehat{m}}$. For any $x>0$, $\lfloor x\rfloor$ is the largest integer smaller than $x$ and $\lceil x\rceil$ is the smallest integer larger than $x$. Finally, $L$, $L_1$, $L_2$,$\ldots$ denote universal constants that may vary from line to line. The notation $L(.)$ specifies the dependency on some quantities. 

\section{Estimation procedure}\label{section_procedure}

Given a collection of models $\mathcal{M}$ and a penalty $pen:\mathcal{M}\rightarrow \mathbb{R}^{+}$, the estimator $\widetilde{\theta}$ is computed as follows:\\

\fbox{
\begin{minipage}{0.9\textwidth}
{\bf Model selection procedure}
\begin{enumerate}
 \item  Compute $\widehat{\theta}_m=\arg\min_{\theta'\in S_m}\|Y-X\theta'\|^2_n$  for all  models $m\in\mathcal{M}$.
\item Compute $\widehat{m}:= \arg\min_{m\in\mathcal{M}}\|{\bf Y} -{\bf X}\widehat{\theta}_m\|_n^2\left[1+pen(m)\right]$.
\item $\widetilde{\theta}:=\widehat{\theta}_{\widehat{m}}$.
\end{enumerate}
\end{minipage}
}
\vspace{0.5cm}

The choice of the collection $\mathcal{M}$ and the penalty function $pen(.)$ depends on the problem under study. In what follows, we provide some preliminary results for the parametric estimators $\widehat{\theta}_m$ and we give an heuristic explanation for our penalization method.

For any vector $\theta'$ in $\mathbb{R}^p$, we define the mean-squared error $\gamma(.)$ and its empirical counterpart $\gamma_n(.)$ as 
\begin{eqnarray}\label{defi_gamman}
 \gamma(\theta') := \mathbb{E}_{\theta}\left[\left(Y-X\theta'\right)^2\right] \hspace{0.5cm} \text{and}\hspace{0.5cm}\gamma_n (\theta'):= \left\|{\bf Y} - {\bf X}\theta'\right\|_n^2\ .
\end{eqnarray}
The function $\gamma(.)$ is closely connected to the loss function $l(.,.)$ through the relation
$l(\beta,\theta) = \gamma(\beta) - \gamma(\theta)$.

Given a model $m$ of size strictly smaller than $n$, we refer to $\theta_m$ as the unique minimizer of $\gamma(.)$ over the subset $S_m$. It then follows that
$\mathbb{E}\left(Y|X_m\right) = \sum_{i\in m}\theta_iX_i$  and $\gamma(\theta_m)$ is the conditional variance of $Y$ given $X_m$.
As for it, the least squares estimator $\widehat{\theta}_m$ is the  minimizer of  $\gamma_n(.)$ over the space $S_m$.
\begin{eqnarray*}
 \widehat{\theta}_m := \arg \min_{\theta'\in S_m} \gamma_n(\theta')\, \text{ a.s. }.
\end{eqnarray*}
It is almost surely uniquely defined since $\Sigma$ is assumed to be non-singular and since $d_m<n$.
Besides $\gamma_n(\widehat{\theta}_m)$ equals $\| {\bf Y}-\Pi_m{\bf Y}\|_n^2$. Let us derive two simple properties of $\widehat{\theta}_m$ that will give us some hints to perform model selection.
\begin{lemma}\label{prte_basique}
For any model $m$ whose dimension is smaller than $n-1$, the expected mean-squared error of $\widehat{\theta}_m$ and the expected least squares of $\widehat{\theta}_m$ respectively equal
\begin{eqnarray}\label{esperance_gamma}
\mathbb{E}\left[\gamma(\widehat{\theta}_m)\right] & = &\left[l(\theta_m,\theta)+\sigma^2\right]
\left(1 + \frac{d_m}{n-d_m-1}\right)\ ,\\
\label{esperance_gamman}
\mathbb{E} \left[\gamma_n(\widehat{\theta}_m)\right] & = & \left[l(\theta_m,\theta)+\sigma^2\right]
\left(1 -\frac{d_m}{n}\right)\ .
\end{eqnarray}
\end{lemma}
The proof is postponed to the Appendix. From Equation (\ref{esperance_gamma}), we derive a bias variance decomposition of the risk of the estimator $\widehat{\theta}_m$:
\begin{eqnarray*}
 \mathbb{E}\left[l(\widehat{\theta}_m,\theta)\right] & = & l(\theta_m,\theta) + \left[\sigma^2 + l(\theta_m,\theta)\right]\frac{d_m}{n-d_m-1} \ .
\end{eqnarray*}
Hence,  $\widehat{\theta}_m$ converges to $\theta_m$ in probability when $n$ converges to infinity. Contrary to the fixed design regression framework, the variance term $\left[\sigma^2 + l(\theta_m,\theta)\right]\frac{d_m}{n-d_m-1}$ depends on the bias term $l(\theta_m,\theta)$. Besides, this variance term does not necessarily increase when the dimension of the model increases. \\

Let us now explain the idea underlying our model selection procedure. We aim at choosing a model $\widehat{m}$ that nearly minimizes the mean-squared error $\gamma(\widehat{\theta}_m)$. Since we do not have access to  $\gamma(\widehat{\theta}_m)$ nor to the bias $l(\theta_m,\theta)$, we perform an unbiased estimation of the risk as done by Mallows  \cite{mallows73} in the fixed design framework.
\begin{eqnarray}
\gamma\left(\widehat{\theta}_m\right)& \approx & \gamma_n\left(\widehat{\theta}_m\right)+ \mathbb{E}\left[\gamma\left(\widehat{\theta}_m\right)-\gamma_n\left(\widehat{\theta}_m\right)\right]\nonumber\\& \approx &\gamma_n\left(\widehat{\theta}_m\right) + \mathbb{E}\left[\gamma_n\left(\widehat{\theta}_m\right)\right]\frac{d_m}{n-d_m}\left[2+\frac{d_m+1}{n-d_m-1}\right]\nonumber \\  
& \approx & \gamma_n\left(\widehat{\theta}_m\right)\left[1+\frac{d_m}{n-d_m}\left(2+\frac{d_m+1}{n-d_m-1}\right)\right] \ .\label{critere_heuristique}
\end{eqnarray}
By Lemma \ref{prte_basique}, these approximations are in fact equalities in expectation. Since the last expression only depends on the data, we may compute its minimizer over the collection $\mathcal{M}$. This approximation is effective and minimizing $(\ref{critere_heuristique})$ provides a good estimator $\widetilde{\theta}$ when the size of the collection $\mathcal{M}$ is moderate as stated in Theorem \ref{ordered_selection}. We recall that $\|{\bf Y}-\Pi_{m}{\bf Y}\|_n^2$ equals $\gamma_n(\widehat{\theta}_m)$. Hence, our previous heuristics would lead to a choice of penalty $pen(m)=\frac{d_m}{n-d_m}\left(2+\frac{d_m+1}{n-d_m-1}\right)$ in our criterion (\ref{methode_penalisation}), whereas FPE criterion corresponds to $pen(m)=\frac{2d_m}{n-d_m}$. These two penalties are equivalent when the dimension $d_m$ is small in front of $n$. In Theorem \ref{ordered_selection}, we explain why these criteria allow to derive approximate oracle inequalities when there is a small number of models. However, when the size of the collections $\mathcal{M}$ increases, we need to design other penalties that take into account the complexity of the collection $\mathcal{M}$ (see Section \ref{Section_oracle_general}).

\section{Oracle inequalities}\label{section_oracle}

\subsection{A small number of models}\label{Section_orderedselection}

In this section, we restrict ourselves to the situation where the collection of models $\mathcal{M}$ only contains a small number of models as defined in \cite{massart_pente} Sect 3.1.2.\\  

($\mathbb{H}_{Pol}$): for each
$d\geq 1$ the number of models $m\in\mathcal{M}$ such that $d_m=d$ grows at most polynomially with respect to $d$. In other words, there exists $\alpha$ and $\beta$ such that for any $d\geq 1$,
$\text{Card}\left(\left\{m\in\mathcal{M},\ d_m=d\right\}\right)\leq \alpha d^{\beta}$.~\\

($\mathbb{H}_{\eta}$): The dimension $d_m$ of every model $m$ in $\mathcal{M}$ is smaller than $\eta n$. Moreover, the number of observations $n$ is larger than $6/(1-\eta)$.\\

Assumption ($\mathbb{H}_{Pol}$) states that there is at most a polynomial number of models with a given dimension. It includes in particular the problem of ordered variable selection, on which we will focus in this section. Let us introduce the collection of models relevant for this issue. For any positive number $i$ smaller or equal to $p$, we define the model $m_i:=\{1,\ldots, i\}$ and the nested collection $\mathcal{M}_i:= \left\{m_0,m_1,\ldots m_i\right\}$. Here, $m_0$ refers to the empty model. Any collection $\mathcal{M}_i$ satisfies ($\mathbb{H}_{Pol}$) with $\beta=0$ and $\alpha=1$.

\begin{thrm}\label{ordered_selection}
Let $\eta$ be any positive number smaller than one. Assume that the collection $\mathcal{M}$ satisfies ($\mathbb{H}_{Pol}$) and ($\mathbb{H}_{\eta}$).
If the penalty $pen(.)$ is lower bounded as follows   
\begin{eqnarray}
pen(m) \geq  K\frac{d_m}{n-d_m} \, \, \text{for
  all $m\in\mathcal{M}$ and some $K>1$}\ ,
\label{hypothese1_penalite}
\end{eqnarray}
then 
\begin{eqnarray}
\mathbb{E}\left[l(\widetilde{\theta},\theta)\right] \leq 
 L(K,\eta)\inf_{m\in\mathcal{M}}\left[l(\theta_m,\theta)+\frac{n-d_m}{n}pen(m)\left[\sigma^2+l(\theta_m,\theta)\right]\right]
 + \tau_n\ ,
\end{eqnarray}
where the error term $\tau_n$ is defined as
\begin{eqnarray*}
\tau_n=\tau_n\left[\var(Y),K,\eta,\alpha,\beta\right]:=L_1(K,\eta,\alpha,\beta)\left[
\frac{\sigma^2}{n}+n^{3+\beta}\var(Y)\exp\left[-nL_2(K,\eta)\right]\right]\ ,
\end{eqnarray*}
and $L_2(K,\eta)$ is positive.
\end{thrm}
The theorem applies for any $n$, any $p$ and  there is no hidden dependency on $n$ or $p$ in the constants. Besides, observe that the theorem does not depend at all on the covariance matrix $\Sigma$ between the covariates.
If we choose the penalty $pen(m)=K\frac{d_m}{n-d_m}$, we obtain an approximate oracle inequality.
\begin{eqnarray*}
\mathbb{E}\left[l(\widetilde{\theta},\theta)\right] \leq 
 L(K,\eta)\inf_{m\in\mathcal{M}}\mathbb{E}\left[l(\widehat{\theta}_m,\theta)\right]
 + \tau_n\left[\var(Y),K,\eta,\alpha,\beta\right]\ ,
\end{eqnarray*}
thanks to Lemma \ref{prte_basique}. The term in $n^{3+\beta}\var(Y)\exp[-nL_2(K,\eta)]$ converges exponentially fast to 0 when $n$ goes to infinity and is therefore considered as negligible.  One interesting feature of this oracle inequality is that it allows to consider models of dimensions as close to $n$ as we want providing that $n$ is large enough. This will not be possible in the next section when handling more complex collections of models.\\

If we have stated that $\widetilde{\theta}$ performs almost as well as the oracle model, one may wonder whether it is possible to perform exactly as well as the oracle. In the next proposition, we shall prove that under additional assumption the estimator $\widetilde{\theta}$ with $K=2$ follows an asymptotic exact oracle inequality. We state the result for the problem of ordered variable selection. Let us assume for a moment that the set of covariates is infinite, i.e. $p=+\infty$. In this setting, we define the subset $\Theta$ of sequences $\theta=(\theta_i)_{i\geq 1}$ such that $<X,\theta>$ converges in $L^2$. In the following proposition, we assume that $\theta\in\Theta$.
\begin{defi}\label{Definition_ellipsoides_complexe}
 Let $s$ and $R$  be two positive numbers. We define the so-called ellipsoid $\mathcal{E}'_s(R)$ as 
\begin{eqnarray*}
 \mathcal{E}'_s(R):= \left\{(\theta_i)_{i\geq 0} ,\hspace{0.5cm} \sum_{i=1}^{+\infty} \frac{l\left(\theta_{m_{i-1}},\theta_{m_{i}}\right)}{i^{-s}}\leq R^2\sigma^2\right\}\ .
\end{eqnarray*}
\end{defi}
In Section \ref{section_adapt_ellips}, we explain why we call this set $\mathcal{E}'_s(R)$ an ellipsoid.
\begin{prte} \label{optimal} Assume there exists $s$, $s'$, and $R$ such that $\theta\in \mathcal{E}'_s(R)$ and such that for any positive numbers $R'$, $\theta \notin \mathcal{E}'_{s'}(R')$. We consider the collection $\mathcal{M}_{\lfloor n/2\rfloor }$ and the penalty $pen(m)= 2\frac{d_m}{n-d_m}$. Then, there exists a constant $L(s,R)$ and a sequence $\tau_n$ converging to zero at infinity such that, with probability, at least $1-L(s,R)\frac{\log n}{n^2}$ ,
\begin{eqnarray}\label{inegalite_oracle_asymptotiques}
 l(\widetilde{\theta},\theta)\leq \left[1+\tau(n)\right]\inf_{m\in\mathcal{M}_{\lfloor n/2\rfloor}}l(\widehat{\theta}_m,\theta)\ . 
\end{eqnarray}
\end{prte}

Admittedly, we make $n$ go to the infinity in this proposition but we are still in a high dimensional setting since $p=+\infty$ and since the size of the collection $\mathcal{M}_{\lfloor n/2 \rfloor }$ goes to infinity with $n$. Let us  briefly discuss the assumption on $\theta$. Roughly speaking, it ensures that the oracle model has a dimension not too close to zero (larger than $\log^2(n)$) and  small before $n$ (smaller than $n/\log n$).  Notice that it is classical to assume that the bias is non-zero for every model $m$ for proving the asymptotic optimality of Mallows' $C_p$ (\emph{cf.} Shibata~\cite{shibata81} and Birg\'e and Massart~\cite{massart_pente}). Here, we make a stronger assumption because the bound (\ref{inegalite_oracle_asymptotiques}) holds in probability and because the design is Gaussian. Moreover, our stronger assumption has already been made by Stone~\cite{stone} and Arlot~\cite{arlot}.
We refer to Arlot~\cite{arlot} Sect.4.1 for a more complete discussion of this assumption.

The choice of the collection $\mathcal{M}_{\lfloor n/2 \rfloor}$ is arbitrary and one can extend it to many collections that satisfy $(\mathbb{H}_{Pol})$ and $(\mathbb{H}_{\eta})$. As mentioned in Section \ref{section_procedure}, the penalty $pen(m)=2\frac{d_m}{n-d_m}$ corresponds to the FPE model selection procedure. In conclusion, the choice of the FPE criterion turns out to be asymptotically optimal when the complexity of $\mathcal{M}$ is small.\\

We now underline that the condition $K>1$ in Theorem \ref{ordered_selection} is almost necessary. Indeed, choosing $K$ smaller than one yields terrible statistical performances.
\begin{prte}\label{Pente_ordonnee}
Suppose that $p$ is larger than $n/2$. Let us consider the collection $\mathcal{M}_{\lfloor n/2\rfloor}$
and assume that for some $\nu>0$,
\begin{eqnarray}\label{condition_petite_penalite}
 pen(m)= (1-\nu)\frac{d_m}{n-d_m}\ ,
\end{eqnarray}
for any model $m\in\mathcal{M}_{\lfloor n/2\rfloor}$. Then given $\delta\in (0,1)$, there exists some $n_0(\nu,\delta)$ only depending on $\nu$ and $\delta$ such that for $n\geq n_0(\nu,\delta)$,
\begin{eqnarray*}
 \mathbb{P}_{\theta}\left[d_{\widehat{m}}\geq \frac{n}{4}\right]\geq 1-\delta\  \text{ and }\ \mathbb{E}\left[l(\widetilde{\theta},\theta)\right]\geq l(\theta_{m_{\lfloor n/2\rfloor }},\theta)+L(\delta,\nu)\sigma^2\ .
\end{eqnarray*}
\end{prte}
If one chooses a too small penalty, then the dimension $d_{\widehat{m}}$ of the selected model is huge and the penalized estimator $\widetilde{\theta}$ performs poorly. The hypothesis $p\geq n/2$ is needed for defining the collection $\mathcal{M}_{\lfloor n/2\rfloor}$. Once again, the choice of the collection $\mathcal{M}_{\lfloor n/2\rfloor}$ is rather arbitrary and the result of Proposition \ref{Pente_ordonnee} still holds for collections $\mathcal{M}$ which satisfy $(\mathbb{H}_{Pol})$ and $(\mathbb{H}_{\eta})$ and contain at least one model of large dimension. Theorem \ref{ordered_selection} and Proposition \ref{Pente_ordonnee} tell us that $\frac{d_m}{n-d_m}$ is the minimal penalty.\\

In practice, we advise to choose $K$ between $2$ and $3$. Admittedly,  $K=2$ is asymptotically optimal by Proposition \ref{optimal}. Nevertheless, we have observed on simulations that $K=3$ gives slightly better results when $n$ is small. For ordered variable selection, we suggest to take the collection $\mathcal{M}_{\lfloor n/2\rfloor}$.

\subsection{A general model selection theorem}\label{Section_oracle_general}

In this section, we study the performance of the penalized estimator $\widetilde{\theta}$ for general collections $\mathcal{M}$. Classically, we need to penalize stronger the models $m$, incorporating the complexity of the collection. As a special case, we shall consider the problem of complete variable selection. This is why we define the collections $\mathcal{M}_p^d$ that consist of all subsets of $\{1,\ldots,p\}$ of size less or equal to $d$.

\begin{defi}\label{complexite_modele}
Given a collection $\mathcal{M}$, we define the function $H(.)$ by
\begin{eqnarray*}
H(d) := \frac{1}{d}\log\left[\text{Card}\left(\left\{m\in\mathcal{M},\, d_m=d \right\}\right)\right]\ ,
\end{eqnarray*}
for any integer $d\geq1$.
\end{defi}
This function measures the complexity of the collection $\mathcal{M}$.  For the collection $\mathcal{M}_p^d$, $H(k)$ is upper bounded by $\log(ep/k)$ for any $k\leq d$ (see Eq.(4.10) in \cite{massartflour}). Contrary to the situation encountered in ordered variable selection, we are not able to consider models of arbitrary dimensions and we shall do the following assumption.~\\

$(\mathbb{H}_{K,\eta})$: Given $K>1$ and $\eta>0$, the collection $\mathcal{M}$ and the number $\eta$ satisfy
\begin{eqnarray}\label{Hypothese_HK}
\forall m\in\mathcal{M}, \,\,\,\,\,\,	\frac{\left[1+\sqrt{2H(d_m)}\right]^2d_m}{n-d_m}\leq \eta<\eta(K)\ ,
\end{eqnarray}
where $\eta(K)$ is defined as
 $\eta(K):= [1-2(3/(K+2))^{1/6}]^2\bigvee [1-(3/K+2)^{1/6}]^2/4$. ~\\

The function $\eta(K)$ is  positive and increases when $K$ is larger than one. Besides, $\eta(K)$ converges to one when $K$ converges to infinity. We do not claim that the expression of $\eta(K)$ is optimal. We are more interested in its behavior when $K$ is large. 

\begin{thrm}\label{thrm_principal}
Let $K>1$ and let $\eta<\eta(K)$. Assume that $n$ is larger than some quantity $n_0(K)$ only depending on $K$ and the collection $\mathcal{M}$ satisfies $(\mathbb{H}_{K,\eta})$. 
If the penalty	$pen(.)$ is lower bounded as follows  
\begin{eqnarray}\label{penalite_complexe}
pen(m)\geq K\frac{d_m}{n-d_m}\left(1+\sqrt{2H(d_m)}\right)^2
\, \, \, \text{for any $m\in \mathcal{M}$}\ ,
\end{eqnarray}
then 
\begin{eqnarray}\label{majoration_risque}
\mathbb{E}\left[l(\widetilde{\theta},\theta)\right] \leq 
 L(K,\eta)\inf_{m\in\mathcal{M}}\left\{l(\theta_m,\theta)+\frac{n-d_m}{n}pen(m)\left[\sigma^2+l(\theta_m,\theta)\right]\right\}
 + \tau_n\ ,
\end{eqnarray}
where $\tau_n$ is defined as
\begin{eqnarray*}
\tau_n=\tau_n\left[\var(Y),K,\eta\right] := \sigma^2\frac{L_1(K,\eta)}{n}+L_2(K,\eta)n^{5/2}\var(Y)\exp\left[-nL_3(K,\eta)\right] \ ,
\end{eqnarray*}
and $L_3(K,\eta)$ is positive.
\end{thrm}
This theorem provides an oracle type inequality of the same type as the one obtained in the Gaussian sequential framework by Birg\'e and Massart \cite{birge2001}. The risk of the penalized estimator $\widetilde{\theta}$ almost achieves the infimum of the risks plus a penalty term depending on the function $H(.)$. As in Theorem \ref{ordered_selection}, the error term $\tau_n\left[\var(Y),K,\eta\right]$ depends on $\theta$ but this part goes exponentially fast to 0 with $n$. ~\\

\textbf{Comments:}
\begin{itemize}
 \item  As for Theorem \ref{ordered_selection}, the result holds for arbitrary large $p$ as long as $n$ is larger than the quantity $n_0(K)$ (independent of $p$). There is no hidden dependency on $p$ except in the complexity function $H(.)$ and Assumption $\mathbb{H}_{K,\eta}$ that  we shall discuss for the particular case of complete variable selection. Moreover, one may easily check Assumption $\mathbb{H}_{K,\eta}$ since it only depends on the collection $\mathcal{M}$ and not on some unknown quantity.
\item This result  (as well as of Theorem \ref{ordered_selection}) does not depend at all on the covariance matrix $\Sigma$ between the covariates. 
\item  The penalty introduced in this theorem only depends on the collection $\mathcal{M}$ and a number $K>1$. Hence, performing the procedure does not require any knowledge on $\sigma^2$, $\Sigma$, or $\theta$.
We give hints at the end of the section for choosing the constant $K$. 
\item Observe that Theorem \ref{ordered_selection} is not just corollary of Theorem \ref{thrm_principal}. If we apply Theorem \ref{thrm_principal} to the problem of ordered selection, then the maximal size of the model has to be smaller than $n\frac{\eta(K)}{1+\eta(K)}$, which depends on $K$ and is always smaller than $n/2$. In contrast, Theorem \ref{ordered_selection} handles models of size up to $n-7$.
\end{itemize}

\subsection{Application to complete variable selection}

Let us now restate Theorem \ref{thrm_principal} for the particular issue of complete variable selection. Consider $K>1$, $\eta<\eta(K)$ and $d>1$ such that $\mathcal{M}_p^d$ satisfies Assumption $(\mathbb{H}_{K,\eta})$. If we take for any model $m\in\mathcal{M}_p^d$ the penalty term 
\begin{eqnarray}\label{penalite_complete}
pen(m)=K\frac{d_m}{n-d_m}\left[1+\sqrt{2\log\left(\frac{ep}{d_{m}}\right)}\right]^2\ ,
\end{eqnarray} then we get
\begin{eqnarray*}
 \mathbb{E}\left[l(\widetilde{\theta},\theta)\right] \leq  L(K,\eta)\inf_{m\in\mathcal{M}_p^{d}}\left\{l(\theta_m,\theta)+\frac{d_m}{n}\log\left(\frac{ep}{d_m}\right)\sigma^2\right\}
 + \tau_n\left[\var(Y),K,\eta\right]\ .
\end{eqnarray*}

We shall prove in Section \ref{section_sparse}, that the term $\log(p/d_m)$ is unavoidable and that the obtained estimator is optimal from a minimax point of view. If the true parameter $\theta$ belongs to some unknown model $m$, then the rates of estimation of $\tilde{\theta}$ is of the order $\frac{d_m}{n}\log(p/d_m)\sigma^2$.  Let us compare our result with other procedures.

\begin{itemize}
\item The oracle type inequalities look similar to the ones obtained by Birg\'e and Massart \cite{birge2001}, Bunea \emph{et al.} \cite{tsybakov_agregation_07} and Baraud \emph{et al.} \cite{baraud08}. However,  Birg\'e and Massart and Bunea \emph{et al.}  assume that the variance $\sigma^2$ is known. Moreover, Birg\'e and Massart and Baraud \emph{et al.} only consider a fixed design setting. Yet, Bunea \emph{et al.} allow the design to be random, but they assume that the regression functions are bounded (Assumption A.2 in their paper) which is not the case here. Moreover, they only get risk bounds with respect to the empirical norm $\|.\|_n$ and not the integrated loss $l(.,.)$.
	
 \item As mentioned previously, our oracle inequality holds for any covariance matrix $\Sigma$. In contrast, Lasso and Dantzig selector estimators have been shown to satisfy oracle inequalities under assumptions on the empirical design ${\bf X}$. In \cite{candes08},  Cand\`es and Tao indeed assume  that the singular values of ${\bf X}$ restricted to any subset of size proportional to the sparsity of $\theta$ are bounded away from zero. Bickel \emph{et al.} \cite{bickeltsy08} introduce an extension of this condition prove both for the Lasso and the Dantzig selector. In a recent work \cite{CandesPlan07}, Cand\`es and Plan state that if the empirical correlation between the covariates is smaller than $L(\log p)^{-1}$,then the Lasso follows an oracle inequality in a majority of cases. Their condition is in fact almost necessary. On the one hand, they give examples of  some low correlated situations, where the Lasso performs poorly. On the other hand, they prove that the Lasso fails to work well if the correlation between the covariates if larger than $L(\log p)^{-1}$.  Yet,  Cand\`es and Plan consider the loss function $\|{\bf X}\widehat{\theta}-{\bf X}\theta\|_n^2$, whereas we use the \emph{integrated} loss $l(\widehat{\theta},\theta)$, but this does not really change the impact of their result. We refer to their paper for  further details. 
The main point is that for some correlation structures, our procedure still works well, whereas the Lasso and the Dantzig selector procedures perform poorly. In many problems such as GGM estimation, the correlation between the covariates may be high and even the relaxed assumptions of Cand\`es and Plan  may not be fulfilled.
In Section \ref{section_simulation}, we illustrate this phenomenon by comparing  our procedure with the Lasso  on numerical examples for independent and highly correlated covariates.

\item Suppose that  the covariates are independent and that $\theta$ belongs to some model $m$, the rates of convergence of the Lasso is then of the order  $\frac{d_m}{n}\log(p)\sigma^2$, whereas ours is $\frac{d_m}{n}\log(p/d_m)\sigma^2$. Consider the case where $p$, and $d_m$ are of the same order whereas $n$ is large. Our model selection procedure therefore outperforms the Lasso by a $\log(p)$ factor even if the covariates are independent.

\item  Let us restate Assumption $(\mathbb{H}_{K,\eta})$ for the particular collection $\mathcal{M}_p^d$. Given some $K>1$ and some $\eta<\eta(K)$, the collection $\mathcal{M}_p^d$  satisfies $(\mathbb{H}_{K,\eta})$ if 
\begin{eqnarray}\label{condition_estimation_complete}
 d\leq \eta\frac{n}{1+\left[1+\sqrt{2(1+\log(p/d))}\right]^2}\ .	
\end{eqnarray}
If $p$ is much larger than $n$, the dimension $d$ of the largest model has to be be smaller than the order $\eta\frac{n}{2\log(p)}$. Cand\`es and Plan state a similar condition for the lasso. We believe that this condition is unimprovable. Indeed, Wainwright states in Th.2 of \cite{wainwright07} a result going in this sense: it is impossible to estimate reliably the support of a $k$-sparse vector  $\theta$ if $n$ is smaller than the order $k\log(p/k)$. If $\log(p)$ is larger than $n$, then we cannot apply Theorem \ref{thrm_principal}. This ultra-high dimensional setting is also not handled by the theory for the Lasso and the Dantzig selector.
Finally, if $p$ is of the same order as $n$, then Condition (\ref{condition_estimation_complete}) is satisfied for dimensions $d$ of the same order as $n$. Hence, our method works well even when the sparsity is of the same order as $n$, which is not the case for the Lasso or the Dantzig selector.
\end{itemize}
\vspace{0.5cm}

Let us discuss the practical choice of $d$ and $K$ for complete variable selection.
From numerical studies, we advise to take $d\leq \frac{n}{2.5\left[2+\log\left(\frac{p}{n}\vee 1\right)\right]} \wedge p$ even if this quantity is slightly larger than what is ensured by the theory. The practical choice of $K$ depends on the aim of the study. If one aims at minimizing the risk, $K=1.1$ gives rather good result. A larger $K$ like $1.5$ or $2$ allows to obtain a more conservative procedure and consequently a lower FDR. We compare these values of $K$ on simulated examples in Section \ref{section_simulation}.

\subsection{Penalties based on a prior distribution}\label{section_prior}

The penalty defined in Theorem \ref{thrm_principal} only depends on the models through their cardinality. However, the methodology developed in the proof may easily extend to the case where the user has some \emph{prior} knowledge of the relevant models. Let $\pi_{\mathcal{M}}$ be a prior probability measure on the collection $\mathcal{M}$. For any non-empty model $m\in\mathcal{M}$, we define $l_m$ by 
\begin{eqnarray*}
l_m:=-\frac{\log\left(\pi_{\mathcal{M}}(m)\right)}{d_m}\ .
\end{eqnarray*}
By convention, we set $l_{\emptyset}$ to $1$. We define in the next proposition penalty functions based on the quantity $l_m$ that allow to get non-asymptotic oracle inequalities.~\\~\\
%
%
%
%
%
%
\textbf{Assumption} $(\mathbb{H}^{l}_{K,\eta})$: Given $K>1$ and $\eta>0$, the collection $\mathcal{M}$, the numbers $l_m$ and the number $\eta$ satisfy
\begin{eqnarray}\label{Hypothese_HKl}
\forall m\in\mathcal{M}, \,\,\,\,\,\,	\frac{\left[1+\sqrt{2l_m}\right]^2d_m}{n-d_m}\leq \eta<\eta(K)\ ,
\end{eqnarray}
where $\eta(K)$ is defined as in $(\mathbb{H}_{K,\eta})$.~\\

\begin{prte}\label{proposition_priori}
Let $K>1$ and let $\eta<\eta(K)$. Assume that $n\geq n_O(K)$ and that Assumption $(\mathbb{H}^l_{K,\eta})$ is fulfilled. 
If the penalty	$pen(.)$ is lower bounded as follows  
\begin{eqnarray}\label{penalite_complexe_priori}
pen(m)\geq K\frac{d_m}{n-d_m}\left(1+\sqrt{2l_m}\right)^2
\, \, \, \text{for any $m\in \mathcal{M}\setminus\{\emptyset\}$}\ ,
\end{eqnarray}
then 
\begin{eqnarray}\label{majoration_risque_priori}
\mathbb{E}\left[l(\widetilde{\theta},\theta)\right] \leq 
 L(K,\eta)\inf_{m\in\mathcal{M}}\left\{l(\theta_m,\theta)+\frac{n-d_m}{n}pen(m)\left[\sigma^2+l(\theta_m,\theta)\right]\right\}
 + \tau_n\ ,
\end{eqnarray}
where $L(K,\eta)$ and $\tau_n$ are the same as in Theorem \ref{thrm_principal}.
\end{prte}

\newpage
\textbf{Comments}:
\begin{itemize}
\item In this proposition, the penalty (\ref{penalite_complexe_priori}) as well as the risk bound (\ref{majoration_risque_priori}) depend on the prior distribution $\pi_{\mathcal{M}}$. In fact, the bound (\ref{majoration_risque_priori}) means that $\widetilde{\theta}$ achieves the trade-off between the bias and some prior weight, which is of the order $$-\log[\pi_{\mathcal{M}}(m)][\sigma^2+l(\theta_m,\theta)])/n\ .$$
This emphasizes that $\widetilde{\theta}$ favours models with a high prior probability. Similar risk bounds are obtained in the fixed design regression framework in Birg\'e and Massart \cite{birge98}.

\item If the proofs of Proposition \ref{proposition_priori} and Theorem \ref{thrm_principal} are very similar, Proposition \ref{proposition_priori} does not imply the theorem.
	
\item Roughly speaking,  Assumption ($\mathbb{H}_{K,\eta}^l$) requires that the prior probability $\pi_\mathcal{M}(m)$ is not exponentially small with respect to $n$.
\end{itemize}


\section{Minimax lower bounds and Adaptivity}\label{Section_minimax}

Throughout this section, we emphasize the dependency of the expectations $\mathbb{E}(.)$ and the probabilities $\mathbb{P}(.)$ on $\theta$ by writing $\mathbb{E}_{\theta}$ and $\mathbb{P}_{\theta}$. We have stated in Section \ref{section_oracle} that the penalized estimator  $\widetilde{\theta}$ performs almost as well as the best of the estimators $\widehat{\theta}_m$. We now want to compare the risk of $\widetilde{\theta}$ with the risk of any other possible estimator estimator $\widehat{\theta}$. There is no hope to make a pointwise comparison with an arbitrary estimator. Therefore, we classically consider the maximal risk over some suitable subsets $\Theta$ of $\mathbb{R}^p$. The \emph{minimax risk} over the set $\Theta$ is given by $\inf_{\widehat{\theta}}\sup_{\theta\in \Theta}\mathbb{E}_{\theta}[l(\widehat{\theta},\theta)]$, where the infimum is taken over all possible estimators $\widehat{\theta}$ of $\theta$. Then, the estimator $\tilde{\theta}$ is said to be \emph{approximately minimax} with respect to the set $\Theta$ if the ratio
\begin{eqnarray*}
 \frac{\sup_{\theta\in \Theta }\mathbb{E}_{\theta}\left[l\left(\tilde{\theta},\theta\right)\right]}{\inf_{\widehat{\theta}}\sup_{\theta\in \Theta}\mathbb{E}_{\theta}\left[l\left(\widehat{\theta},\theta\right)\right]}
\end{eqnarray*}
is smaller than a constant that does not depend on $\sigma^2$, $n$, or $p$. The minimax rates of estimation were extensively studied in the fixed design Gaussian regression framework and we refer for instance to \cite{birge2001} for a detailed discussion. In this section, we apply a classical methodology known as Fano's Lemma in order to derive 
 minimax rates of estimation for ordered and complete variable selection. Then, we deduce adaptive properties of the penalized estimator $\widetilde{\theta}$.

\subsection{Adaptivity with respect to ellipsoids}\label{section_adapt_ellips}

In this section, we prove that the estimator $\widetilde{\theta}$ introduced in Section \ref{Section_orderedselection} to perform ordered variable selection is adaptive to a large class of ellipsoids.

\begin{defi}\label{Definition_ellipsoides}
For any non increasing sequence $(a_i)_{1\leq i \leq {p+1}}$ such 
that $a_1=1$ and $a_{p+1}=0$ and any $R>0$, we define the ellipsoid $\mathcal{E}_{a}(R)$ by 
\begin{eqnarray*}\label{definition_formelle_ellipsoide}
 \mathcal{E}_a(R):= \left\{\theta \in \mathbb{R}^p,\sum_{i=1}^p \frac{l\left(\theta_{m_{i-1}},\theta_{m_{i}}\right)}{a_i^2}\leq R^2\right\}\ .
\end{eqnarray*}
\end{defi}
This definition is very similar to the notion of ellipsoids introduced in \cite{villers1}.
Let us explain why we call this set an ellipsoid. Assume for one moment  
that the $(X_i)_{1\leq i \leq p}$ are independent identically distributed with variance one. In this case, the 
term $l\left(\theta_{m_{i-1}},\theta_{m_{i}}\right)$ equals $\theta_i^2$ 
and the definition of $\mathcal{E}_{a}(R)$ translates in 
$$ \mathcal{E}_{a}(R) = \left\{\theta\in \mathbb{R}^p,\sum_{i=1}^p 
  \frac{\theta_i^2}{a_i^2} \leq R^2 
  \right\}\ ,$$which precisely corresponds to a \emph{classical} definition of an ellipsoid. 
 If the $(X_i)_{1\leq i\leq p}$ are not i.i.d. with unit 
 variance, it is always possible to create a sequence $X'_i$ of 
 i.i.d. standard Gaussian variables by orthonormalizing the $X_i$ using 
 Gram-Schmidt process. If we call $\theta'$ the vector in $\mathbb{R}^p$ such 
 that $X\theta=X'\theta'$, then it holds that $l\left(\theta_{m_{i-1}},\theta_{m_{i}}\right)=\theta'^2_i$. Then, we can express $\mathcal{E}_a(R)$ using 
  the coordinates of $\theta'$ as previously: 
$$ \mathcal{E}_{a}(R) = \left\{\theta\in \mathbb{R}^p,\sum_{i=1}^p \frac{\theta'^2_i}{a_i^2} \leq 
 R^2 \right\}\ .$$ 
The main advantage of this definition is that it does not directly depend on the covariance of $(X_i)_{1\leq i\leq p}$.

\begin{prte}\label{minoration_ellipsoides}
For any sequence $(a_i)_{1\leq i\leq p}$ and any positive number $R$, the minimax rate of estimation over the ellipsoid $\mathcal{E}_a(R)$ is lower bounded by
 \begin{eqnarray}\label{minoration_vrai_ellipsoide}
 \inf_{\widehat{\theta}}\sup_{\theta\in \mathcal{E}_a(R)}\mathbb{E}_{\theta}\left[l(\widehat{\theta},\theta)\right]\geq L\sup_{1\leq i\leq p}\left[a_i^2R^2 \wedge \frac{\sigma^2 i}{n}\right]\ .
 \end{eqnarray}
\end{prte}

This result is analogous to the lower bounds obtained in the fixed design regression framework (see e.g. \cite{massartflour} Th. 4.9). Hence, the estimator $\widetilde{\theta}$ built in Section \ref{Section_orderedselection} is adaptive to a large class of ellipsoids.
\begin{cor}\label{adaptation_ellipsoide}
Assume that $n$ is larger than $12$. We consider the penalized estimator $\widetilde{\theta}$ with the collection $\mathcal{M}_{\lfloor n/2\rfloor}$ and the penalty $pen(m)=K\frac{d_m}{n-d_m}$. Let $\mathcal{E}_a(R)$ be an ellipsoid whose radius $R$ satisfies $\frac{\sigma^2}{n}\leq R^2\leq \sigma^2n^\beta$ for some $\beta>0$. Then, $\widetilde{\theta}$ is approximately minimax  on $\mathcal{E}_a(R)$
\begin{eqnarray*}
\sup_{\theta\in\mathcal{E}_a(R)}l(\widetilde{\theta},\theta)\leq L(K,\beta)\inf_{\widehat{\theta}}\sup_{\theta\in \mathcal{E}_a(R)}\mathbb{E}_{\theta}\left[l(\widehat{\theta},\theta)\right]\ ,
\end{eqnarray*}
if either $n\geq 2p$ or 
$a_{\lfloor n/2\rfloor +1}^2R^2\leq \sigma^2/2$.
\end{cor}

In the fixed design framework, one may build adaptive estimators to any ellipsoid satisfying $R^2\geq \sigma^2/n$ so that the ellipsoid is not degenerate (see e.g. \cite{massartflour} Sect. 4.3.3). In our setting, when $p$ is small the estimator $\widetilde{\theta}$ is adaptive to all the ellipsoids that have a moderate radius $\sigma^2/n\leq R^2\leq n^\beta$. The technical condition $R^2\leq n^\beta$ is not really restrictive. It comes from the term $n^3l(0_p,\theta)\exp(-nL(K))$ in Theorem \ref{ordered_selection} which goes exponentially fast to $0$ with $n$. When $p$ is larger, $\widetilde{\theta}$ is adaptive to the ellipsoids that also satisfies $a_{\lfloor n/2\rfloor +1}^2R^2\leq \sigma^2/2$. 
In other words, we require that the ellipsoid is well approximated by the space $S_m{_{\lfloor n/2\rfloor}}$ of vectors $\theta$ whose support is included in $\{1,\ldots,\lfloor n/2\rfloor \}$. If this condition is not fulfilled, the estimator $\widetilde{\theta}$ is not proved to be minimax on $\mathcal{E}_a(R)$. For such situations, we believe on the one hand that the estimator $\widetilde{\theta}$ should be refined and on the other hand that our lower bounds are not sharp. Finally, the collection $\mathcal{M}_{\lfloor n/2\rfloor}$ may be replaced by any $\mathcal{M}_{\lfloor n\eta\rfloor}$ in Corollary \ref{adaptation_ellipsoide}.\\

Since the methods used for minimax lower bounds and the oracle inequalities are analogous to the ones in the Gaussian sequence framework, one may also adapt in our setting the arguments developed in \cite{massartflour} Sect. 4.3.5 to derive  minimax rates of estimation over other sets such Besov bodies. However, this is not really relevant for the regression model (\ref{modele_regression}).

\subsection{Adaptivity with respect to sparsity}\label{section_sparse}

Our aim is now to analyze the minimax risk for the complete variable selection problem. Let us fix an integer $k$ between $1$ and $p$.  We are  interested in estimating the vector $\theta$ within the class of vectors with a most $k$
non-zero components. This typically corresponds to the situation encountered in graphical modeling when estimating the neighborhoods of large sparse graphs. As the graph is assumed to be sparse, only a small number of components of $\theta$ are non-zero.

In the sequel, the set  $\Theta[k,p]$ stands for the subset of vectors $\theta \in 
\mathbb{R}^{p}$, such that at most $k$ coordinates of $\theta$ are non-zero. For any $r>0$, we denote $\Theta[k,p](r)$ the subset of $\Theta[k,p]$ such that any component of $\theta$ is smaller than $r$ in absolute value.

First, we derive a lower bound for the minimax rates of estimation when the covariates are independent. Then, we prove the estimator $\widetilde{\theta}$ defined with some collection $\mathcal{M}_p^d$ and the penalty (\ref{penalite_complete}) is adaptive to any sparse vector $\theta$. Finally, we investigate the minimax rates of estimation for correlated covariates.

\begin{prte}\label{minoration_minimax}
Assume that the covariates $X_i$ are independent and have a unit variance. For any $k\leq p$ and any radius $r>0$, 
\begin{eqnarray}
 \inf_{\widehat{\theta}}\sup_{\theta\in \Theta[k,p](r)}\mathbb{E}_{\theta}\left[l(\widehat{\theta},\theta)\right]\geq Lk\left[r^2\wedge\sigma^2\frac{1+\log\left(\frac{p}{k}\right)}{n}\right]\ .
\end{eqnarray}
\end{prte}~\\

Thanks to Theorem \ref{thrm_principal}, we derive the minimax rate of estimation over $\Theta[k,p]$.

\begin{cor}\label{adaptation_sparsite}
Consider $K>0$, $\beta>0$, and $\eta<\eta(K)$. Assume that $n\geq n_0(K)$ and that the covariates $X_i$ are independent and have a unit variance. Let $d$ be a positive integer such that $\mathcal{M}_p^d$ satisfies $(\mathbb{H}_{K,\eta})$. The penalized estimator $\widetilde{\theta}$ defined with the collection $\mathcal{M}_p^d$ and the penalty (\ref{penalite_complete}) is adaptive minimax over the sets $\Theta[k,p](n^\beta)$
\begin{eqnarray*}
\sup_{\theta\in\Theta[k,p]}\mathbb{E}_{\theta} \left[l(\widetilde{\theta},\theta)\right]\leq L(K,\beta,\eta) \inf_{\widehat{\theta}}\sup_{\theta\in \Theta[k,p](n^\beta)}\mathbb{E}_{\theta}\left[l(\widehat{\theta},\theta)\right]\ ,
\end{eqnarray*}
for any $k$ smaller than $d$.
\end{cor}
Hence, the minimax rates of estimation over $\Theta[k,p](n^\beta)$ is of order $k\frac{\log \left(\frac{ep}{k}\right)}{n}$, which is similar to the rates obtained in the fixed design regression framework.
As in previous Section, we restrict ourselves to a radius $r$ in $\Theta[k,p](r)$ smaller than $n^\beta$ because of the term $\tau_n(\var(Y),K,\eta)$ which depends on $l(0_p,\theta)$ but goes exponentially fast to 0 when $n$ goes to infinity. Let us interpret Corollary \ref{adaptation_sparsite} with regard to Condition (\ref{condition_estimation_complete}). If $p$ is of the same order as $n$, the estimator $\widetilde{\theta}$ is simultaneously minimax over all sets $\Theta[k,p](n^\beta)$ when $k$ is smaller than a constant times $n$. If $p$ is much larger than $n$, the estimator $\widetilde{\theta}$ is simultaneously minimax over all sets $\Theta[k,p](n^\beta)$ with $k$ smaller than $Ln/\log(p)$. We conjecture that the minimax rate of estimation is larger than $k\log(p/k)/n$ when $k$ becomes larger than $n/\log p$.
Let us mention that Tsybakov \cite{tsybakov_minimax} has proved general minimax lower bounds for aggregation in Gaussian random design regression. However, his result does not apply in our Gaussian design setting setting since he assumes that the density of the covariates $X_i$ is lower bounded by a constant $\mu_0$.\\

We have proved that the estimator $\widetilde{\theta}$ is adaptive to an unknown sparsity when the covariates are independent. The performance of $\widetilde{\theta}$ exhibited in Theorem \ref{thrm_principal} do not depend on the covariance matrix $\Sigma$. Hence, the minimax rates of estimation on $\Theta[k,p]$ is smaller or equal to the order $k\log(p/k)/n$ for any dependence between the covariance.
One may then wonder whether the minimax rate of estimation over $\Theta[k,p]$ is not faster when the covariates are correlated. We are unable to derive the minimax rates for a general covariance matrix $\Sigma$. This is why we restrict ourselves to particular examples of correlation structures. Let us first consider a pathological situation: Assume that $X_1,\ldots , X_k$ are independent and that $X_{k+1},\ldots ,X_p$ are all equal to $X_1$. Admittedly, the covariance matrix $\Sigma$ is henceforth non invertible. In the discussion, we mention that Theorems \ref{ordered_selection} and \ref{thrm_principal} easily extend when $\Sigma$ is non-invertible if we take into account that the estimators $\widehat{\theta}_m$ and $\widehat{m}$ are non-necessarily uniquely defined. We may derive from Lemma \ref{prte_basique} that the estimator $\widehat{\theta}_{\{1,\ldots, k\}}$ achieves the rate $k/n$ over $\theta[k,p](n^\beta)$. Conversely, the parametric rate $k/n$ is optimal. However, the estimator $\widetilde{\theta}$ defined with the collection $\mathcal{M}_p^k$ and penalty (\ref{penalite_complete}) only achieves the rate $k\log(p/k)/n$. Hence, $\widetilde{\theta}$ is not minimax over $\Theta[k,p]$ for this particular covariance matrix and the minimax rate is degenerate. This emergence of faster rates for correlation covariates also occurs for testing problems in the model (\ref{modele_regression}) as stated in \cite{villers1} Sect. 4.3.
This is why we provide sufficient conditions on $\Sigma$ so that the minimax rate of estimation is still of the same order as in the independent case. In the following proposition, $\|.\|$ refers to the canonical norm in $\mathbb{R}^p$.
\begin{prte}\label{minoration_restricted_isometry}
Let $\Psi$ denote the correlation matrix of the covariates $(X_i)_{1\leq i \leq p}$.
 Let $k$ be a positive number smaller $p/2$ and let $\delta>0$. Assume that 
\begin{eqnarray}\label{condition_restrited}
(1-\delta)^2 \|\theta\|^2\ \leq \theta^*\Psi\theta\leq (1+\delta)^2 \|\theta\|^2\ ,
\end{eqnarray}
for all $\theta\in\mathbb{R}^p$ with at most $2k$ non-zero components. Then, the minimax rate of estimation over $\Theta[k,p](r)$ is lower bounded as follows 
\begin{eqnarray*}
 \inf_{\widehat{\theta}}\sup_{\theta\in\Theta[k,p](r)}\mathbb{E}_{\theta}\left[l(\widehat{\theta},\theta)\right]\geq L(1-\delta)^2k\left[r^2\wedge\sigma^2\frac{1+\log\left(\frac{p}{k}\right)}{(1+\delta)^2n}\right]\ .
\end{eqnarray*}
\end{prte}
Assumption (\ref{condition_restrited}) corresponds to  the $\delta$-Restricted Isometry Property of order $2k$ introduced by Cand\`es and Tao \cite{candes2005}. Under such a condition, the minimax rates of estimation is the same as the one in the independent case up to a constant depending on $\delta$ and the estimator $\widetilde{\theta}$ defined in Corollary \ref{adaptation_sparsite} is still approximately minimax over such sets $\Theta[k,p]$.

However, the $\delta$-Restricted Isometry Property is quite restrictive and seems not to be necessary so that the minimax rate of estimation stays of the order $k\log(p/k)/n$. Besides, in many situations this condition is not fulfilled. Assume for instance that the random vector $X$ is a Gaussian Graphical model with respect to a given sparse graph. We expect that the correlation between two covariates is large if they are neighbors in the graph and small if they are far-off (w.r.t. the graph distance). This is why we derive lower bounds on the rate of 
estimation for correlation matrices often used to model stationary processes.

\begin{prte}\label{spatial} 
Let $X_1,\ldots, X_p$ form a stationary 
process on the one dimensional torus. More precisely, the correlation between 
$X_i$ and $X_j$ is a function of $|i-j|_p$ where $|.|_p$ refers to the 
toroidal distance defined by: 
$$|i-j|_p := (|i-j|) \wedge \left(p-|i-j|\right)\ .$$ 
$\Psi_1(\omega)$ and $\Psi_2(t)$ respectively refer to the correlation matrix of $X$ such that 
\begin{eqnarray*} 
\text{\emph{corr}}(X_i,X_j)& :=&  \exp(-\omega|i-j|_p) \text{ where }\omega>0,\\ 
\text{\emph{corr}}(X_i,X_j)& := & (1+|i-j|_p)^{-t}\text{ where }t>0. 
\end{eqnarray*} 
 Then, the minimax rates of estimation are lower bounded as follows
\begin{eqnarray*} 
 \inf_{\widehat{\theta}}\sup_{\theta\in\Theta[k,p]}\mathbb{E}_{\theta,\Psi_1(\omega)}\left[l(\widehat{\theta},\theta)\right]\geq L\frac{k\sigma^2}{n}\left[1+\log\left(\frac{\left\lfloor p \lceil \log (4k)/\omega\rceil^{-1}\right\rfloor }{k}\right)\right]\ ,\\ 
\end{eqnarray*}
if $k$ is smaller than $p /\lceil \log (4k)/\omega\rceil$ and 
\begin{eqnarray*}
\inf_{\widehat{\theta}}\sup_{\theta\in\Theta[k,p]}\mathbb{E}_{\theta,\Psi_2(t)}\left[l(\widehat{\theta},\theta)\right]\geq L\frac{k\sigma^2}{n}\left[1+ \log\left(\frac{\lfloor p \lceil (4k)^{\frac{1}{t}}-1\rceil^{-1}\rfloor }{k}\right)\right]\ ;
\end{eqnarray*} 
if $k$ is smaller than $p /\lceil(4k)^{\frac{1}{t}}-1 \rceil$.
\end{prte}

In the proof of the proposition, we justify that the correlations considered are well-defined at least when $p$ is odd. Let us mention that these correlation models are quite classical when modelling the correlation of time series (see e.g. \cite{gneiting})

If the range $\omega$ is larger than $1/p^{\gamma}$  or if the range $t$ is larger than $\gamma$ for some $\gamma<1$, the lower bounds are of order $\sigma^2\frac{k}{n}(1+\log p/k)$. As a consequence, for any of these correlation models the minimax rate of 
estimation is of the same order as the minimax rate of estimation for independent 
covariates. This means that the estimator  $\widetilde{\theta}$ defined in Proposition \ref{adaptation_sparsite} is rate-optimal for these correlations matrices. \\

In conclusion, the estimator $\widetilde{\theta}$ defined in Corollary \ref{adaptation_sparsite} may not be adaptive to the covariance matrix $\Sigma$ but rather achieves the minimax rate over all covariance matrices $\Sigma$: 
\begin{eqnarray*}
\sup_{\Sigma\geq 0}\sup_{\theta\in\Theta[k,p](n^\beta)}\mathbb{E}_{\theta} \left[l(\widetilde{\theta},\theta)\right]\leq L(K,\beta,\eta) \inf_{\widehat{\theta}}\sup_{\Sigma\geq 0}\sup_{\theta\in \Theta[k,p](n^\beta)}\mathbb{E}_{\theta}\left[l(\widehat{\theta},\theta)\right]\ . 
\end{eqnarray*}
Nevertheless, the result makes sense if one considers GGMs since the resulting covariance matrices are typically far from being independent.

\section{Numerical study}\label{section_simulation}

In this section, we carry out a small simulation study to evaluate the performance of our estimator $\widetilde{\theta}$. As pointed out earlier, an interesting feature of our criterion lies in its flexibility. 
However, we restrict ourselves here to the variable selection problem. Indeed, it allows to assess the efficiency of our procedure with having regard to the Lasso \cite{tibshirani96} and adaptive Lasso proposed by Zou \cite{zou_adaptive}. Even if these two procedures assume that the conditional variance $\sigma^2$ is known, they give good results in practice and the comparison with our method is of interest. The calculations are made with $R$ \url{www.r-project.org/}.

\subsection{Simulation scheme}

We consider the regression model (\ref{modele_regression}) with $p=20$, and $\sigma^2=1$. The number of observations $n$ equal $15$, $20$, and $30$.
We perform two simulation experiments.
\begin{enumerate}
 \item First simulation experiment: The covariance matrix $\Sigma_1$ is the identity matrix. This corresponds to the situation where the covariates are all independent. The vector $\theta_1$ has all its components to zero except the three first ones, which respectively equal $2$, $1$, and $0.5$.
\item Second simulation experiment: Let $A$ be the $p\times p$ matrix whose lines ($a_1,\ldots,a_p$) are respectively defined by
\begin{eqnarray*}
a_1 & :=& (1,-1,0,\ldots,0)/\sqrt{2}\\
a_2 & :=& (-1,1.2,0,\ldots,0)/\sqrt{1+1.2^2}\\
a_3  & :=&(1/\sqrt{2},1/\sqrt{2},1/p,\ldots,1/p)/\sqrt{1/2+(p-2)/p^2}\ ,
\end{eqnarray*}
and for $4\leq j\leq p$, $a_j$ corresponds to the $j^{\text{th}}$ canonical vector of $\mathbb{R}^p$. Then, we take the covariance matrix $\Sigma_2=A^*A$ and the vector $\theta_2^*=(40,40,0,\ldots,0)$. This choice of parameters derives from the simulation experiments of \cite{baraud08}. Observe that the two first covariates are highly correlated.		 
\end{enumerate}

For each sample we estimate $\theta$ with our procedure, the Lasso and the adaptive Lasso. For our procedure we use the collection $\mathcal{M}_{p}^3$ for $n=15$, $\mathcal{M}_{p}^4$ for $n=20$ and,  $\mathcal{M}_{p}^5$ for $n=30$. The choice of smaller collections for $n=15$ and $20$ is due to Condition (\ref{condition_estimation_complete}). We take the penalty
 (\ref{penalite_complete}) with  $K=1.1$ $1.5$, and $2$. For the Lasso and adaptive Lasso procedures, we first normalize the covariates $({\bf X}_i)$. Here, $2\sqrt{\log p}\sigma$ would be a good choice for the parameter $\lambda$ of the Lasso. However, we do not have access to $\sigma$. Hence, we use an estimation of the variance $\widehat{\var}(Y)$ which is a (possibly inaccurate) upper bound of $\sigma^2$. This is why 
 we choose the parameter $\lambda$ of the Lasso between $0.3\times2\sqrt{\log p\widehat{\var}(Y)}$ and  $2\sqrt{\log p\widehat{\var}(Y)}$ by leave-one-out cross-validation. The number $0.3$ is rather arbitrary. In practice, the performances of the Lasso do not really depend on this number as soon it is neither too small nor close to one. For the adaptive Lasso procedure, the parameters $\gamma$ and $\lambda$ are also estimated thanks to leave-one-out cross-validation: $\gamma$ can take three values $(0.5,1,2)$ and the values of $\lambda$ vary between $0.3\times2\sqrt{\log p\widehat{\var}(Y)}$ and $2\sqrt{\log(p)\widehat{\var}(Y)}$.  ~\\

We evaluate the risk ratio
\begin{eqnarray*}
 \text{ratio.Risk}:= \frac{\mathbb{E}\left[l(\widehat{\theta},\theta)\right]}{\inf_{m\in\mathcal{M}_p^5}\mathbb{E}\left[l(\widehat{\theta}_m,\theta)\right]}
\end{eqnarray*}
as well as the power and the FDR on the basis of $1000$ simulations. Here, the power corresponds to the fraction of non-zero components $\theta$ estimated as non-zero by the estimator $\widehat{\theta}$, while the FDR is the ratio of the false discoveries over the true discoveries.

\begin{eqnarray*}
 \text{Power}:= \mathbb{E}\left[\frac{\text{Card}(\{i,\ \theta_i\neq 0\ \text{and}\ \widehat{\theta}_i\neq 0 \})}{\text{Card}\left(\left\{i,\ \theta_i\neq 0\right\}\right)}\right]\hspace{0.3cm}\text{and}\hspace{0.3cm}\text{FDR}:=\mathbb{E}\left[\frac{\text{Card}(\{i,\ \theta_i= 0\ \text{and}\ \widehat{\theta}_i\neq 0 \})}{\text{Card}(\{i,\  \widehat{\theta}_i\neq 0 \})}\right]\ .
\end{eqnarray*}

\subsection{Results}
 
\begin{Table}[h] 
\caption{Our procedure with $K=1.1$, $1.5$, and $2$ and Lasso and  adaptive Lasso procedures: Estimation and $95\%$ confidence interval of Risk ratio (ratio.Risk), Power and FDR when $p=20$, $\Sigma=\Sigma_2$, $\theta=\theta_2$, and $n=15$, $20$, and $30$.\label{tab1}}  
\begin{center}
\begin{tabular}{|l|ccc|ccc|}
\hline  & & $n=15$ & & & $n=20$ &     \\ \hline
Estimator &  ratio.Risk & Power & FDR & ratio.Risk & Power & FDR  \\  
\hline $K=1.1$ & $4.8\pm 0.4$ & $0.67\pm 0.02$ & $0.23\pm 0.02$ &  $4.8\pm 0.3$ & $0.77\pm 0.01$& $0.28\pm 0.02$\\
\hline $K=1.5$  &  $5.7\pm 0.4$ & $0.62\pm 0.02 $ & $0.20\pm 0.01$  & $5.3\pm 0.4$ & $0.74\pm 0.02$ & $0.25\pm 0.01$ \\ \hline 
 $K=2$  &  $7.3\pm 0.5$ & $0.54\pm 0.02 $ & $0.17\pm 0.01$  & $6.6\pm 0.5$ & $0.68\pm 0.02$ & $0.21\pm 0.01$ \\ \hline 
 Lasso &  $5.8\pm 0.2$ & $0.64\pm 0.01$ & $0.29\pm 0.02$  & $6.0\pm 0.2$ & $0.74\pm0.01 $ & $0.23\pm 0.01$ \\ \hline 
 A. Lasso &  $4.8\pm 0.3$ & $0.64\pm 0.02$ & $0.30\pm 0.02$  & $4.7\pm 0.4 $ & $0.75\pm 0.02 $ & $0.30\pm 0.01$ \\ \hline 
\end{tabular}
\end{center}
\begin{center}
\begin{tabular}{|l|ccc|}
\hline  &  & $n=30$ &    \\ \hline
Estimator &  ratio.Risk & Power & FDR  \\  
\hline $K=1.1$ & $4.2\pm 0.3$ & $0.87\pm 0.01$ & $0.23\pm 0.02$ \\
\hline $K=1.5$ &  $4.1\pm0.2 $ & $0.84\pm 0.01$ & $0.19\pm 0.01$  \\ \hline 
 $K=2$  &  $4.3\pm 0.2$ & $0.81\pm 0.01$ & $0.14\pm 0.01$   \\ \hline 
 Lasso &  $6.6\pm 0.2$ & $0.83\pm 0.01$ & $0.18\pm 0.01$ \\ \hline 
 A. Lasso &  $4.3\pm 0.5$ & $0.86\pm 0.02$ & $0.26\pm 0.01$  \\ \hline 
\end{tabular}
\end{center}
\end{Table}

\begin{Table}[h] 
\caption{Our procedure with $K=1.1$, $1.5$, and $2$ and Lasso and  adaptive Lasso procedures: Estimation and $95\%$ confidence interval of Risk ratio (ratio.Risk), Power and FDR when $p=20$, $\Sigma=\Sigma_1$, $\theta=\theta_1$, and $n=15$, $20$, and $30$.\label{tab2}}

\begin{center}
\begin{tabular}{|l|ccc|ccc|}
\hline  & & $n=15$ & & & $n=20$ &     \\ \hline
Estimator &  ratio.Risk & Power & FDR & ratio.Risk & Power & FDR  \\
  \hline $K=1.1$  &  $5.3\pm 0.4$ & $0.77\pm 0.03$ & $ 0.41\pm 0.02$ & $6.4\pm 0.5$ & $0.87\pm 0.02$ & $0.39\pm 0.02$\\
\hline $K=1.5$  &  $5.3\pm 0.4$ & $0.76\pm 0.03$ & $0.41\pm 0.02$  & $5.9\pm 0.5$ & $0.87\pm 0.02$ & $0.36\pm 0.02$ \\ \hline 
 $K=2$  &  $5.5\pm 0.5$ & $0.75 \pm 0.03$& $0.40\pm 0.02$  &   $5.5 \pm 0.5$&  $0.86\pm 0.02$& $0.33\pm 0.02$ \\ \hline
 Lasso & $13.5\pm 0.3$  & $0.02\pm 0.01$ & $0.99\pm 0.01$ & $16.7\pm 0.3$ & $0.02\pm 0.01$ & $0.98\pm 0.01$ \\ \hline
 A. Lasso & $15.0 \pm 1.2$ & $0.02\pm 0.01$ & $0.90 \pm 0.02$ & $20.5\pm 1.8$ & $0.04\pm 0.01$ & $0.89\pm 0.02$ \\ \hline 
\end{tabular}
\end{center}
\begin{center}
\begin{tabular}{|l|ccc|}
\hline  &  & $n=30$ &    \\ \hline
Estimator &  ratio.Risk & Power & FDR  \\  
\hline $K=1.1$  &    $4.5\pm 0.3$ & $0.96\pm 0.02$ & $0.24\pm 0.02$    \\
\hline $K=1.5$  & $3.9\pm 0.3$ &$0.95\pm 0.01$ & $0.19\pm 0.02$ \\ \hline 
 $K=2$  &   $3.5\pm 0.3$ & $0.94\pm 0.01$& $0.16\pm 0.02$\\ \hline
 Lasso & $22.0\pm 0.3$ &$0.02\pm 0.01$ &$0.99\pm 0.01$\\ \hline
 A. Lasso &  $31.8\pm 3.0$& $0.04\pm 0.01$& $0.88\pm 0.02$\\ \hline 
\end{tabular}
\end{center}

\end{Table} 

The results of the first simulation experiment are given in Table \ref{tab1}. We observe that the five estimators perform more or less similarly as expected by the theory.
The results of the second simulation study are reported in Table \ref{tab2}. Clearly, the Lasso and adaptive Lasso procedures are not consistent in this situation since the power is close to $0$ and the FDR is close to one. Consequently, the risk ratio is quite large and the adaptive Lasso even seems unstable. In contrast, our method exhibits a large power and a reasonable FDR. \\

In the two studies, choosing a larger $K$ reduces the power of the estimator but also decreases the FDR. It seems that the choice $K=1.1$ yields a good risk ratio, whereas $K=2$ gives a better control of the FDR. Contrary to the parameter $\lambda$ for the lasso, we do not need an \emph{ad-hoc} method such as cross-validation to calibrate $K$.
The second example is certainly quite pathological but it  illustrates that our estimator $\widetilde{\theta}$ performs well even when the Lasso does not provide an accurate estimation. The good behavior of our method illustrates the strength of Theorem \ref{thrm_principal} that does not depend on the correlation of the explanatory variables.


\section{Discussion and concluding remarks}\label{section_discussion}

Until now, we have assumed that the covariance matrix $\Sigma$ of the covariates is non-singular. If $\Sigma$ is singular, the estimators $\widehat{\theta}_m$ and the model $\widehat{m}$ are not necessarily uniquely defined. However, upon defining $\widehat{\theta}_m$ as \emph{one} of the minimizers of $\gamma_n(\theta')$ over $S_m$, one may readily extend the oracle inequalities stated in Theorem \ref{ordered_selection} and \ref{thrm_principal}.\\

Let us recall the main features of our method. We have defined a model selection criterion that satisfies  oracle inequalities regardless  of the correlation between the covariates and regardless of the collection of models. Hence, the estimator $\widetilde{\theta}$ achieves nice adaptive properties for ordered variable selection or for complete variable selection. Besides, one can easily combine this method with prior knowledge on the model by choosing a proper collection $\mathcal{M}$ or by modulating the penalty $pen(.)$. 
 Moreover, we may easily calibrate the penalty even when $\sigma^2$ is unknown, whereas the Lasso-type procedures require a cross-validation strategy to choose the parameter $\lambda$. The compensation for these nice properties is a  computational cost that depends linearly on the size of $\mathcal{M}$. Hence, the complete variable selection problem is NP-hard. This makes it intractable when $p$ becomes too large (i.e. more than 50). In contrast, our criterion applies for arbitrary $p$ when considering ordered variable selection since the size of $\mathcal{M}$ is linear with $n$. In situations where one has a good prior knowledge on the true model, the collection $\mathcal{M}$ is then not too large and our criterion is also fastly calculable even for large $p$.\\

For complete variable selection, Lasso-type procedures are computationally feasible even when $p$ is large and achieve oracle inequalities under assumptions on the covariance structure. However, there are both theoretical and practical problems with these estimators. On the one hand, they are known to perform poorly for some covariance structures. On the other hand, there is some room for improvement in the practical calibration of the lasso, especially when $\sigma^2$ is unknown. 
In a future work, we would like to combine the strength of our method with these computationally fast algorithms. The 
problem at hand is to design a fast data-driven method that picks a subcollection $\widehat{\mathcal{M}}$ of reasonable size. Afterwards, one applies our procedure to $\widehat{\mathcal{M}}$ instead of $\mathcal{M}$.  A direction that needs further investigation is taking for $\widehat{M}$ all the subsets of the regularization path given by the lasso.

\section{Proofs}\label{section_proofs}

\subsection{Some notations and probabilistic tools}\label{section_notation_preuve}

First, let us define the random variable $\epsilon_m$ by 
\begin{eqnarray}\label{definition_epsilonm}
Y & = & X\theta_m + \epsilon_m +\epsilon\ a.s.\ .
\end{eqnarray}
By definition of $\theta_m$, $\epsilon_m$ follows a normal distribution and is independent of $\epsilon$ and of $X_m$. Hence, the variance of $\epsilon_m$ equals $l(\theta_m,\theta)$. The vectors $\boldsymbol{\epsilon}$ and $\boldsymbol{\epsilon}_m$ refer to the $n$ samples of $\epsilon$ and $\epsilon_m$. For any model $m$ and any vector $Z$ of size $n$, $\Pi^\perp_{m}Z$ stands for $Z-\Pi_m Z$. For any subset $m$ of $\{1,\ldots,p\}$, $\Sigma_m$ denotes the covariance matrix of the vector $X^*_m$. Moreover, we define the row vector $Z_{m}:=X_{m}\sqrt{\Sigma_{m}^{-1}}$ 
in order to deal with  standard Gaussian vectors. Similarly to the matrix ${\bf X}_{m}$, the  $n\times d_{m}$ matrix ${\bf Z}_{m}$ stands for the $n$ observations of $Z_{m}$. The notation $\langle.,.\rangle_n$  refers to the empirical inner product associated with the norm $\|.\|_n$. Lastly, $\varphi_{\text{max}}(A)$ denotes the largest eigenvalue (in absolute value) of a symmetric square matrix $A$.

We shall extensively
use the explicit expression of $\widehat{\theta}_m$:
\begin{eqnarray}
{\bf X}\widehat{\theta}_m & = & {\bf X}_m({\bf X}^*_m{\bf X}_m)^{-1}{\bf X}^*_m{\bf Y}\ . \label{expression_thetam}
\end{eqnarray}

Let us state a first lemma that gives the expressions of  $\gamma_n(\widehat{\theta}_m)$,  $\gamma(\widehat{\theta}_m)$, and the loss $l(\widehat{\theta}_m,\theta_m)$.
\begin{lemma}\label{lemma_expression_perte_perte_empirique}
For any model $m$ of size smaller than $n$,  
\begin{eqnarray}
\gamma_n\left(\widehat{\theta}_m\right) & = & \|\Pi_m^\perp \left(\boldsymbol{\epsilon}+\boldsymbol{\epsilon}_m\right)\|_n^2\label{perte_empirique_estimateur}\ ,\\
\gamma\left(\widehat{\theta}_m\right)& = &\sigma^2+l(\theta_m,\theta)+l(\widehat{\theta}_m,\theta_m)\ ,\label{perte_estimateur}\\
l(\widehat{\theta}_m,\theta_m) & = & (\boldsymbol{\epsilon}+\boldsymbol{\epsilon}_{m})^*{\bf Z}_{m}({\bf Z}^*_{m}{\bf Z}_{m})^{-2}{\bf Z}^*_{m}(\boldsymbol{\epsilon}+\boldsymbol{\epsilon}_{m})\ .\label{perte_biais}
\end{eqnarray}
\end{lemma}
The proof is postponed to the Appendix.\\

We now introduce the main probabilistic tools used throughout the  proofs. First, we need to bound the deviations of $\chi^2$ random variables.
\begin{lemma}\label{concentration_chi2}
 For any integer $d>0$ and any positive number $x$, 
 \begin{eqnarray}
\mathbb{P}\left(\chi^2(d) \leq d - 2\sqrt{dx} \right)& \leq& \exp(-x)
\nonumber\ ,\\
\mathbb{P}\left(\chi^2(d) \geq d + 2\sqrt{dx} + 2x \right) &\leq &\exp(-x)\ .
\nonumber
\end{eqnarray}
\end{lemma}
These bounds are classical and are shown by applying Laplace method. We refer to
Lemma 1 in \cite{Laurent98} for more details. Moreover, we state a refined bound for the lower deviations of a $\chi^2$ distribution.

\begin{lemma}\label{concentration_chi2_fine}
 For any integer $d>0$ and any positive number $x$, 
\begin{eqnarray}
\mathbb{P}\left[\chi^2(d) \leq d\left[\left(1 - \delta_d-\sqrt{\frac{2x}{d}}\right)\vee 0 \right]^2\right]& \leq& \exp(-x)
\nonumber\ ,
\end{eqnarray}
\begin{eqnarray}\label{definition_delta}
\text{where }\delta_d:= \sqrt{\frac{\pi }{2d}}+\exp(-d/16)\ .\end{eqnarray}
\end{lemma}
The proof is postponed the Appendix. Finally, we shall bound the largest eigenvalue of standard Wishart matrices and standard inverse Wishart matrices. The following
deviation inequality is taken from Theorem 2.13 in \cite{Davidson2001}.
\begin{lemma}\label{concentration_vp_wishart}
Let $Z^*Z$ be a standard Wishart matrix of parameters $(n,d)$ with $n>d$. For any positive number $x$, 
\begin{eqnarray*}
\mathbb{P}\left\{\varphi_{\text{\emph{max}}}\left[(Z^*Z)^{-1}\right] \geq \left[n\left(1-\sqrt{\frac{d}{n}}-x\right)^2\right]^{-1} \right\}& \leq &\exp(-nx^2/2)\ ,
\end{eqnarray*}
and 
\begin{eqnarray*}
\mathbb{P}\left[\varphi_{\text{\emph{max}}}\left( Z^*Z\right) \leq n\left(1+\sqrt{\frac{d}{n}}+x\right)^2 \right]& \leq &\exp(-nx^2/2)\ .	
\end{eqnarray*}
\end{lemma}

\subsection{Proof of Theorem \ref{ordered_selection}}

\begin{proof}[Proof of Theorem \ref{ordered_selection}]
For the sake of simplicity we divide the main steps of the proof in several lemmas. First, let us fix a model $m$ in the collection  
  $\mathcal{M}$. By definition of $\widehat{m}$, we know that
\begin{eqnarray*}
\gamma_n(\widetilde{\theta})\left[1+pen(\widehat{m})\right] \leq \gamma_n(\theta_m)\left[1+ pen(m)\right] \ .
\end{eqnarray*}
Subtracting $\gamma(\theta)$ to both sides of this inequality yields 
\begin{eqnarray}
l(\widetilde{\theta},\theta) \leq l(\theta_m,\theta) + \gamma_n(\theta_m)pen(m) +
\overline{\gamma}_n(\theta_m) - \gamma_n(\widetilde{\theta})pen(\widehat{m}) - \overline{\gamma}_n(\widetilde{\theta})\ , 
\label{majoration_principal}\end{eqnarray}
where $\overline{\gamma}_n(.):=\gamma_n(.)-\gamma(.)$.
The proof is based on the
concentration of the term $-\overline{\gamma}_n(\widetilde{\theta})$. More precisely, we shall prove that
with overwhelming probability this quantity is of the same order as the penalty term $\gamma_n(\widetilde{\theta})pen(\widehat{m})$. 

 Let $\kappa_1$ and $\kappa_2$ be two positive numbers smaller than one that we shall fix later. For any model $m'\in\mathcal{M}$, we introduce the random variables $A_{m'}$ and $B_{m'}$ as
\begin{eqnarray}
A_{{m'}} & := &  \kappa_1+ 1 -
  \frac{\|\Pi^{\perp}_{m'}\boldsymbol{\epsilon}_{{m'}}\|^2_n}{l(\theta_{m'},\theta)}
  + \kappa_2 n \varphi_{\text{max}}\left[({\bf Z}^*_{{m'}}{\bf Z}_{{m'}})^{-1}\right]
   \frac{\|\Pi_m (\boldsymbol{\epsilon}
  +\boldsymbol{\epsilon}_{{m'}})\|^2_n}{l(\theta_{m'},\theta)+\sigma^2}  \nonumber \\ \label{definition_A}
  & - & K\frac{d_{{m'}}}{n-d_{{m'}}}   \frac{\|\Pi_{{m'}^\perp} (\boldsymbol{\epsilon}
  +\boldsymbol{\epsilon}_{{m'}})\|^2_n}{l(\theta_{m'},\theta)+\sigma^2}\ ,\\
B_{{m'}}& := & \kappa_1^{-1}\frac{\langle\Pi^{\perp}_{m'}\boldsymbol{\epsilon},\Pi^{\perp}_{m'}\boldsymbol{\epsilon}_{{m'}}
  \rangle^2_n}{\sigma^2l(\theta_{m'},\theta)} + \frac{\|\Pi_{{m'}}
  \boldsymbol{\epsilon}\|^2_n}{\sigma^2} +\kappa_2 n
 \varphi_{\text{max}}\left[ \left( {\bf Z}^*_{{m'}}{\bf Z}_{{m'}}\right)^{-1}\right]\frac{\|\Pi_{{m'}} (\boldsymbol{\epsilon}
   +\boldsymbol{\epsilon}_{{m'}})\|^2_n}{l(\theta_{m'},\theta)+\sigma^2}\nonumber\\  &- &K\frac{d_{{m'}}}{n-d_{{m'}}}\frac{\|\Pi^{\perp}_{m'} (\boldsymbol{\epsilon}
   +\boldsymbol{\epsilon}_{{m'}})\|^2_n}{l(\theta_{m'},\theta)+\sigma^2} \ . \label{definition_B}
\end{eqnarray}
We recall that the notations $\epsilon_m$, $Z_m$, $\langle.,.\rangle_n$, and $\varphi_{\text{max}}( .)$ are defined in Section \ref{section_notation_preuve}. We may upper bound the expression $-\overline{\gamma}_n(\widetilde{\theta})-  \gamma_n(\widetilde{\theta})pen(\widehat{m})$ with respect to $A_{\widehat{m}}$ and $B_{\widehat{m}}$ as follows. 
\begin{lemma}\label{lemme_decomposition}
Almost surely, it holds that
\begin{eqnarray}
 -\overline{\gamma}_n(\widetilde{\theta}) -  \gamma_n(\widetilde{\theta})pen(\widehat{m}) -\sigma^2 +
\|\boldsymbol{\epsilon}\|^2_n \leq l(\widetilde{\theta},\theta)  \left[A_{\widehat{m}}\vee (1-\kappa_2)\right]+ \sigma^2 B_{\widehat{m}}\ .\label{decomposition_generale2}
\end{eqnarray}
\end{lemma}
Let us set the constants 
\begin{eqnarray}
\kappa_1 : = \frac{1}{4} \hspace{0.5cm}\text{and} \hspace{0.5cm}\kappa_2 : = \frac{(K-1)(1-\sqrt{\eta})^2}{16}\wedge 1 \label{definition_kappa0}\ .
\end{eqnarray} 
 We do not claim that this choice is optimal, but we are not really concerned about the constants for this result. The core of this proof consists in showing that with overwhelming probability
the variable $A_{\widehat{m}}$ is smaller than $1$ and 
$B_{\widehat{m}}$ is smaller than a constant over $n$.

\begin{lemma}\label{concentration_thrm1}
The event $\Omega_1$ defined as
\begin{eqnarray*}
\Omega_1 :=\left\{A_{\widehat{m}}\leq  \frac{7}{8}\right\}\bigcap\left\{ \kappa_2 n\varphi_{\text{max}}\left[  ({\bf Z}^*_{\widehat{m}}{\bf Z}_{\widehat{m}})^{-1}\right]  \leq   \frac{K-1}{4}\right\}
\end{eqnarray*}
satisfies  
$\mathbb{P}(\Omega_1^c)\leq L\text{\emph{Card}}(\mathcal{M})\exp\left[-
nL'(K,\eta)\right]$, where $L'(K,\eta)$ is positive.
\end{lemma}

\begin{lemma}\label{majoration_esperance_conditionelle}
There exists an event $\Omega_2$ of probability larger than $1-\exp\left(-nL\right)$ with $L>0$ such that
\begin{eqnarray*}
\mathbb{E}\left[B_{\widehat{m}}\mathbf{1}_{\Omega_1\cap \Omega_2}\right] \leq \frac{L(K,\eta,\alpha,\beta)}{n}\ .
\end{eqnarray*}
\end{lemma}
Gathering the upper bound (\ref{majoration_principal}) and Lemma \ref{lemme_decomposition},
\ref{concentration_thrm1}, and \ref{majoration_esperance_conditionelle},
we conclude that
\begin{eqnarray*}
  \mathbb{E}\left[l(\widetilde{\theta},\theta)\mathbf{1}_{\Omega_1\cap \Omega_2}\left(\kappa_2\wedge\frac{1}{8}\right)
  \right] & \leq & l(\theta_m,\theta)+\mathbb{E}\left[\gamma_n(\theta_m)pen(m)\right] \\
   & +&  \sigma^2\frac{L(K,\eta,\alpha,\beta)}{n}+\mathbb{E}\left[\mathbf{1}_{\Omega_1\cap \Omega_2}\left(\overline{\gamma}_n\left(\theta_m\right)+\sigma^2 - \|\boldsymbol{\epsilon}\|_n^2\right)\right]\ .
\end{eqnarray*}
As the expectation of the random variable $\overline{\gamma}_n\left(\theta_m\right)+\sigma^2 - \|\boldsymbol{\epsilon}\|_n^2$ is zero, it holds that
\begin{eqnarray*}
\mathbb{E}\left[\mathbf{1}_{\Omega_1\cap \Omega_2}\left(\overline{\gamma}_n\left(\theta_m\right)+\sigma^2 - \|\boldsymbol{\epsilon}\|_n^2\right)\right]  =   \mathbb{E}\left[\mathbf{1}_{\Omega^c_1\cup \Omega^c_2}\left(\overline{\gamma}_n\left(\theta_m\right)+\sigma^2 - \|\boldsymbol{\epsilon}\|_n^2\right)\right]\hspace{3cm} \\
 \hspace{3cm}\leq  \sqrt{\mathbb{P}(\Omega_1^c)+ \mathbb{P}(\Omega_2^c)}\left[\sqrt{\mathbb{E}\left[\|\boldsymbol{\epsilon}_m\|_n^2-l(\theta_m,\theta)\right]^2}+2\sqrt{\mathbb{E}\left[\langle\boldsymbol{\epsilon},\boldsymbol{\epsilon}_m\rangle_n^2\right]}\right]\\
 \hspace{0cm}\leq  \sqrt{\mathbb{P}(\Omega_1^c)+ \mathbb{P}(\Omega_2^c)}\sqrt{\frac{2}{n}} \left[l(\theta_m,\theta) + \sigma \sqrt{2l(\theta_m,\theta)}\right]\ .\hspace{2cm}
\end{eqnarray*}
The probabilities $\mathbb{P}(\Omega_1^c)$ and $\mathbb{P}(\Omega_2^c)$ converge to $0$ at an exponential rate with respect to $n$. Hence, by taking the infimum over all the models $m\in\mathcal{M}$, we obtain
\begin{eqnarray}\label{majoration_risque_C}
 \mathbb{E}\left[l(\widetilde{\theta},\theta)\mathbf{1}_{\Omega_1\cap \Omega_2}
  \right] &\leq&
  L(K,\eta)\inf_{m\in\mathcal{M}}\left[l(\theta_m,\theta)+\left(\sigma^2+l(\theta_m,\theta)\right)pen(m)\right] + L_2(K,\eta,\alpha,\beta)\frac{\sigma^2}{n} + \nonumber\\&+ &L_3(K,\eta)\sqrt{\frac{\text{Card}(\mathcal{M})}{n}}\left[\sigma^2+l(0_p,\theta)\right]\exp\left[-nL_4(K,\eta)\right] \ ,
\end{eqnarray}
with $L_4(K,\eta)>0$.
In order to conclude, we need to control the loss of the estimator $\widetilde{\theta}$ on
the event of small probability $\Omega_1^c\cup \Omega_2^c$. Thanks to the following lemma, we may upper bound the $r$-th risk of the estimators $\widehat{\theta}_m$.
\begin{prte}\label{controle_risque_lp}
For any model $m$ and any integer $r\geq 2$ such that $n-d_m-2r+1>0$,  
\begin{eqnarray*}
\mathbb{E}\left[l(\widehat{\theta}_m,\theta_m)^{r}\right]^{\frac{1}{r}} \leq Lr d_m n\left[\sigma^2+l(\theta_m,\theta)\right]\ .
\end{eqnarray*} 
\end{prte}
The proof is postponed to Section \ref{section_controle_lp}. We derive from this bound a strong control on $\mathbb{E}\left[l(\widetilde{\theta},\theta)\mathbf{1}_{\Omega_1^c\cup \Omega_2^c}\right]$.
\begin{lemma}\label{controlerisquesimple}
\begin{eqnarray}\label{majoration_risque_complementaire}
\mathbb{E}\left[l(\widetilde{\theta},\theta)\mathbf{1}_{\Omega_1^c\cup \Omega_2^c}\right] \leq
L(K,\eta)n^{2}\text{Card}(\mathcal{M})\var(Y)\exp\left[-nL'(K,\eta)\right]\ ,
\end{eqnarray}
where $L'(K,\eta)$ is positive.
\end{lemma}
By Assumptions $(\mathbb{H}_{Pol})$ and $ (\mathbb{H}_{\eta})$, the cardinality of the collection of $\mathcal{M}$ is smaller than $\alpha n^{1+\beta}$.
We gather the upper bounds (\ref{majoration_risque_C}) and
(\ref{majoration_risque_complementaire}) and so we conclude.
\end{proof}\vspace{0.5cm}

\begin{proof}[Proof of Lemma \ref{lemme_decomposition}]
Thanks to Lemma \ref{lemma_expression_perte_perte_empirique}, we  decompose $\overline{\gamma}_n(\widetilde{\theta})$ as
\begin{eqnarray*}
 \overline{\gamma}_n(\widetilde{\theta})= \|\Pi_{\widehat{m}}^{\perp}(\boldsymbol{\epsilon}+\boldsymbol{\epsilon}_{\widehat{m}})\|_n^2-\sigma^2-l(\theta_{\widehat{m}},\theta)-(1-\kappa_2)l(\widetilde{\theta},\theta_{\widehat{m}})-\kappa_2(\boldsymbol{\epsilon}+\boldsymbol{\epsilon}_{\widehat{m}})^*{\bf Z}_{\widehat{m}}({\bf Z}^*_{\widehat{m}}{\bf Z}_{\widehat{m}})^{-2}{\bf Z	}^*_{\widehat{m}}(\boldsymbol{\epsilon}+\boldsymbol{\epsilon}_{\widehat{m}})\ .
\end{eqnarray*}
Since $2ab\leq \kappa_1a^2+ \kappa_1^{-1}b^2$ for any $\kappa_1>0$, it holds that
\begin{eqnarray*}
 -\|\Pi_{\widehat{m}}^{\perp}(\boldsymbol{\epsilon}+\boldsymbol{\epsilon}_{\widehat{m}})\|_n^2+\|\boldsymbol{\epsilon}\|_n^2 & = & \|\Pi_{\widehat{m}}\boldsymbol{\epsilon}\|_n^2 -\|\Pi_{\widehat{m}}^{\perp}\boldsymbol{\epsilon}_{\widehat{m}}\|_n^2-2\langle\Pi_{\widehat{m}}^\perp\boldsymbol{\epsilon} , \Pi_{\widehat{m}}^{\perp}
\boldsymbol{\epsilon}_{\widehat{m}}\rangle_n\\
& \leq & \sigma^2\left[\kappa_1^{-1} \frac{\langle\Pi_{\widehat{m}}^{\perp}\boldsymbol{\epsilon},\Pi^\perp_{\widehat{m}}\boldsymbol{\epsilon}_{\widehat{m}}\rangle^2_n}{\sigma^2l(\theta_{\widehat{m}},\theta)}+\frac{\|\Pi_{\widehat{m}}\boldsymbol{\epsilon}\|_n^2}{\sigma^2}\right]+ l(\theta_{\widehat{m}},\theta)\left[-\frac{\|\Pi_{\widehat{m}}^{\perp}\boldsymbol{\epsilon}_{\widehat{m}}\|_n^2}{l(\theta_{\widehat{m}},\theta)}+\kappa_1\right]\ .
\end{eqnarray*}
Besides, we upper bound  Expression (\ref{perte_biais}) of $l(\widetilde{\theta},\theta_{\widehat{m}})$ using the largest eigenvalue of $\left({\bf Z}^*_{\widehat{m}}{\bf Z}_{\widehat{m}}\right)^{-1}$.
\begin{eqnarray}
 (\boldsymbol{\epsilon}+\boldsymbol{\epsilon}_{\widehat{m}})^*{\bf Z}_{\widehat{m}}({\bf Z}^*_{\widehat{m}}{\bf Z}_{\widehat{m}})^{-2}{\bf Z}^*_{\widehat{m}}(\boldsymbol{\epsilon}+\boldsymbol{\epsilon}_{\widehat{m}})& \leq & \varphi_{\text{max}}\left[ ({\bf Z}^*_{\widehat{m}}{\bf Z}_{\widehat{m}})^{-1}\right] (\boldsymbol{\epsilon}+\boldsymbol{\epsilon}_{\widehat{m}})^*{\bf Z}_{\widehat{m}}({\bf Z}^*_{\widehat{m}}{\bf Z}_{\widehat{m}})^{-1}{\bf
  Z}^*_{\widehat{m}}(\boldsymbol{\epsilon}+\boldsymbol{\epsilon}_{\widehat{m}})\nonumber\\
& \leq & \left[\sigma^2 + l(\theta_{\widehat{m}},\theta)\right]n\varphi_{\text{max}}\left[ \left({\bf Z}^*_{\widehat{m}}{\bf Z}_{\widehat{m}}\right)^{-1}\right] \frac{\|\Pi_{\widehat{m}}(\boldsymbol{\epsilon}+ \boldsymbol{\epsilon}_{\widehat{m}})\|_n^2}{\sigma^2 + l(\theta_{\widehat{m}},\theta)}\ . \label{decomposition3}
\end{eqnarray}
Thanks to Assumption (\ref{hypothese1_penalite}), we upper bound the penalty terms as follows:
\begin{eqnarray*}
 -\gamma_n(\widetilde{\theta})pen(\widehat{m})\leq -\left[\sigma^2+l(\theta_{\widehat{m}},\theta)\right]\frac{\|\Pi_{\widehat{m}}^\perp(\boldsymbol{\epsilon}+ \boldsymbol{\epsilon}_{\widehat{m}})\|_n^2}{\sigma^2+l(\theta_{\widehat{m}},\theta)} K\frac{d_{\widehat{m}}}{n-d_{\widehat{m}}} \ .
\end{eqnarray*}
By gathering the four last identities, we get
\begin{eqnarray*}
 -\overline{\gamma}_n(\widetilde{\theta}) -  \gamma_n(\widetilde{\theta})pen(\widehat{m}) -\sigma^2 +
\|\boldsymbol{\epsilon}\|^2_n \leq l(\widetilde{\theta},\theta)  \left[A_{\widehat{m}}\vee (1-\kappa_2)\right]+ \sigma^2 B_{\widehat{m}}\ ,
\end{eqnarray*}
since $l(\widetilde{\theta},\theta)$ decomposes into the sum $l(\widetilde{\theta},\theta_{\widehat{m}}) + l(\theta_{\widehat{m}},\theta)$.
\end{proof}

\begin{proof}[Proof of Lemma \ref{concentration_thrm1}]
We recall that for any model $m\in \mathcal{M}$,
\begin{eqnarray*}
A_{m}&  := & \frac{5}{4} -
  \frac{\|\Pi_{m}^{\perp}\boldsymbol{\epsilon}_{{m}}\|^2_n}{l(\theta_{m},\theta)}
  + \kappa_2 n \varphi_{\text{max}}\left[ ({\bf Z}^*_{{m}}{\bf Z}_{{m}})^{-1}\right]
   \frac{\|\Pi_m (\boldsymbol{\epsilon}
  +\boldsymbol{\epsilon}_{{m}})\|^2_n}{l(\theta_{m},\theta)+\sigma^2}\\
 & - & K\frac{d_{{m}}}{n-d_{{m}}}   \frac{\|\Pi_{m}^\perp (\boldsymbol{\epsilon}
  + \boldsymbol{\epsilon}_{{m}})\|^2_n}{l(\theta_{m},\theta)+\sigma^2}\ .
\end{eqnarray*}
In order to control the variable $A_{\widehat{m}}$, we shall
simultaneously bound the deviations of the four random variables involved in 
any variable $A_{m}$. 

Since ${\bf X}_{m}$ is independent of $\boldsymbol{\epsilon}_{m}/\sqrt{l(\theta_{m},\theta)}$ and since $\boldsymbol{\epsilon}_{m}/\sqrt{l(\theta_{m},\theta)}$ is a standard Gaussian vector of size $n$, the random variable $n\|\Pi_{m}^\perp\boldsymbol{\epsilon}_{{m}}\|^2_n/l(\theta_{m},\theta)$ follows a $\chi^2$ distribution with $n-d_{m}$ degrees of freedom conditionally on ${\bf X}_{m}$. As this distribution does not depend on ${\bf X}_{m}$,  $n\|\Pi_{m}^\perp\boldsymbol{\epsilon}_{m}\|^2_n/l(\theta_{m},\theta)$ follows a $\chi^2$ distribution with $n-d_{m}$ degrees of freedom. Similarly, the random variables $ n\|\Pi_m (\boldsymbol{\epsilon}
  +\boldsymbol{\epsilon}_{m})\|^2_n/[l(\theta_{m},\theta)+\sigma^2]$ and $n\|\Pi_{m}^\perp (\boldsymbol{\epsilon}
  +\boldsymbol{\epsilon}_{m})\|^2_n/[l(\theta_{m},\theta)+\sigma^2]$ follow $\chi^2$ distributions with respectively $d_{m}$ and $n-d_{m}$ degrees of freedom. Besides, the matrix $({\bf {\bf Z}}^*_{m}{\bf {\bf Z}}_{m})$ follows a standard Wishart distribution with parameters $(n,d_{m})$.

Let $x$ be a positive number we shall
fix later. By Lemma \ref{concentration_chi2} and \ref{concentration_vp_wishart}, there exists an event $\Omega'_1$
of large probability
$$P(\Omega_1^{'c})\leq 4\exp(-nx)\text{Card}(\mathcal{M})
 \ ,$$
such that for conditionally on $\Omega'_1$,
\begin{eqnarray}
\frac{\|\Pi_{m}^{\perp}\boldsymbol{\epsilon}_{m}\|^2_n}{l(\theta_{m},\theta)} &\geq & \frac{n-d_{m}}{n} -
2\sqrt{\frac{(n-d_{m})x}{n}}\ , \label{majoration_1}\\ 
\frac{\|\Pi_{m}(\boldsymbol{\epsilon}+\boldsymbol{\epsilon}_{m})\|_n^2}{\sigma^2+ l(\theta_{m},\theta)} & \leq & \frac{d_{m}}{n}
+2\sqrt{\frac{d_{m}x}{n}}+2x\ ,\label{majoration_2}\\
\frac{\|\Pi_{m}^\perp(\boldsymbol{\epsilon}+\boldsymbol{\epsilon}_{m})\|_n^2}{\sigma^2+ l(\theta_{m},\theta)} & \geq & \frac{n-d_{m}}{n}
-2\sqrt{\frac{(n-d_{m})x}{n}}\ ,\label{majoration_3}\\
\varphi_{\text{max}}\left[ \left({\bf {\bf Z}}^*_{m}{\bf {\bf Z}}_{m}\right)^{-1}\right] & \leq & \left\{n\left[\left(1-\sqrt{\frac{d_{m}}{n}}-\sqrt{2x}\right)\vee 0\right]^2\right\}^{-1} \ , \label{majoration2_wishart}
\end{eqnarray}
for every model $m\in\mathcal{M}$.
Let us prove that for a suitable choice of the number $x$, $A_{\widehat{m}}\mathbf{1}_{\Omega'_1}$ is smaller than $7/8$. First, we constrain $n\kappa_2 \varphi_{\text{max}}\left[ \left({\bf Z}^*_{\widehat{m}}{\bf Z}_{\widehat{m}}\right)^{-1}\right]$ to be smaller than $\frac{K-1}{4}$ on the event $\Omega'_1$. By (\ref{majoration2_wishart}), it holds that
$$n\varphi_{\text{max}}\left[ \left({\bf Z}^*_{\widehat{m}}{\bf Z}_{\widehat{m}}\right)^{-1}\right] \leq \left[\left(1-\sqrt{\eta}-\sqrt{2x}\right)\vee 0\right]^{-2}\ .	$$
Constraining $x$ to be smaller than $\frac{\left(1-\sqrt{\eta}\right)^2}{8}$
ensures that the largest eigenvalue of $\left({\bf {\bf Z}}^*_{\widehat{m}} {\bf Z}_{\widehat{m}}\right)^{-1}$ satisfies 
$$n\varphi_{\text{max}}\left[ \left({\bf {\bf Z}}^*_{\widehat{m}}{\bf {\bf Z}}_{\widehat{m}}\right)^{-1}\right]\leq \frac{4}{\left(1-\sqrt{\eta}\right)^2}\ .$$
By definition (\ref{definition_kappa0}) of $\kappa_2$, it follows that
$n\kappa_2\varphi_{\text{max}}\left[ \left( {\bf Z}^*_{\widehat{m}} {\bf Z}_{\widehat{m}}\right)^{-1}\right] \leq (K-1)/4$.
Applying inequality $2ab\leq \delta a^2 + \delta^{-1}b^2$ to the bounds 
(\ref{majoration_1}), (\ref{majoration_2}), and (\ref{majoration_3}) yields 
\begin{eqnarray*}
-\frac{\|\Pi_{\widehat{m}}^{\perp}\boldsymbol{\epsilon}_{\widehat{m}}\|^2_n}{l(\theta_{\widehat{m}},\theta)} & \leq & -\frac{1}{2} +\frac{d_{\widehat{m}}}{2n}+2x\ \\
\kappa_2n \varphi_{\text{max}}\left[ \left({\bf {\bf Z}}^*_{\widehat{m}}{\bf {\bf Z}}_{\widehat{m}}\right)^{-1}\right]  \frac{\|\Pi_{\widehat{m}}(\boldsymbol{\epsilon}+\boldsymbol{\epsilon}_{\widehat{m}})\|_n^2}{\sigma^2+ l(\theta_{\widehat{m}},\theta)} & \leq & \frac{K-1}{2}\left[\frac{d_{\widehat{m}}}{n}+ \frac{3x}{2}\right]\\
-K\frac{d_{\widehat{m}}}{n-d{\widehat{m}}} \frac{\|\Pi_{\widehat{m}}^\perp(\boldsymbol{\epsilon}+\boldsymbol{\epsilon}_{\widehat{m}})\|_n^2}{\sigma^2+ l(\theta_{\widehat{m}},\theta)} & \leq & -K\frac{d_{\widehat{m}}}{2n}+ x\frac{2K\eta}{1-\eta}\ .
\end{eqnarray*}
Gathering these three inequalities, we get 
\begin{eqnarray*}
A_{\widehat{m}}\mathbf{1}_{\Omega'_1} \leq \frac{3}{4} + x\left[2+\frac{3(K-1)}{4}+2K\frac{\eta}{1-\eta}\right]\ .
\end{eqnarray*}
 If we set $x$ to 
$$x:=\left[8\left(2+\frac{3(K-1)}{4}+2K\frac{\eta}{1-\eta}\right)\right]^{-1}\wedge \frac{\left(1-\sqrt{\eta}\right)^2}{8}\ ,$$
then $A_{\widehat{m}}\mathbf{1}_{\Omega'_1}$ is smaller than $\frac{7}{8}$ 
and the result follows.
\end{proof}

\begin{proof}[Proof of Lemma \ref{majoration_esperance_conditionelle}]
We shall simultaneously bound the deviations of the random variables involved in the definition of $B_{m}$ for all models $m\in\mathcal{M}$. Let us first define the random variable $E_{m}$ as
\begin{eqnarray*}
E_{m} & := & \kappa_1^{-1}\frac{\langle\Pi_{m}^{\perp}\boldsymbol{\epsilon},\Pi_{m}^{\perp}\boldsymbol{\epsilon}_{m}
  \rangle^2_n}{\sigma^2l(\theta_{m},\theta)} + \frac{\|\Pi_{m}
  \boldsymbol{\epsilon}\|^2_n}{\sigma^2}\ .\label{E-defi}
\end{eqnarray*}
Factorizing by the norm of $\boldsymbol{\epsilon}$, we get
\begin{eqnarray}
E_{m} & \leq &
\kappa_1^{-1}\frac{\|\boldsymbol{\epsilon}\|_n^2}{\sigma^2}\frac{\langle\frac{\Pi_{m}^{\perp}\boldsymbol{\epsilon}}{\|\Pi_{m}^{\perp}\boldsymbol{\epsilon}
  \|_n},\Pi_{m}^{\perp}\boldsymbol{\epsilon}_{m}
  \rangle^2_n}{l(\theta_{m},\theta)} + \frac{\|\Pi_{m}
  \boldsymbol{\epsilon}\|^2_n}{\sigma^2}\ .\label{E_major}
\end{eqnarray}
The variable $n\frac{\|\boldsymbol{\epsilon}\|^2_n}{\sigma^2}$ follows a $\chi^2$
distribution with $n$ degrees of freedom. By Lemma \ref{concentration_chi2} there exists an event
$\Omega_2$ of probability larger than $1-\exp\left(n/8\right)$ such that 
$\frac{\|\boldsymbol{\epsilon}\|^2_n}{\sigma^2}$ is smaller than
$2$. As  $\kappa_1^{-1}=4$,  we obtain
\begin{eqnarray*}
E_{m}\mathbf{1}_{\Omega_2} & \leq &
8\frac{\langle\frac{\Pi_{m}^{\perp}\boldsymbol{\epsilon}}{\|\Pi_{m}^{\perp}\boldsymbol{\epsilon}
  \|_n},\Pi_{m}^{\perp}\boldsymbol{\epsilon}_{m}
  \rangle^2_n}{l(\theta_{m},\theta)} + \frac{\|\Pi_{m}
  \boldsymbol{\epsilon}\|^2_n}{\sigma^2}\ .
\end{eqnarray*}
Since $\boldsymbol{\epsilon}$, $\boldsymbol{\epsilon}_{m}$, and ${\bf X}_{m}$ are independent, it holds that conditionally on ${\bf X}_{m}$ and  $\boldsymbol{\epsilon}$, $$n\frac{\langle\frac{\Pi_{m}^{\perp}\boldsymbol{\epsilon}}{\|\Pi_{m}^{\perp}\boldsymbol{\epsilon}
  \|_n},\Pi_{m}^{\perp}\boldsymbol{\epsilon}_{m}
  \rangle^2_n}{l(\theta_{m},\theta)}\sim \chi^2(1)\ .$$
Since the distribution depends neither on ${\bf X}_{m}$ nor on $\boldsymbol{\epsilon}$, this random variable follows a $\chi^2$ distribution with $1$ degree of freedom. Besides, it is independent of the variable $\frac{\|\Pi_{m}.
  \boldsymbol{\epsilon}\|^2_n}{\sigma^2}$. Arguing as previously, we work out the distribution 
$$\frac{n\|\Pi_{m}
  \boldsymbol{\epsilon}\|^2_n}{\sigma^2} \sim \chi^2(d_{m})\ .$$
 Consequently, the variable $E_{m}\mathbf{1}_{\Omega_2}$ is upper bounded by a random variable that follows the distribution of 
$$\frac{8}{n}T_1 +\frac{1}{n}T_2\ ,$$
where $T_1$ and $T_2$ are two independent $\chi^2$ distribution with respectively $1$ and $d_{m}$ degrees of freedom.
Moreover, the random variables $n\frac{\|\Pi_{m}(\boldsymbol{\epsilon}+\boldsymbol{\epsilon}_{m})
    \|_n^2}{l(\theta_{m},\theta)+\sigma^2}$ and  $n\frac{\|\Pi_{m}^{\perp}(\boldsymbol{\epsilon}+\boldsymbol{\epsilon}_m)\|_n^2}{l(\theta_{m},\theta)+\sigma^2}$ respectively follow a $\chi^2$ distribution with $d_{m}$ and $n-d_{m}$ degrees of freedom. \\

Let us bound the deviations of the random variables $E_m\mathbf{1}_{\Omega_2}$, $\frac{\|\Pi_{m}(\boldsymbol{\epsilon}+\boldsymbol{\epsilon}_{m})
\|_n^2}{l(\theta_{m},\theta)+\sigma^2}$, and $\frac{\|\Pi_{m}^{\perp}(\boldsymbol{\epsilon}+\boldsymbol{\epsilon}_m)\|_n^2}{l(\theta_{m},\theta)+\sigma^2}$ for any model $m\in\mathcal{M}$. We apply Lemma 1 in \cite{Laurent98} for $E_m\mathbf{1}_{\Omega_2}$ and Lemma \ref{concentration_chi2} for the two remaining random variables. Hence, for any $x>0$, there exists an event $\mathbb{F}(x)$ of large probability
\begin{eqnarray*}
\mathbb{P}\left[\mathbb{F}(x)^c\right]& \leq &  e^{-x}\left(\sum_{m\in \mathcal{M}}
e^{-\xi_1d_m}+ e^{-\xi_2d_m}+ e^{-\xi_3d_m}\right) \\& \leq &e^{-x}\left[3+\alpha\sum_{d=1}^{+\infty} d^\beta(e^{-\xi_1d}+ e^{-\xi_2d}+ e^{-\xi_3d})\right]\ ,
\end{eqnarray*}
such that conditionally on $\mathbb{F}(x)$,
\begin{eqnarray*}\left\{\begin{array}{ccc}
E_{m}\mathbf{1}_{\Omega_2}
&  \leq & \frac{d_{m}+8}{n} + \frac{2}{n}\sqrt{\left[d_{m} +8^2\right]\left(\xi_1d_{m} +x\right)}+16\frac{\xi_1d_{m}+x}{n} \\
\frac{\|\Pi_{m}(\boldsymbol{\epsilon}+\boldsymbol{\epsilon}_{m})
    \|_n^2}{l(\theta_{m},\theta)+\sigma^2} & \leq & \frac{1}{n}\left(d_{m}+2\sqrt{d_{m}\left[d_{m}\xi_2+x\right]}+
2\left(d_{m}\xi_2 +x \right)\right)\\
- \frac{Kd_{m}}{n-d_{m}}\frac{\|(\Pi_{m}^{\perp}\boldsymbol{\epsilon}+\boldsymbol{\epsilon}_m)\|_n^2}{\sigma^2+l(\theta_{m},\theta)} & \leq & -\frac{Kd_{m}}{n(n-d_{m})}\left(n-d_{m}-
2\sqrt{(n-d_{m})(\xi_3d_{m}+x)}\right)\ ,\end{array}\right.
\end{eqnarray*}
for all models $m\in\mathcal{M}$. We shall fix later the positive constants $\xi_1$, $\xi_2$, and $\xi_3$.
Let us apply extensively the inequality $2ab\leq \tau a^2 + \tau^{-1}b^2$. Hence, conditionally on $\mathbb{F}(x)$, the model $\widehat{m}$ satisfies
\begin{eqnarray*}\left\{\begin{array}{ccc}
E_{\widehat{m}}\mathbf{1}_{\Omega_2} & \leq &  \frac{d_{\widehat{m}}}{n}\left[1+2\sqrt{\xi_1}+17\xi_1
  + \tau_1 \right]+
  \frac{x}{n}\left[17+\tau_1^{-1}\right] +
  \frac{72}{n}\\
\frac{\|\Pi_{\widehat{m}}(\boldsymbol{\epsilon}+\boldsymbol{\epsilon}_{\widehat{m}})
    \|_n^2}{l(\theta_{\widehat{m}},\theta)+\sigma^2} & \leq &   \frac{d_{\widehat{m}}}{n}\left[1+2\sqrt{\xi_2}+2\xi_2+\tau_2\right]+\frac{x}{n}\left[2+\tau_2^{-1}\right]\\
- \frac{Kd_{\widehat{m}}}{n-d_{\widehat{m}}}\frac{\|\Pi_{\widehat{m}}^{\perp}(\boldsymbol{\epsilon}+\boldsymbol{\epsilon}_{\widehat{m}})\|_n^2}{\sigma^2+ l(\theta_{\widehat{m}},\theta)}  & \leq &   -K\frac{d_{\widehat{m}}}{n}\left[1-2\sqrt{\xi_3\frac{d_{\widehat{m}}}{n-d_{\widehat{m}}}}-\tau_3\right]+K\frac{x}{n}\tau_3^{-1}\frac{d_{\widehat{m}}}{n-d_{\widehat{m}}}\ .\end{array}\right.
\end{eqnarray*}

By Lemma \ref{concentration_thrm1}, we know that conditionally on $\Omega_1$, $\kappa_2 n
\varphi_{\text{max}}\left[\left({\bf Z}^*_{\widehat{m}}{\bf Z}_{\widehat{m}}\right)^{-1}\right]$ is smaller than $\frac{K-1}{4}$. By assumption ($\mathbb{H}_{\eta}$),  the ratio $\frac{d_{\widehat{m}}}{n-d_{\widehat{m}}}$ is smaller than $\frac{\eta}{1-\eta}$. Gathering these inequalities we upper bound $B_{\widehat{m}}$ on the event $\Omega_1\cap \Omega_2\cap\mathbb{F}(x)$,
\begin{eqnarray*}
B_{\widehat{m}} \leq \frac{d_{\widehat{m}}}{n}U +\frac{x}{n}V +  \frac{72}{n}\ ,
\end{eqnarray*}
where $U$ and $V$ are defined as
\begin{eqnarray*}
  U&  := & 1+2\sqrt{\xi_1}+17\xi_1
  + \tau_1 + \frac{K-1}{4}\left[1+2\sqrt{\xi_2}+2\xi_2+\tau_2\right] - K\left[1 -
  2\sqrt{\xi_3}\sqrt{\frac{\eta}{1-\eta}} - \tau_3\right]\\
V & := & 17+\tau_1^{-1} + \frac{K-1}{4}\left[2 +
  \tau_2^{-1}\right]+ K\tau_3^{-1}\frac{\eta}{1-\eta}\ .
\end{eqnarray*}
Looking closely at $U$, one observes that it is the sum of the quantity $-\frac{3(K-1)}{4}$ and an expression that we can make arbitrary small by choosing the positive constants $\xi_1$, $\xi_2$,
$\xi_3$, $\tau_1$, $\tau_2$, and $\tau_3$ small enough. Consequently, there exists a suitable choice of these constants only depending on $K$ and $\eta$ that constrains the
quantity $U$ to be non positive. It follows that for any $x>0$, with probability larger than
$ 1- e^{-x}L(K,\eta,\alpha,\beta)$,
\begin{eqnarray*}
B_{\widehat{m}}\mathbf{1}_{\Omega_1\cap \Omega_2 } \leq \frac{x}{n} L(K,\eta)+ \frac{L'(K,\eta)}{n} \ .
\end{eqnarray*}
Integrating this upper bound for any $x>0$, we conclude 
\begin{eqnarray*}
\mathbb{E}\left[B_{\widehat{m}}\mathbf{1}_{\Omega_1\cap \Omega_2}\right] \leq \frac{L(K,\eta,\alpha,\beta)}{n} \ .
\end{eqnarray*}
\end{proof}

\begin{proof}[Proof of Lemma \ref{controlerisquesimple}]
We perform a very crude upper bound by controlling the
sum of the risk of every estimator $\widehat{\theta}_{m}$.
\begin{eqnarray*}
\mathbb{E}\left[l(\widetilde{\theta},\theta)1_{\Omega_1^c\cup \Omega_2^c}\right] \leq
\sqrt{\mathbb{P}(\Omega_1^c)+\mathbb{P}(\Omega_2^c)}\sqrt{\sum_{m\in \mathcal{M}}\mathbb{E}\left[l(\widehat{\theta}_{m},\theta)^2\right]}\ .
\end{eqnarray*}
As for any model $m\in\mathcal{M}$, $l(\widehat{\theta}_m,\theta) = l(\theta_m,\theta) + l(\widehat{\theta}_m,\theta_m)$, it follows that 
\begin{eqnarray*}
 \mathbb{E}\left[l(\widehat{\theta}_m,\theta)^2\right] \leq 2\left\{ l(\theta_m,\theta)^2 +\mathbb{E}\left[l(\widehat{\theta}_m,\theta_m)^2\right]\right\} \ .
\end{eqnarray*}
For any model $m\in \mathcal{M}$, it holds that $n-d_m-3\geq (1-\eta)n-3$, which is positive by assumption ($\mathbb{H}_{\eta}$). Hence, we may apply Proposition \ref{controle_risque_lp} with $r=2$ to all models $m\in \mathcal{M}$:
\begin{eqnarray*}
\mathbb{E}\left[l(\widehat{\theta}_m,\theta_m)^2\right] & \leq &L\left[ d_m n(\sigma^2+l(\theta_m,\theta))\right]^2 \\
& \leq & Ln^4\var(Y)^2\ , 
\end{eqnarray*}
since for any model $m$, $\sigma^2+l(\theta_m,\theta)\leq \var(Y)$. By summing this bound for all models $m\in\mathcal{M}$ and applying Lemma \ref{concentration_thrm1} and \ref{majoration_esperance_conditionelle}, we get
\begin{eqnarray*}
\mathbb{E}\left[l(\widetilde{\theta},\theta)\mathbf{1}_{\Omega_1^c\cup \Omega_2^c}\right] \leq
n^{2}\text{Card}(\mathcal{M})L(K,\eta)\var(Y)\exp\left[-nL'(K,\eta)\right]\ ,
\end{eqnarray*}
where $L'(K,\eta)$ is positive.

\end{proof}

\subsection{Proof of Theorem \ref{thrm_principal} and Proposition \ref{proposition_priori}}

\begin{proof}[Proof of Theorem \ref{thrm_principal}]
This proof follows the same approach as the one of Theorem
\ref{ordered_selection}. We shall only emphasize the differences with this previous proof. The bound (\ref{majoration_principal}) still holds. Let us respectively define the three constants $\kappa_1$, $\kappa_2$ and $\nu(K)$ as
\begin{eqnarray*}
\kappa_1  & := & \frac{\sqrt{\frac{3}{K+2}}}{1-\sqrt{\eta}-\nu(K)}\ , \, \, \, \, \, \, \, \, \, \, \, \,
\kappa_2  :=  \frac{(K-1)\left[1-\sqrt{\eta}\right]^2\left[1-\sqrt{\eta}-\nu(K)\right]^2}{16}\wedge 1\ ,\\
\nu(K) &:= &\left(\frac{3}{K+2}\right)^{1/6}\wedge \frac{1-\left(\frac{3}{K+2}\right)^{1/6}}{2}\ .
\end{eqnarray*}
We also introduce the random variables $A_{m'}$ and $B_{m'}$ for any model $m'\in\mathcal{M}$.
\begin{eqnarray*}
 A_{{m'}} & := &  \kappa_1+ 1 -
  \frac{\|\Pi^{\perp}_{m'}\boldsymbol{\epsilon}_{{m'}}\|^2_n}{l(\theta_{m'},\theta)}
  + \kappa_2 n \varphi_{\text{max}}\left[ ({\bf Z}^*_{{m'}}{\bf Z}_{{m'}})^{-1}\right]
   \frac{\|\Pi_{m'}(\boldsymbol{\epsilon} +\boldsymbol{\epsilon}_{{m'}})\|^2_n}{l(\theta_{m'},\theta)+\sigma^2}\\
  & - &K\left[1+\sqrt{2H(d_m')}\right]^2\frac{d_{{m'}}}{n-d_{{m'}}}   \frac{\|\Pi_{{m'}^\perp} (\boldsymbol{\epsilon}
  +\boldsymbol{\epsilon}_{{m'}})\|^2_n}{l(\theta_{m'},\theta)+\sigma^2}\ ,\\
B_{{m'}}&  := & \kappa_1^{-1}\frac{\langle\Pi^{\perp}_{m'}\boldsymbol{\epsilon},\Pi^{\perp}_{m'}\boldsymbol{\epsilon}_{{m'}}
  \rangle^2_n}{\sigma^2l(\theta_{m'},\theta)} + \frac{\|\Pi_{{m'}}
  \boldsymbol{\epsilon}\|^2_n}{\sigma^2} +\kappa_2 n
 \varphi_{\text{max}}\left[ \left( {\bf Z}^*_{{m'}}{\bf Z}_{{m'}}\right)^{-1}\right]\frac{\|\Pi_{{m'}} (\boldsymbol{\epsilon}
   +\boldsymbol{\epsilon}_{{m'}})\|^2_n}{l(\theta_{m'},\theta)+\sigma^2} \\ & - &  K\frac{d_{{m'}}}{n-d_{{m'}}}\left[1+\sqrt{2H(d_m')}\right]^2\frac{\|\Pi^{\perp}_{m'} (\boldsymbol{\epsilon}
   +\boldsymbol{\epsilon}_{{m'}})\|^2_n}{l(\theta_{m'},\theta)+\sigma^2}\  .
\end{eqnarray*}
The bound given in Lemma \ref{lemme_decomposition} clearly extends to
\begin{eqnarray*}
 -\overline{\gamma}_n(\widetilde{\theta}) -  \gamma_n(\widetilde{\theta})pen(\widehat{m}) -\sigma^2 +
\|\boldsymbol{\epsilon}\|^2_n \leq l(\widetilde{\theta},\theta)  \left[A_{\widehat{m}}\vee (1-\kappa_2)\right]+ \sigma^2 B_{\widehat{m}}\ .
\end{eqnarray*}
As previously, we control the variable $A_{\widehat{m}}$ on an event of large probability $\Omega_1$ and take the expectation of $B_{\widehat{m}}$ on an event of large probability $\Omega_1\cap \Omega_2$.

\begin{lemma}\label{concentration_thrm1_complete}
Let $\Omega_1$ be the event
\begin{eqnarray*}
\Omega_1 :=\left\{A_{\widehat{m}}\leq  s(K,\eta)\right\}\bigcap\left\{ \kappa_2 n\varphi_{\text{max}}\left[  ({\bf Z}^*_{\widehat{m}}{\bf Z}_{\widehat{m}})^{-1} \right]  \leq   \frac{(K-1)\left(1-\sqrt{\eta}-\nu(K)\right)^2}{4}\right\}\ ,
\end{eqnarray*}
where $s(K,\eta)$ is a function smaller than one. Then, $\mathbb{P}\left(\Omega_1^c\right)\leq L(K)n\exp\left[-nL'(K,\eta)\right]$ with $L'(K,\eta)>0$.
\end{lemma}
The function $s(K,\eta)$ is given explicitly in the proof of Lemma \ref{concentration_thrm1_complete}

\begin{lemma}\label{majoration_esperance_conditionelle_complete}
Let us assume that $n$ is larger than some quantities $n_0(K)$.
Then, there exists an event $\Omega_2$ of probability larger than $1-\exp\left[-nL(K,\eta)\right]$ where $L(K,\eta)>0$ such that
\begin{eqnarray*}
\mathbb{E}\left[B_{\widehat{m}}\mathbf{1}_{\Omega_1\cap \Omega_2}\right] \leq \frac{L(K,\eta)}{n}\ .
\end{eqnarray*}
\end{lemma}
Gathering inequalities (\ref{majoration_principal}), (\ref{decomposition_generale2}), Lemma
\ref{concentration_thrm1_complete} and \ref{majoration_esperance_conditionelle_complete},
we obtain as on the previous proof that
\begin{eqnarray}
 \mathbb{E}\left[l(\widetilde{\theta},\theta)\mathbf{1}_{\Omega_1\cap \Omega_2}
  \right] & \leq &
  L(K,\eta)\inf_{m\in\mathcal{M}}\left[l(\theta_m,\theta)+\left(\sigma^2+l(\theta_m,\theta)\right)pen(m)\right] + \nonumber\\ & + &L'(K,\eta)\left[\frac{\sigma^2}{n} +\left(\sigma^2+l(0_p,\theta)\right)n\exp\left[-nL''(K,\eta)\right]\right] \ .\label{majoration_risque_C_complete}
\end{eqnarray}
Afterwards, we control  the loss of the estimator $\widetilde{\theta}$ on
the event of small probability $\Omega_1^c\cup \Omega_2^c$. 
\begin{lemma}\label{lemme_controle_risque_complet}
If $n$ is larger than some quantity $n_0(K)$, 
 \begin{eqnarray*}
 \mathbb{E}\left[l(\widetilde{\theta},\theta)1_{\Omega_1^c\cup \Omega_2^c}\right] 
& \leq & n^{5/2}\left(\sigma^2 + l(0_p,\theta)\right)L(K,\eta)\exp\left[-nL'(K,\eta)\right]\ ,
\end{eqnarray*}
where $L(K,\eta)$ is positive.
\end{lemma}
Gathering this last bound with (\ref{majoration_risque_C_complete}) enables to conclude.
\end{proof}\vspace{0.5cm}

\begin{proof}[Proof of Lemma \ref{concentration_thrm1_complete}]
This proof is analogous to the proof of Lemma \ref{concentration_thrm1}, except that we shall
change the weights in the concentration inequalities in order to take into account the complexity of the collection of models. 
Let $x$ be a positive number we shall fix later. Applying Lemma \ref{concentration_chi2}, Lemma \ref{concentration_chi2_fine}, and Lemma \ref{concentration_vp_wishart} ensures that there exists an event $\Omega'_1$ such that
$$P(\Omega_1^{'c})\leq 4\exp(-nx)\sum_{m\in
  \mathcal{M}}\exp\left[-d_{m}H(d_m)\right]\ ,$$
and for all models $m\in \mathcal{M}$,
\begin{eqnarray}
\frac{\|\Pi_{m}^{\perp}\boldsymbol{\epsilon}_{m}\|^2_n}{l(\theta_{m},\theta)} &\geq & \frac{n-d_m}{n}\left[\left(1-\delta_{n-d_m}-\sqrt{\frac{2d_mH(d_m)}{n-d_m}}-\sqrt{\frac{2xn}{n-d_m}}\right)\vee 0\right]^2\label{minoration_1_complete}
\ , \\ 
\frac{\|\Pi_{m}(\boldsymbol{\epsilon}+\boldsymbol{\epsilon}_{m})\|_n^2}{\sigma^2+ l(\theta_{m},\theta)} & \leq\label{minoration_2_complete}
&\frac{2d_{m}}{n}\left[1+\sqrt{H(d_m)}+H(d_m)\right]+3x\ , \\
\frac{\|\Pi_{m}^\perp(\boldsymbol{\epsilon}+\boldsymbol{\epsilon}_{m})\|_n^2}{\sigma^2+ l(\theta_{m},\theta)} & \geq & 
\frac{n-d_m}{n}\left[\left(1-\delta_{n-d_m}-\sqrt{\frac{2d_mH(d_m)}{n-d_m}}-\sqrt{\frac{2xn}{n-d_m}}\right)\vee 0\right]^2\label{minoration_3_complete}
\ , \\ 
n\varphi_{\text{max}}\left[ \left({\bf {\bf Z}}^*_{m}{\bf {\bf Z}}_{m}\right)^{-1}\right] & \leq &
\left[\left(1-\left(1+\sqrt{2H(d_m)}\right)\sqrt{\frac{d_{m}}{n}}-\sqrt{2x}\right)\vee 0\right]^{-2}\ .\nonumber
\end{eqnarray} 
 We recall that $\delta_d$ is defined in (\ref{definition_delta}).
Besides, it holds that $$\mathbb{P}(\Omega_1^{'c})\leq 4\exp[-nx]\sum_{d=0}^{n}\text{Card}\left[\left\{m\in\mathcal{M},\ d_m=d\right\}\right]\exp[-dH(d)]\leq 4n\exp[-nx]\ .$$ By Assumption $(\mathbb{H}_{K,\eta})$, the expression $\left(1+\sqrt{2H(d_m)}\right)\sqrt{\frac{d_{m}}{n}}$ is bounded by $\sqrt{\eta}$. Hence, conditionally on $\Omega'_1$, 
\begin{eqnarray*}
n\varphi_{\text{max}}\left[ \left({\bf {\bf Z}}^*_{\widehat{m}}{\bf {\bf Z}}_{\widehat{m}}\right)^{-1}\right] & \leq & \left[\left(1-\sqrt{\eta}-\sqrt{2x}\right)\vee 0\right]^{-2}\ ,
\end{eqnarray*}
Constraining $x$ to be smaller than $\frac{\left(1-\sqrt{\eta}\right)^2}{8}$
ensures that $$n\kappa_2\varphi_{\text{max}}\left[\left({\bf {\bf Z}}^*_{\widehat{m}}{\bf {\bf Z}}_{\widehat{m}}\right)^{-1}\right]\mathbf{1}_{\Omega'_1}\leq \frac{(K-1)(1-\sqrt{\eta}-\nu(K))^2}{4}\ .$$
By assumption $(\mathbb{H}_{K,\eta})$, the dimension of any model $m\in\mathcal{M}$ is smaller than $n/2$. If  
$n$ is larger than some quantities only depending on $K$, then $\delta_{n/2}$ is smaller than $\nu(K)$. Let us assume first that this is the case. We recall that $\nu(K)$ is defined at the beginning of the proof of Theorem \ref{thrm_principal}.
Since $\nu(K)\leq 1-\sqrt{\eta}$, inequality (\ref{minoration_1_complete}) becomes
\begin{eqnarray*}
 \frac{\|\Pi_{m}^{\perp}\boldsymbol{\epsilon}_{\widehat{m}}\|^2_n}{l(\theta_{\widehat{m}},\theta)} &\geq & \left(1-\frac{d_{\widehat{m}}}{n}\right)\left[1-\nu(K)-\sqrt{\eta}\right]^2-2\sqrt{2x}\ .
\end{eqnarray*}
Bounding analogously the remaining terms of $A_{\widehat{m}}$, we get 
\begin{eqnarray*} 
A_{\widehat{m}} & \leq & \kappa_1+ 1-\left[1-\sqrt{\eta}-\delta_{n/2}\right]^2+ \frac{d_{\widehat{m}}}{n}(1-\sqrt{\eta}-\delta_{n/2})^2U_1
 +\sqrt{x}U_2+xU_3\ ,
\end{eqnarray*}
where $U_1$, $U_2$, and $U_3$ are respectively defined as
\begin{eqnarray*}\left\{\begin{array}{ccl}
U_1 & := & -K\left[1+\sqrt{2H(d_{\widehat{m}})}\right]^2+1 + (K-1)/2\left[1+\sqrt{H(d_{\widehat{m}})}\right]^2 \leq  0\\
U_2 & := & 2\sqrt{2}\left[1+K\eta\right]\\
U_3 & := & \frac{3}{4}(K-1)\left[1-\sqrt{\eta}-\nu(K)\right]^2 \ .
 \end{array}\right.
\end{eqnarray*}
Since $U_1$ is non-positive, we obtain an upper bound of $A_{\widehat{m}}$ that does not depend anymore on $\widehat{m}$. By assumption $(\mathbb{H}_{K,\eta})$, we know that  $\eta<(1-\nu(K)-(\frac{3}{K+2})^{1/6})^2$. Hence, coming back to the definition of $\kappa_1$ allows to prove that $\kappa_1$ is strictly smaller than $[1-\sqrt{\eta}-\nu(K)]^2$. Setting $$ x:= \left[\frac{\left[1-\sqrt{\eta}-\nu(K)\right]^2-\kappa_1}{4U_2}\right]^2\wedge\frac{\left[1-\sqrt{\eta}-\nu(K)\right]^2-\kappa_1}{4U_3} \wedge \frac{\left(1-\sqrt{\eta}\right)^2}{8}\ , $$ we get  
$$A_{\widehat{m}}\leq 1-\frac{1}{2}\left[\left(1-\sqrt{\eta}-\nu(K)\right)^2-\kappa_1\right]<1\ ,$$
on the event $\Omega'_1$. \\

In order to take into account the case $\delta_{n/2}\geq \nu(K)$, we only have to choose a large constant $L(K)$ in the upper bound of $\mathbb{P}(\Omega_1^c)$.
\end{proof}

\begin{proof}[Proof of Lemma \ref{majoration_esperance_conditionelle_complete}]

Once again, the sketch of the proof closely follows the proof of Lemma
\ref{majoration_esperance_conditionelle_complete}. Let us consider the random variables $E_{m}$ defined as
\begin{eqnarray*}
 E_m :=\kappa_1^{-1}\frac{\langle\Pi^{\perp}_{m'}\boldsymbol{\epsilon},\Pi^{\perp}_{m'}\boldsymbol{\epsilon}_{{m'}}
  \rangle^2_n}{\sigma^2l(\theta_{m'},\theta)} + \frac{\|\Pi_{{m'}}
  \boldsymbol{\epsilon}\|^2_n}{\sigma^2} \ .
\end{eqnarray*}
Since $n\|\boldsymbol{\epsilon}\|_n^2/\sigma^2$ follows a $\chi^2$
 distribution with $n$ degrees of freedom, there exists an event $\Omega_2$ of
 probability larger than $1- \exp\left[-nL(K) \right]$such that $\|\boldsymbol{\epsilon}\|_n^2/\sigma^2$ is smaller than $\kappa_1^{-1} = \sqrt{(K+2)/3}[1-\sqrt{\eta}-\nu(K)]$ on $\Omega_2$.  The constant $L(K)$in the exponential is positive. We shall simultaneously upper bound the deviations of the random variables $E_{m}$, $\frac{\|\Pi_{m}(\boldsymbol{\epsilon}+\boldsymbol{\epsilon}_{m})
    \|_n^2}{l(\theta_{m},\theta)+\sigma^2}$, and  $ \frac{\|\Pi_{m}^\perp(\boldsymbol{\epsilon}+\boldsymbol{\epsilon}_m)\|_n^2}{\sigma^2+l(\theta_{m},\theta)}$. 
Let $\xi$ be some positive constant that we shall fix later. For any $x>0$, we define an event $\mathbb{F}(x)$ such that conditionally on $\mathbb{F}(x)\cap \Omega_2$,
\begin{eqnarray*}\left\{\begin{array}{ccc}
E_{m}
&  \leq & \frac{d_{m}+\kappa_1^{-2}}{n} + \frac{2}{n}\sqrt{\left[d_{m} +\kappa_1^{-4}\right]\left[d_{m}(\xi+H(d_m)) +x\right]}\\
&+&2\kappa_1^{-2}\frac{\xi(d_{m}+H(d_m))+x}{n} \\
\frac{\|\Pi_{m}(\boldsymbol{\epsilon}+\boldsymbol{\epsilon}_{m})
    \|_n^2}{l(\theta_{m},\theta)+\sigma^2} & \leq & \frac{1}{n}\left[d_{m}+2\sqrt{d_{m}\left[d_{m}(\frac{1}{16}+H(d_m))+x\right]}+
2\left[ d_{m}(\frac{1}{16}+H(d_m)) +x \right]\right]\\
 \frac{\|\Pi_{m}^{\perp}\epsilon_{m}+\epsilon\|_n^2}{\sigma^2+l(\theta_{m},\theta)} & \geq & \frac{n-d_m}{n}\left[\left(1-\delta_{n-d_m}-
\sqrt{\frac{d_{m}(1+2H(d_m))}{n-d_m}}-\sqrt{\frac{ 2x}{n-d_m}}\right)\vee 0\right]^2\ ,\end{array}\right.
\end{eqnarray*}
for any model $m\in\mathcal{M}$. Then, the probability of $\mathbb{F}(x)$ satisfies
\begin{eqnarray*}
\mathbb{P}\left[\mathbb{F}(x)^c\right]& \leq &  e^{-x}\left[\sum_{m\in \mathcal{M}}\exp\left[-d_mH(d_m)\right]\left(
e^{-\xi d_m}+ e^{-\frac{d_m}{16}}+ e^{-\frac{d_m}{2}}\right)\right]\\
&\leq & e^{-x}\left(\frac{1}{1-e^{-\xi}}+\frac{1}{1-e^{-1/16}}+\frac{1}{1-e^{-1/2}}\right)\ .
\end{eqnarray*}
Let us expand the three deviation bounds thanks to the inequality $2ab\leq \tau a^2 + \tau^{-1}b^2$: 
\begin{eqnarray*}
E_{m} & \leq & \frac{d_{m}}{n}\left[1+2\sqrt{\xi} +
    2\kappa_1^{-2}\xi+ \tau_1\xi+\tau_2\right] +
    \frac{x}{n}\left[2\kappa_1^{-2} + \tau_2^{-1}+\tau_1\right]
\nonumber\\& + &
    \frac{\kappa_1^{-2}}{n} \left[1 + 
    \tau_1^{-1}\kappa_1^{-2}\right] +
    \frac{d_{m}H(d_m)}{n}\left[2\kappa_1^{-2} + \tau_1\right]
+2\frac{d_{m}\sqrt{H(d_m)}}{n}\\
& \leq & \frac{d_m}{n}\left(1+\sqrt{2H(d_m)}\right)^2\left[\kappa_1^{-2}+2\sqrt{\xi} +
    2\kappa_1^{-2}\xi+ \tau_1\xi+\tau_2 \right]\\
& + & \frac{x}{n}\left[2\kappa_1^{-2} + \tau_2^{-1}+\tau_1\right]+    \frac{\kappa_1^{-2}}{n} \left[1 + 
    \tau_1^{-1}\kappa_1^{-2}\right] \ .
\end{eqnarray*}
Similarly, we get
\begin{eqnarray*}
 \frac{\|\Pi_{m}(\boldsymbol{\epsilon}+\boldsymbol{\epsilon}_{m})
    \|_n^2}{l(\theta_{m},\theta)+\sigma^2} & \leq & 2\frac{d_m}{n}\left[1+\sqrt{2H(d_m)}\right]^2 + 5\frac{x}{n} \ .
\end{eqnarray*}
If $n$ is larger than some quantity $n_0(K)$, then $\delta_{n/2}$ is smaller than $\nu(K)$. Applying Assumption $(\mathbb{H}_{K,\eta})$, we get
\begin{eqnarray*}
 - K \frac{d_{m}}{n-d_{m}}\left(1+\sqrt{2H(d_m)}\right)^2
   \frac{\|\Pi_{m}^{\perp}(\boldsymbol{\epsilon}+\boldsymbol{\epsilon}_m)\|_n^2}{l(\theta_m,\theta)+\sigma^2}\hspace{6cm}\\
\hspace{2cm}\leq   -K\frac{d_{m}}{n}\left(1+\sqrt{2H(d_m)}\right)^2\left[\left(1
   -\sqrt{\eta}- \nu(K)-\sqrt{\frac{2x}{n-d_m}}\right)\vee 0\right]^2\\
\hspace{2cm}\leq -K\frac{d_{m}}{n}\left(1+\sqrt{2H(d_m)}\right)^2\left[\left(1
   -\sqrt{\eta}- \nu(K)\right)^2 -\tau_3\right]+2K\eta\tau_3^{-1}\frac{x}{n}\ .
\end{eqnarray*}
Let us combine these three bounds with the definitions of $B_{m}$, $\kappa_1$, and $\kappa_2$. Hence,
Conditionally to the event $\Omega_1\cap \Omega_2\cap \mathbb{F}(x)$, 
\begin{eqnarray}
B_{\widehat{m}} \leq \frac{d_{\widehat{m}}}{n}\left[1+\sqrt{2H({\widehat{m}})}\right]^2U_1 
+\frac{x}{n}U_2+ \frac{L(K,\eta)}{n}U_3\ ,	
\end{eqnarray}
where
\begin{eqnarray*}\left\{\begin{array}{ccl}
U_1 & := & -\frac{K-1}{6}\left(1-\sqrt{\eta}-\nu(K)\right)^2+K\tau_3+2\sqrt{\xi} +
    2\kappa_1^{-2}\xi+ \tau_1\xi+\tau_2 \ , \\
U_2 & := & \tau_2^{-1}+\tau_1+L(K,\eta)(1+\tau_3^{-1}) \ , \\
U_3  & := & 1+\tau_1^{-1} \ . 
\end{array}\right.
\end{eqnarray*}
Since $K>1$, there exists a suitable choice of the constants $\xi$, $\tau_1$, and $\tau_2$,
only depending on $K$ and $\eta$ that constrains $U_1$ to be non positive. Hence, 
conditionally on the event $\Omega_1\cap\Omega_2 \cap\mathbb{F}(x)$, 
$$B_{\widehat{m}}\leq \frac{L(K,\eta)}{n} + L'(K,\eta)\frac{x}{n}\ .$$
Since $\mathbb{P}\left[\mathbb{F}(x)^c\right]\leq e^{-x} L(K,\eta)$, we  conclude by integrating the last expression with respect to $x$.

\end{proof}

\begin{proof}[Proof of Lemma \ref{lemme_controle_risque_complet}]
As in the ordered selection case, we apply Cauchy-Schwarz inequality
\begin{eqnarray*}
\mathbb{E}\left[l(\widetilde{\theta},\theta)1_{\Omega_1^c\cup \Omega_2^c}\right] & \leq &
\sqrt{\mathbb{P}(\Omega_1^c)+\mathbb{P}(\Omega_2^c)}\sqrt{\mathbb{E}\left[l(\widetilde{\theta},\theta)^{2}\right]}\ .
\end{eqnarray*}
However, there are too many models to bound efficiently the risk of $\widetilde{\theta}$ by the sum of the risks of the estimators $\widehat{\theta}_m$. This is why we use here H\"older's inequality
\begin{eqnarray}
\mathbb{E}\left[l(\widetilde{\theta},\theta)1_{\Omega_1^c\cup \Omega_2^c}\right] & \leq & L(K)\sqrt{n}\exp\left[-nL(K,\eta)\right]\sqrt{\mathbb{E}\left[\sum_{m\in\mathcal{M}}\mathbf{1}_{m=\widehat{m}} l(\widehat{\theta}_m,\theta)^2\right] }\nonumber\\
& \leq & L(K)\sqrt{n}\exp\left[-nL(K,\eta)\right] \sqrt{\sum_{m\in\mathcal{M}}\mathbb{P}\left(m=\widehat{m}\right)^{1/u}\mathbb{E}\left[l(\widehat{\theta}_m,\theta )^{2v}\right]^{1/v}  },\label{majoration_complexe}
\end{eqnarray}
where  $v:=\left\lfloor\frac{n}{8}\right\rfloor$, and $u=:\frac{v}{v-1}$. We assume here that $n$ is larger than $8$.
For any model $m\in\mathcal{M}$, the loss $l(\widehat{\theta}_m,\theta)$ decomposes into the sum $l(\theta_m,\theta) + l(\widehat{\theta}_m,\theta_m)$. 
Hence,we obtain the following upper bound by applying Minkowski's inequality 
\begin{eqnarray}\label{majoration_decomposition}
\mathbb{E}\left[l(\widehat{\theta}_m ,\theta )^{2v}\right]^{1/2v}\leq l(\theta_m,\theta)+ \mathbb{E}\left[l(\widehat{\theta}_m,\theta_m)^{2v}\right]^{1/2v}\leq \var(Y)+ \mathbb{E}\left[l(\widehat{\theta}_m,\theta_m)^{2v}\right]^{1/2v}\ .
\end{eqnarray}

We shall upper bound this last term thanks to Proposition \ref{controle_risque_lp}. Since $v$ is smaller than $n/8$ and  since $d_m$ is smaller than $n/2$, it follows that for any model $m\in\mathcal{M}$, $n-d_m-4v+1$ is positive and 

\begin{eqnarray*}
\mathbb{E}\left[l(\widehat{\theta}_m,\theta_m)^{2v}\right]^{1/2v}\leq 2vLnd_m\left(\sigma^2+l(\theta_m,\theta)\right)\ ,
\end{eqnarray*}
for any model $m\in\mathcal{M}$. Since $d_m\leq n$ and since $\sigma^2+l(\theta_m,\theta)\leq \var(Y)$, we obtain
\begin{eqnarray}\label{majoration_risque_2v}
\mathbb{E}\left[l(\widehat{\theta}_m,\theta_m)^{2v}\right]^{1/2v}\leq 2vLn^2\var(Y)\ .
\end{eqnarray}
Gathering upper bounds (\ref{majoration_complexe}), (\ref{majoration_decomposition}), and (\ref{majoration_risque_2v})  we get
\begin{eqnarray*}
 \mathbb{E}\left[l(\widetilde{\theta},\theta)1_{\Omega_1^c\cup \Omega_2^c}\right] &\leq& L(K)\sqrt{n}\exp\left[-nL'(K,\eta)\right]\\&\times & \left[\var(Y)+ 2vLn^2\var(Y)\right]\sqrt{\sum_{m\in\mathcal{M}}\mathbb{P}\left(m=\widehat{m}\right)^{1/u}} \ .
\end{eqnarray*}
Since the sum over $m\in\mathcal{M}$ of  $\mathbb{P}\left(m=\widehat{m}\right)$ is one, the last term of the previous expression is maximized when every $\mathbb{P}\left(m=\widehat{m}\right)$ equals $\frac{1}{\text{Card}(\mathcal{M})}$. Hence, 
\begin{eqnarray*}
 \mathbb{E}\left[l(\widetilde{\theta},\theta)1_{\Omega_1^c\cup \Omega_2^c}\right] & \leq &n^{5/2}\var(Y)L(K,\eta) \text{Card}(\mathcal{M})^{1/(2v)}\exp\left[-nL'(K,\eta)\right]\ ,
\end{eqnarray*}
where $L'(K,\eta)$ is positive.
Let us first bound the cardinality of the collection $\mathcal{M}$. We recall that the dimension of any model $m\in\mathcal{M}$ is assumed to be smaller than $n/2$ by $(\mathbb{H}_{K,\eta})$.
Besides, for any $d\in \left\{1,\ldots,n/2\right\}$, there are less than $\exp(dH(d))$ models of dimension $d$. Hence,
$$\log\left(\text{Card}(\mathcal{M})\right)\leq \log(n)+\sup_{d=1,\ldots, n/2}dH(d) \ .$$
By assumption $(\mathbb{H}_{K,\eta})$, $dH(d)$ is smaller than  $n/2$. Thus, $\log(\text{Card}(\mathcal{M}))\leq \log(n)+n/2$ and it follows that $\text{Card}(\mathcal{M})^{1/(2v)}$ is smaller than an universal constant providing that $n$ is larger than 8. All in all, we get
\begin{eqnarray*}
\mathbb{E}\left[l(\widetilde{\theta},\theta)1_{\Omega_1^c\cup \Omega_2^c}\right] & \leq & n^{5/2}\var(Y)L(K,\eta)\exp\left[-nL'(K,\eta)\right]\ ,
\end{eqnarray*}
where $L'(K,\eta)$ is positive.
\end{proof}

\begin{proof}[Proof of Proposition \ref{proposition_priori}]
We apply the same arguments as in the proof of Theorem \ref{thrm_principal}, except that we replace $H(d_m)$ by $l_m$.
\begin{eqnarray*}
 A_{{m'}} & := &  \kappa_1+ 1 -
  \frac{\|\Pi^{\perp}_{m'}\boldsymbol{\epsilon}_{{m'}}\|^2_n}{l(\theta_{m'},\theta)}
  + \kappa_2 n \varphi_{\text{max}}\left[ ({\bf Z}^*_{{m'}}{\bf Z}_{{m'}})^{-1}\right]
   \frac{\|\Pi_{m'}(\boldsymbol{\epsilon} +\boldsymbol{\epsilon}_{{m'}})\|^2_n}{l(\theta_{m'},\theta)+\sigma^2}\\
  & - &K\left[1+\sqrt{2l_{m'}}\right]^2\frac{d_{{m'}}}{n-d_{{m'}}}   \frac{\|\Pi_{{m'}^\perp} (\boldsymbol{\epsilon}
  +\boldsymbol{\epsilon}_{{m'}})\|^2_n}{l(\theta_{m'},\theta)+\sigma^2}\ ,\\
B_{{m'}}&  := & \kappa_1^{-1}\frac{\langle\Pi^{\perp}_{m'}\boldsymbol{\epsilon},\Pi^{\perp}_{m'}\boldsymbol{\epsilon}_{{m'}}
  \rangle^2_n}{\sigma^2l(\theta_{m'},\theta)} + \frac{\|\Pi_{{m'}}
  \boldsymbol{\epsilon}\|^2_n}{\sigma^2} +\kappa_2 n
 \varphi_{\text{max}}\left[ \left( {\bf Z}^*_{{m'}}{\bf Z}_{{m'}}\right)^{-1}\right]\frac{\|\Pi_{{m'}} (\boldsymbol{\epsilon}
   +\boldsymbol{\epsilon}_{{m'}})\|^2_n}{l(\theta_{m'},\theta)+\sigma^2} \\ & - &  K\frac{d_{{m'}}}{n-d_{{m'}}}\left[1+\sqrt{2l{_m'}}\right]^2\frac{\|\Pi^{\perp}_{m'} (\boldsymbol{\epsilon}
   +\boldsymbol{\epsilon}_{{m'}})\|^2_n}{l(\theta_{m'},\theta)+\sigma^2}\  .
\end{eqnarray*}

In fact, Lemma \ref{concentration_thrm1_complete}, \ref{majoration_esperance_conditionelle_complete}, and \ref{lemme_controle_risque_complet} are still valid for this penalty. The previous proofs of these three lemma depend on the quantity $H(d_m)$ through the properties:
\begin{center}
$H(d_m)$ satisfies assumption $(\mathbb{H}_{K,\eta})$ and $\sum_{m\in\mathcal{M},\ d_m=d}\exp(-dH(d_m))\leq 1$.
\end{center}

Under the assumptions of Proposition \ref{proposition_priori}, $l_m$ satisfies the corresponding Assumption $(\mathbb{H}^l_{K,\eta})$ and is such that $\sum_{m\in\mathcal{M},\ d_m=d}\exp(-dl_m))\leq 1$. Hence, the proofs of these lemma remain valid in this setting if we replace $H(d_m)$ by $l_m$.

There is only one small difference at the end of the proof of Lemma \ref{lemme_controle_risque_complet} when bounding $\log\left(\text{Card}(\mathcal{M})\right)$. By definition of $l_m$,
$$\text{Card}(\mathcal{M})-1\leq \sup_{m\in\mathcal{M}\setminus\{\emptyset\}}\exp(d_ml_m)\ .$$ Hence, $\log(\text{Card}(\mathcal{M})\leq 1+\sup_{m\in\mathcal{M}\setminus\{\emptyset\}}d_ml_m$, which is smaller than $1+n/2$ by Assumption $(\mathbb{H}^l_{K,\eta})$. Hence, the upper bound shown in the proof of Lemma \ref{lemme_controle_risque_complet} is still valid.

\end{proof}

\subsection{Proof of Proposition \ref{controle_risque_lp}}\label{section_controle_lp}
\begin{proof}[Proof of Proposition \ref{controle_risque_lp}]
Let $m$ be a subset of $\{1,\ldots,p\}$. Thanks to (\ref{perte_biais}), we know that  $$l(\widehat{\theta}_m,\theta_m)= \left(\boldsymbol{\epsilon}+\boldsymbol{\epsilon}_m\right)^*{\bf Z}_m\left({\bf Z}^*_m{\bf Z}_m\right)^{-2}{\bf Z}^*_m\left(\boldsymbol{\epsilon}+\boldsymbol{\epsilon}_m\right)\ .$$ 
Applying Cauchy-Schwarz inequality, we decompose the $r$-th loss of $\widehat{\theta}_m$ in two terms
\begin{eqnarray}
 \mathbb{E}\left[l(\widehat{\theta}_m,\theta_m)^{r}\right]^{\frac{1}{r}} & \leq & \mathbb{E}\left[\left\|\left(\boldsymbol{\epsilon}+\boldsymbol{\epsilon}_m\right)\left(\boldsymbol{\epsilon}+\boldsymbol{\epsilon}_m\right)^*\right\|_F^r\|{\bf Z}_m\left({\bf Z}^*_m{\bf Z}_m\right)^{-2}{\bf Z}^*_m \|_F^r\right]^{\frac{1}{r}} \nonumber\\
& \leq & \mathbb{E}\left[\left\|\left(\boldsymbol{\epsilon}+\boldsymbol{\epsilon}_m\right)\left(\boldsymbol{\epsilon}+\boldsymbol{\epsilon}_m\right)^*\right\|_F^r\right]^{\frac{1}{r}}\mathbb{E}\left\{tr\left[\left({\bf Z}^*_m{\bf Z}_m\right)^{-2}\right]^\frac{r}{2}\right\}^{\frac{1}{r}},\label{lp_principal}
\end{eqnarray}
by independence of $\boldsymbol{\epsilon}$, $\boldsymbol{\epsilon}_m$, and ${\bf Z}_m$. Here,  $\|.\|_F$ stands for the Frobenius norm in the space of square matrices. We shall successively upper bound the two terms involved in (\ref{lp_principal}). 
\begin{eqnarray*}
 \left\|\left(\boldsymbol{\epsilon}+\boldsymbol{\epsilon}_m\right)\left(\boldsymbol{\epsilon}+\boldsymbol{\epsilon}_m\right)^*\right\|_F^r & = &  \left[\sum_{1\leq i ,j \leq n}\left(\boldsymbol{\epsilon}+\boldsymbol{\epsilon}_m\right)[i]^2\left(\boldsymbol{\epsilon}+\boldsymbol{\epsilon}_m\right)[j]^2\right]^{r/2}\ .
\end{eqnarray*}
This last expression corresponds to the $L_{r/2}$ norm of a Gaussian chaos of order 4. By Theorem 3.2.10 in \cite{pena}, such chaos satisfy a Khintchine-Kahane type inequality: 
\begin{lemma}
For all $d\in\mathbb{N}$ there exists a constant $L_d\in (0,\infty)$ such that, if $X$ is a Gaussian chaos of order $d$ with values in any normed space $F$ with norm $\|.\|$ and if $1<s<q<\infty$, then
\begin{eqnarray*}
 \left(\mathbb{E}\left\|X\right\|^q\right)^{\frac{1}{q}}\leq L_d\left(\frac{q-1}{s-1}\right)^{d/2}\mathbb{E}\left[\left\|X\right\|^s\right]^\frac{1}{s} \ .
\end{eqnarray*}
\end{lemma}
Let us assume that $r$ is larger than four. Applying the last lemma with $d=4$, $q=r/2$, and $s=2$ yields
$$ \mathbb{E}\left[\left\|\left(\boldsymbol{\epsilon}+\boldsymbol{\epsilon}_m\right)\left(\boldsymbol{\epsilon}+\boldsymbol{\epsilon}_m\right)^*\right\|_F^r\right]^{\frac{2}{r}}\leq L_4(r/2-1)^2\mathbb{E}\left[\left\|\left(\boldsymbol{\epsilon}+\boldsymbol{\epsilon}_m\right)\left(\boldsymbol{\epsilon}+\boldsymbol{\epsilon}_m\right)^*\right\|_F^4\right]^{\frac{1}{2}}. $$
By standard Gaussian properties, we compute the fourth moment of this chaos and obtain
\begin{eqnarray*}
 \mathbb{E}\left[\left\|\left(\boldsymbol{\epsilon}+\boldsymbol{\epsilon}_m\right)\left(\boldsymbol{\epsilon}+\boldsymbol{\epsilon}_m\right)^*\right\|_F^4\right]^{\frac{1}{2}} & \leq & L n^2\left[\sigma^2 +l(\theta_m,\theta)\right]^2\ .
\end{eqnarray*}
Hence, we get the upper bound
\begin{eqnarray}\label{majoration_chaos}
 \mathbb{E}\left[\left\|\left(\boldsymbol{\epsilon}+\boldsymbol{\epsilon}_m\right)\left(\boldsymbol{\epsilon}+\boldsymbol{\epsilon}_m\right)^*\right\|_F^r\right]^{\frac{1}{r}}\leq L(r-1)n\left[\sigma^2+l(\theta_m,\theta)\right]\ .
\end{eqnarray}
Straightforward computations allow to extend this bound to $r=2$ and $r=3$.~\\

Let us turn to bounding the second term of (\ref{lp_principal}).
Since the eigenvalues of the matrix $\left({\bf Z}^*_m{\bf Z}_m\right)^{-1}$ are almost surely non-negative,  it follows that $$tr\left[\left({\bf Z}^*_m{\bf Z}_m\right)^{-2}\right] \leq tr\left[\left({\bf Z}^*_m{\bf Z}_m\right)^{-1}\right]^2\ .$$ Consequently, we shall upper bound the $r$-th moment of the trace of an inverse standard Wishart matrix. 
For any couple of matrices $A$ and $B$ respectively of size $p_1\times q_1$ and $p_2\times q_2$, we define the Kronecker product matrix $A\otimes B$ as the matrix of size $p_1p_2\times q_1 q_2$ that satisfies:
$$A\otimes B[i_2+p_2(i_1-1);j_2+q_2(j_1-1)] := A[i_1;j_1]B[i_2;j_2]\ ,\hspace{0.5cm}\text{for any }\left\{\begin{array}{c} 1\leq i_1\leq p_1\\ 1\leq i_2\leq p_2\\ 1\leq j_1\leq q_1\\ 1\leq j_2\leq q_2\end{array}\right.\ .$$ 
For any matrix $A$, $\otimes^kA$ refers to the $k$-th power of $A$ with respect to the Kronecker product. Since $tr(A)^k = tr\left(\otimes^k A\right)$ for any square matrix $A$, we obtain 
\begin{eqnarray*}
 \mathbb{E}\left[tr({\bf Z}^*_m{\bf Z}_m)^{-1}\right]^k & = &\mathbb{E}\left[tr\left(\otimes^k ({\bf Z}^*_m{\bf Z}_m)^{-1}\right)\right]\\
& = & tr\left[\mathbb{E}\left(\otimes^k ({\bf Z}^*_m{\bf Z}_m)^{-1}\right)\right]\\
& \leq & \sqrt{d_m^k}\left\|\mathbb{E}\left[\otimes^k ({\bf Z}^*_m{\bf Z}_m)^{-1}\right]\right\|_F\ , 
\end{eqnarray*}
thanks to Cauchy-Schwarz inequality. In Equation (4.2) of  \cite{rosen1988}, Von Rosen has characterized  recursively the expectation of $\otimes^k ({\bf Z}^*_m{\bf Z}_m)^{-1}$ as long as $n-d_m-2k-1$ is positive:
\begin{eqnarray}\label{majoration_kronecker}
\text{vec}\left(\mathbb{E}\left[\otimes^{k+1} ({\bf Z}^*_m{\bf Z}_m)^{-1}\right]\right) = A(n,d_m,k)^{-1}\text{vec}\left(\mathbb{E}\left[\otimes^{k} ({\bf Z}'_m{\bf Z}_m)^{-1}\right]\otimes I\right)\ ,
\end{eqnarray}
 where $\text{'vec'}$ refers to the vectorized version of the matrix. See Section 2 of \cite{rosen1988} for more details about this definition. $A(n,d_m,k)$ is a symmetric matrix of size $d_m^{k+1}\times d_m^{k+1}$ which only depends on $n$, $d_m$, and $k$ and is known to be diagonally dominant. More precisely, any diagonal element of $A(n,d_m,k)$ is greater or equal to one plus the corresponding row sums of the absolute values of the off-diagonal elements. Hence, the matrix $A$ is invertible and its smallest eigenvalue is larger or equal to one. Consequently, $\varphi_{\text{max}}\left( A^{-1}\right)$ is smaller or equal to one. It then follows from (\ref{majoration_kronecker}) that
\begin{eqnarray*}
\left\|\mathbb{E}\left[\otimes^{k+1} ({\bf Z}^*_m{\bf Z}_m)^{-1}\right]\right\|_F & = &\left\|\text{vec}\left(\mathbb{E}\left[\otimes^{k+1} ({\bf Z}^*_m{\bf Z}_m)^{-1}\right]\right)\right\|_F \\& \leq& \varphi_{\text{max}}( A^{-1} ) \left\|\text{vec}\left(\mathbb{E}\left[\otimes^{k} ({\bf Z}^*_m{\bf Z}_m)^{-1}\right]\otimes I\right)\right\|_F\\
& \leq & \sqrt{d_m} \left\|\mathbb{E}\left[\otimes^k ({\bf Z}^*_m{\bf Z}_m)^{-1}\right]\right\|_F\ .
\end{eqnarray*}
By induction, we obtain	
\begin{eqnarray}\label{majoration_risque2}
\mathbb{E}\left[tr({\bf Z}^*_m{\bf Z}_m)^{-1}\right]^r  \leq d_m^r\ ,
\end{eqnarray}
if $n-d_m-2r+1>0$. Combining upper bounds (\ref{majoration_chaos}) and (\ref{majoration_risque2}) enables to conclude
\begin{eqnarray*}
 \mathbb{E}\left[l(\widehat{\theta}_m,\theta_m)^{r}\right]^{\frac{1}{r}} & \leq &L rd_mn(\sigma^2+l(\theta_m,\theta))\ .
\end{eqnarray*}
\end{proof}

\subsection{Proof of Proposition \ref{optimal}}
\begin{proof}[Proof of Proposition \ref{optimal}]

Let $m_*$ be the model that minimizes the loss function $l(\widehat{\theta}_m,\theta)$:
\begin{eqnarray*}
m_* = \arg \inf_{m\in\mathcal{M}_{\lfloor n/2\rfloor}} l(\widehat{\theta}_m,\theta)\ .
\end{eqnarray*}
It is almost surely uniquely defined. Contrary to the oracle $m^*$,  the model $m_*$ is random.
By definition of $\widehat{m}$, we derive that
\begin{eqnarray}
l(\widetilde{\theta},\theta) \leq l(\widehat{\theta}_{m_*},\theta) + \gamma_n(\widehat{\theta}_{m^*})pen(m_*) +
\overline{\gamma}_n(\widehat{\theta}_{m_*}) - \gamma_n(\widetilde{\theta})pen(\widehat{m}) - \overline{\gamma}_n(\widetilde{\theta})\ , 
\label{majoration_principal_optimal}\end{eqnarray}
where $\overline{\gamma}_n$ is defined in the proof of Theorem $\ref{ordered_selection}$.
The proof divides in two parts. First, we state that on an event $\Omega_1$ of large probability, the dimensions of $\widehat{m}$ and of $m^*$ are moderate. Afterwards, we prove that on another event of large probability  $\Omega_1\cap \Omega_2\cap \Omega_3$, the ratio $l(\widetilde{\theta},\theta)/l(\widehat{\theta}_{m^*},\theta)$ is close to one.
\begin{lemma}\label{controle_dimension}
Let us define the event $\Omega_1$ as:
$$\Omega_1:=\left\{\log^2(n) <d_{m_*}< \frac{n}{\log n} \, \, \, \, \, \text{and}  \, \, \, \, \,\log^2(n) <d_{\widehat{m}}< \frac{n}{\log n}\right\}\ .$$
The event $\Omega_1$ is achieved with large probability: $\mathbb{P}\left(\Omega_1\right)\geq 1-\frac{L(R,s)}{n^2}$. 
\end{lemma}

\begin{lemma}\label{controle_mhat}
There exists an event $\Omega_2$  of probability larger than $1- L\frac{\log n}{n}$ such that
\begin{eqnarray*}
 \left[-\overline{\gamma}_n(\widetilde{\theta}) -  \gamma_n(\widetilde{\theta})pen(\widehat{m}) -\sigma^2 +
\|\boldsymbol{\epsilon}\|^2_n\right]\mathbf{1}_{\Omega_1\cap\Omega_2} \leq  l(\widetilde{\theta},\theta)\tau_1(n),
\end{eqnarray*}
where $\tau_1(n)$ is a positive sequence converging to zero when $n$ goes to infinity.
\end{lemma}
\begin{lemma}\label{controle_m*}
There exists an event $\Omega_3$  of probability larger than $1- L\frac{\log n}{n}$ such that
\begin{eqnarray*}
 \left[\overline{\gamma}_n(\widehat{\theta}_{m^*}) +  \gamma_n(\widehat{\theta}_{m^*})pen(m^*) +\sigma^2 -
\|\boldsymbol{\epsilon}\|^2_n\right]\mathbf{1}_{\Omega_1\cap\Omega_3} \leq  l\left(\widehat{\theta}_{m_*},\theta\right)\tau_2(n),
\end{eqnarray*}
where $\tau_2(n)$ is a positive sequence converging to zero when $n$ goes to infinity. 
\end{lemma}

Gathering these three lemma, we derive from the upper bound (\ref{majoration_principal_optimal}) the inequality
\begin{eqnarray*}
 \frac{l(\widetilde{\theta},\theta)}{l(\widehat{\theta}_{m_*},\theta)}\mathbf{1}_{\Omega_1\cap\Omega_2\cap\Omega_3}\leq \frac{1+\tau_2(n)}{1-\tau_1(n)} \ ,
\end{eqnarray*}
 which allows to conclude.

\end{proof}
\begin{proof}[Proof of Lemma \ref{controle_dimension}]

Let us consider the model $m_{R,s}$ defined by $d_{m_{R,s}} := \lfloor(nR^2)^{\frac{1}{1+s}}\rfloor$. If $n$ is larger than some quantity $L(R,s)$, then $d_{m_{R,s}}$ is smaller than $n/2$ and $m_{R,s}$ therefore belongs to the collection $\mathcal{M}_{\lfloor n/2\rfloor}$. We shall prove that outside an event of small probability, the loss $l(\widehat{\theta}_{m_{R,s}},\theta)$ is smaller than the loss $l(\widehat{\theta}_{m},\theta)$ of all models $m\in \mathcal{M}_{\lfloor n/2\rfloor}$ whose dimension is smaller than $\log^2(n)$ or larger than $\frac{n}{\log n}$. Hence, the model $m_*$ satisfies $\log^2(n) < d_{m_*} <\frac{n}{\log n}$ with large probability. ~\\

First, we need to upper bound the loss $l(\widehat{\theta}_{m_{R,s}},\theta)$.
Since $l(\widehat{\theta}_{m_{R,s}},\theta)= l(\theta_{m_{R,s}},\theta) + l(\widehat{\theta}_{m_{R,s}},\theta_{m_{R,s}})$, it comes to upper bounding both the bias term and the variance term. Since $\theta$ belongs to $\mathcal{E}'_s(R)$, 
\begin{eqnarray}
 l\left(\theta_{m_{R,s}},\theta\right) &  = & \sum_{i>d_{m_{R,s}}}^{+\infty} l(\theta_{m_{i-1}},\theta_{m_i}) \nonumber\\
& \leq & (d_{m_i}+1)^{-s} \sum_{i>d_{m_{R,s}}}^{+\infty} \frac{l(\theta_{m_{i-1}},\theta_{m_i})}{i^{-s}} \leq  \sigma^2 \left(\frac{R^2}{n^s}\right)^\frac{1}{1+s} \ .\label{majoration_biais}
\end{eqnarray}
Then, we bound the variance term $l(\widehat{\theta}_{m_{R,s}},\theta_{m_{R,s}})$  thanks to (\ref{decomposition3}) as in the proof of Lemma  \ref{lemme_decomposition}. 
\begin{eqnarray*}
 l\left(\widehat{\theta}_{m_{R,s}},\theta_{m_{R,s}}\right)& \leq & \left[\sigma^2+l\left(\theta_{m_{R,s}},\theta\right)\right]\varphi_{\text{max}}\left[ n({\bf Z}^*_{m_{R,s}} {\bf Z}_{m_{R,s}})^{-1}\right] \frac{\left\|\Pi_{m_{R,s}}(\boldsymbol{\epsilon}+\boldsymbol{\epsilon}_{m_{R,s}})\right\|_n^2}{\sigma^2 + l(\theta_{m_{R,s}},\theta)}\ .
\end{eqnarray*}
The two random variables involved in this last expression respectively follow (up to a factor $n$) the distribution of an inverse Wishart matrix with parameters $(n,d_{m_{R,s}})$ and a $\chi^2$ distribution with $d_{m_{R,s}}$ degrees of freedom. Thanks to Lemma \ref{concentration_chi2} and \ref{concentration_vp_wishart}, we prove that outside an event of probability smaller than $L(R,s)\exp[-L'(R,s)n^\frac{1}{1+s}]$ with $L'(R,s)>0$, 
\begin{eqnarray*}
 l\left(\widehat{\theta}_{m_{R,s}},\theta_{m_{R,s}}\right)& \leq & 4\left[\sigma^2+l\left(\theta_{m_{R,s}},\theta\right)\right]   \frac{d_{m_{R,s}}}{n}\ ,
\end{eqnarray*}
if $n$ is large enough. Gathering this last upper bound with (\ref{majoration_biais}) yields 
\begin{eqnarray}
 l\left(\widehat{\theta}_{m_{R,s}},\theta\right) \leq \sigma^2\left[5\frac{R^{\frac{2}{1+s}}}{n^{\frac{s}{1+s}}}+ 4\left(\frac{R^{\frac{2}{1+s}}}{n^{\frac{s}{1+s}}}\right)^2\right]\leq \sigma^2\frac{C(R,s)}{n^{\frac{s}{1+s}}}\ \label{majoration_perte_m1}
\end{eqnarray} 
where $C(R,s)$ is a constant that only depends on $R$ and $s$.

Let us prove that the bias term of any model of dimension smaller than $\log^2(n)$ is larger than (\ref{majoration_perte_m1}) if $n$ is large enough. Obviously, we only have to consider the model of dimension $\lfloor \log^2(n)\rfloor$. Assume that there exists an infinite increasing sequence of integers $u_n$ satisfying:
\begin{eqnarray}\label{condition_impossible}
\sum_{i>\log^2(u_n)} l\left(\theta_{m_{i-1}},\theta_{m_{i}}\right) \leq  \frac{C(R,s)}{(u_{n+1})^{\frac{s}{1+s}}} \ . 
\end{eqnarray}
 Then, the sequence $(v_n)$ defined by $v_n:= \log^2(u_n)$ satisfies
\begin{eqnarray*}
\sum_{i>v_n} l\left(\theta_{m_{i-1}},\theta_{m_{i}}\right) \leq  C(R,s)\exp\left[-\sqrt{v_{n+1}}\frac{s}{1+s}\right] \ . 
\end{eqnarray*}
Let us consider a subsequence of $(v_n)$ such that $\lfloor v_n\rfloor$ is strictly increasing. For the sake of simplicity we still call it $v_n$. It follows that
\begin{eqnarray*}
 \sum_{i=\lfloor v_0\rfloor +1}^{+\infty} \frac{l\left(\theta_{m_{i-1}},\theta_{m_{i}}\right)}{i^{-s'}} & = & \sum_{n=0}^{+\infty}\sum_{i=\lfloor v_n\rfloor +1}^{\lfloor v_{n+1}\rfloor} \frac{l\left(\theta_{m_{i-1}},\theta_{m_{i}}\right)}{i^{-s'}} \nonumber\\
& \leq & C(R,s)\sum_{n=0}^{+\infty} \lfloor v_{n+1}\rfloor^{s'}\exp\left[-\sqrt{\lfloor v_{n+1}\rfloor}\frac{s}{1+s}\right]  < \infty \ ,
\end{eqnarray*}
and $\theta$ therefore belongs to some ellipsoid $\mathcal{E}_{s'}(R')$. 
This contradicts the assumption $\theta$ does not belong to any ellipsoid $\mathcal{E}_{s'}(R')$. As a consequence, there only exists a finite sequence of integers $u_n$ that satisfy Condition (\ref{condition_impossible}). For $n$ large enough, the bias term of any model of dimension less than $\log^2(n)$ is therefore larger than the loss  $l(\widehat{\theta}_{m_{R,s}},\theta)$ with overwhelming probability.
	
Let us turn to the models of dimension larger than $n/\log n$. We shall prove that with large probability, for any model $m$ of dimension larger than $n/\log n$, the variance term $l(\widehat{\theta}_m,\theta_m)$ is larger than the order $\sigma^2/\log n$.
For any model $m\in \mathcal{M}_{\lfloor n/2\rfloor}$,  
\begin{eqnarray*}
l\left(\widehat{\theta}_m,\theta_m\right) & \geq & \frac{n\sigma^2}{\varphi_{\text{max}}\left( {\bf Z}^*_m {\bf Z}_m\right)}\frac{\left\|\Pi_m(\boldsymbol{\epsilon}+\boldsymbol{\epsilon}_m)\right\|_n^2}{\sigma^2+l(\theta_m,\theta)} \ .
\end{eqnarray*}
The two random variables involved in this expression respectively follow (up to a factor $n$) a Wishart distribution with parameters $(n,d_m)$ and a $\chi^2$ distribution with $d_m$. Again, we apply Lemma \ref{concentration_chi2} and \ref{concentration_vp_wishart} to control the deviations of these random variables. Hence, outside an event of probability smaller than $L(\xi)\exp[-n\xi/\log n]$, 
\begin{eqnarray*}
 l\left(\widehat{\theta}_m,\theta_m\right) & \geq & \sigma^2\left(1+\sqrt{\frac{d_m}{n}}+ \sqrt{2\xi\frac{d_m}{n}}\right)^{-2}\frac{d_m}{n}\left(1-2\sqrt{\xi}\right) \ ,
\end{eqnarray*}
for any model $m$ of dimension larger than $n/\log n$. For any model $m\in \mathcal{M}_{\lfloor n/2\rfloor}$, the ratio $d_m/n$ is smaller than $1/2$. As a consequence, we get
\begin{eqnarray*}
  l\left(\widehat{\theta}_m,\theta_m\right)  \geq  \frac{\sigma^2}{\log n}\left(1-2\sqrt{\xi}\right)\left(1+\sqrt{1/2}+\sqrt{\xi}\right)^{-2}\ .
\end{eqnarray*}
Choosing for instance $\xi=1/16$ ensures that for $n$ large enough the loss $l(\widehat{\theta}_m,\theta_m)$ is larger than  $l(\widehat{\theta}_{m_{R,s}},\theta)$ for every model $m$ of dimension larger than $n/\log n$ outside an event of probability smaller than $L_1\exp[-L_2 n/\log n]+ L_3(R,s)\exp[-L_4(R,s)n^{1/(1+s)}]$ with $L_4(R,s)>0$.\\

Let us now turn to the selected model $\widehat{m}$. We shall prove that outside an event of small probability, 
\begin{eqnarray}\label{comparaison_critere}
 \gamma_{n}\left(\widehat{\theta}_{m_{R,s}}\right)[1+pen(m_{R,s})]\leq \gamma_{n}\left(\widehat{\theta}_{m}\right)[1+pen(m)]\ , 
\end{eqnarray}
for all models $m$ of dimension smaller than $\log^2 n$ or larger than $n/\log n$.
We first consider the models of dimension smaller than $\log^2(n)$. For any model $m\in \mathcal{M}_{\lfloor n/2\rfloor}$ , $\gamma_{n}(\widehat{\theta}_{m})*n/[\sigma^2+l(\theta_m,\theta)]$ follows a $\chi^2$ distribution with $n-d_m$ degrees of freedom.
Again, we apply Lemma \ref{concentration_chi2}. Hence, with probability larger than $1- e/[n^2(e-1)]$, the following upper bound holds for any model $m$ of dimension smaller than $\log^2(n)$.
\begin{eqnarray*}
\gamma_{n}\left(\widehat{\theta}_{m}\right)[1+pen(m)]& \geq & \sigma^2\left[1+\frac{l(\theta_m,\theta)}{\sigma^2}\right]\left(1+2\frac{d_m}{n-d_m}\right)\left[\frac{n-d_m}{n}-2\frac{\sqrt{(n-d_m)(d_m+2\log(n))}}{n}\right] \\
& \geq & \sigma^2\left[1+\frac{l(\theta_m,\theta)}{\sigma^2}\right]\left(1+\frac{d_m}{n}\right)\left[1-2\sqrt{\frac{d_m+2\log(n)}{n-d_m}}\right] \\
& \geq & \sigma^2\left[1+\frac{l(\theta_m,\theta)}{\sigma^2}\right] \left[1-4\frac{\log n}{\sqrt{n}}\right]\ ,
\end{eqnarray*}
for $n$ large enough. Besides, outside an event of probability smaller than $\frac{1}{n^2}$,  
\begin{eqnarray*}
\gamma_{n}\left(\widehat{\theta}_{m_{R,s}}\right)[1+pen(m_{R,s})]& \leq & \sigma^2\left[1+\frac{l(\theta_{m_{R,s}},\theta)}{\sigma^2}\right]\left(1+2\frac{d_{m_{R,s}}}{n-d_{m_{R,s}}}\right)\times \\& &\left[\frac{n-d_{m_{R,s}}}{n}+2\frac{\sqrt{(n-d_{m_{R,s}})2\log n}}{n}+4\frac{\log n}{n}\right] \\ 
& \leq & \sigma^2\left[1+\frac{l(\theta_{m_{R,s}},\theta)}{\sigma^2}\right] \left(1+\frac{d_{m_{R,s}}}{n}\right)\left[1 +2\frac{\sqrt{2\log n}}{\sqrt{n-d_{m_{R,s}}}}+4\frac{\log n}{n-d_{m_{R,s}}}\right]\ .
\end{eqnarray*}
For $n$ large enough, $d_{m_{R,s}}$ is smaller than $\frac{n}{2}$, and the last upper bound becomes:
\begin{eqnarray*}
 \gamma_{n}\left(\widehat{\theta}_{m_{R,s}}\right)[1+pen(m_{R,s})] \leq \sigma^2\left[1+\frac{C(R,s)}{n^{\frac{s}{1+s}}}\right]^2\left(1+10\frac{\log(n)}{\sqrt{n}}\right)\ .
\end{eqnarray*}
Hence, $\gamma_{n}\left(\widehat{\theta}_{m_{R,s}}\right)[1+pen(m_{R,s})]\leq \gamma_{n}\left(\widehat{\theta}_{m}\right)[1+pen(m)]$ if 
$$\frac{l(\theta_{m_{\lfloor \log^2 n\rfloor}},\theta)}{\sigma^2}\geq 3\frac{C(R,s)}{n^{\frac{s}{1+s}}}\times\frac{1+10\log(n)/\sqrt{n}}{1-4\log(n)/\sqrt{n}}+ 14\frac{\log(n)}{\sqrt{n}}\ .$$   
As previously, this inequality always holds except for a finite number of $n$, since $\theta$ does not belong to any ellipsoid $\mathcal{E}_{s'}(R')$. Thus, outside an event of probability smaller than $\frac{L}{n^2}$, $d_{\widehat{m}}$ is larger than $\log^2 n$.\\

Let us now turn to the models of large dimension. Inequality (\ref{comparaison_critere}) holds if the quantity
\begin{eqnarray}\label{comparaison_critere2}
{\small \|\boldsymbol{\epsilon}\|_n^2\left(\frac{2d_{m_{R,s}}}{n-d_{m_{R,s}}}- \frac{2d_{m}}{n-d_{m}}\right)+  \|\Pi_{m}\boldsymbol{\epsilon}\|_n^2\left(1+\frac{2d_m}{n-d_{m}}\right) + \langle\Pi_{m_{R,s}}^\perp\boldsymbol{\epsilon}_{m_{R,s}}, \Pi_{m_{R,s}}^\perp\boldsymbol{\epsilon}+2\boldsymbol{\epsilon}_{m_{R,s}}\rangle_n\left(1+\frac{2d_{m_{R,s}}}{n-d_{m_{R,s}}}\right)} 
\end{eqnarray}
is non-positive. The three following bounds hold outside an event of probability smaller than $\frac{L(\xi)}{n^2}$:
\begin{eqnarray*}
 \|\boldsymbol{\epsilon}\|_n^2 & \geq & 1 -4\frac{\sqrt{\log n}}{\sqrt{n}}\ ,\\
\|\Pi_{m}\boldsymbol{\epsilon}\|_n^2 & \leq & (1+\xi)\frac{d_m}{n}, \, \, \, \text{for all models $m$ of dimension } d_m >\frac{n}{\log n}\ ,\\
\langle\Pi_{m_{R,s}}^\perp\boldsymbol{\epsilon}_{m_{R,s}}, \Pi_{m_{R,s}}^\perp\boldsymbol{\epsilon}+2\boldsymbol{\epsilon}_{m_{R,s}}\rangle_n & \leq & 	 l(\theta_{m_{R,s}},\theta)\left[\frac{n-d_{m_{R,s}}}{n}+ 4\frac{\sqrt{(n-d_{m_{R,s}})\log n}}{n}+\frac{4\log n}{n}\right] \\& + & 4\sqrt{l(\theta_{m_{R,s}},\theta)}\sigma \frac{\sqrt{(n-d_{m_{R,s}})\log n}}{n}\ .
\end{eqnarray*}
Gathering these three inequalities we upper bound  (\ref{comparaison_critere2}) by 
\begin{eqnarray*}
 \sigma^2\frac{d_m}{n-d_m}\left[-2+8\sqrt{\frac{\log n}{n}} +(1+\xi)\left(\frac{n+d_m}{n}\right)\right]+ 2\sigma^2\frac{d_{m_{R,s}}}{n-d_{m_{R,s}}}+\hspace{2.5cm}\\ \hspace{2.5cm}+   \sigma^2L\left(1+\frac{d_{m_{R,s}}}{n}\right)\left(\frac{l(\theta_{m_{R,s}},\theta)}{\sigma^2}+\frac{\sqrt{l(\theta_{m_{R,s}},\theta)}}{\sigma}\right)\left(1+\sqrt{\frac{\log n }{n-d_{m_{R,s}}}}\right) \ .
\end{eqnarray*}
The dimension of any model $m\in \mathcal{M}_{\lfloor n/2 \rfloor}$ is assumed to be smaller than $n/2$ and the dimensions of the models $m$ considered are larger than $\frac{n}{\log n}$. For $\xi$ small enough and $n$ large enough, the previous expression is therefore upper bounded by
\begin{eqnarray}
 \sigma^2 \frac{2}{\log n}\left[\frac{3}{2}(1+\xi)-2+ 8\sqrt{\frac{\log n}{n}}\right]+ L\sigma^2 \left[\frac{R^{\frac{2}{1+s}}}{n^\frac{s}{1+s}}+ \frac{R^{\frac{1}{1+s}}}{n^\frac{a}{2(1+a)}}\right] \ . 	
\end{eqnarray}
For $n$ large enough, this last quantity is clearly non-positive. 

All in all, we have proved that for $n$ large enough outside an event of probability smaller than $\frac{L(R,s)}{n^2}$, it holds that
$$\log^2(n) <d_{m_*}< \frac{n}{\log n} \, \, \, \, \, \text{and}  \, \, \, \, \,\log^2(n) <d_{\widehat{m}}< \frac{n}{\log n}\ .$$

\end{proof}

\begin{proof}[Proof of Lemma \ref{controle_mhat}]

Arguing as in the proof of Theorem \ref{ordered_selection}, we upper bound 
\begin{eqnarray}
 -\overline{\gamma}_n(\widetilde{\theta}) - \gamma_n(\widetilde{\theta}) pen(\widehat{m}) +\sigma^2 +
\|\boldsymbol{\epsilon}\|^2_n \leq l(\theta_{\widehat{m}},\theta)  A_{\widehat{m}} + \sigma^2 B_{\widehat{m}} +(1-\kappa_2(n))l(\widetilde{\theta},\theta_{\widehat{m}})\ ,\label{decomposition_generale_optimale}
\end{eqnarray}	
 where $A_{\widehat{m}}$ and $B_{\widehat{m}}$ are respectively defined in (\ref{definition_A}) and in (\ref{definition_B}). We will fix the quantities $\kappa_1(n)$ and $\kappa_2(n)$ later. Besides, we define and bound the quantity $E_{\widehat{m}}$ as in (\ref{E_major}).

Applying Lemma \ref{concentration_chi2} and Lemma \ref{concentration_vp_wishart} and arguing as in the proofs of Lemma \ref{concentration_thrm1} and Lemma \ref{majoration_esperance_conditionelle}, there exists an event $\Omega_2$ of large probability
\begin{eqnarray*}
 \mathbb{P}(\Omega_1^c)\leq \exp\left[-n/8\right]+5\sum_{d=\log^2(n)}^{\frac{n}{\log n}}\exp\left[-\frac{2d}{\log n}\right]\leq \exp\left[-n/8\right]+\frac{5\log n}{2n^2(1-1/\log n)}\ , 
\end{eqnarray*}
 and such that conditionally on $\Omega_1\cap\Omega_2$,
\begin{eqnarray*}
 \frac{\|\Pi_{\widehat{m}}^{\perp}\boldsymbol{\epsilon}_{\widehat{m}}\|^2_n}{l(\theta_{\widehat{m}},\theta)} &\geq & \frac{n-d_{\widehat{m}}}{n} -
2\frac{\sqrt{2(n-d_{\widehat{m}})d_{\widehat{m}}/\log n}}{n}, \\ 
\frac{\|\Pi_{\widehat{m}}(\boldsymbol{\epsilon}+\boldsymbol{\epsilon}_{\widehat{m}})\|_n^2}{\sigma^2+ l(\theta_{\widehat{m}},\theta)} & \leq & \frac{d_{\widehat{m}}}{n}
+\frac{2\sqrt{2}d_{\widehat{m}}}{n\sqrt{\log n}} +
4\frac{ d_{\widehat{m}}}{n\log n},\\
\frac{\|\Pi_{\widehat{m}}^\perp(\boldsymbol{\epsilon}+\boldsymbol{\epsilon}_{\widehat{m}})\|_n^2}{\sigma^2+ l(\theta_{\widehat{m}},\theta)} & \geq & \frac{n-d_{\widehat{m}}}{n}
-2\frac{\sqrt{2(n-d_{\widehat{m}})d_{\widehat{m}}/\log n} }{n}\\
\varphi_{\text{max}}\left[ \left({\bf Z}^*_{\widehat{m}}{\bf Z}_{\widehat{m}}\right)^{-1}\right] & \leq & n^{-1}\left(1-\left(1+\sqrt{\frac{4}{\log n}}\right)\sqrt{\frac{d_{\widehat{m}}}{n}}\right)^{-2}\\
\|\boldsymbol{\epsilon}\|^2_n & \leq & 2\\
E_{\widehat{m}}
&  \leq & \frac{d_{\widehat{m}}+2\kappa_1^{-1}(n)}{n} + \frac{2}{n}\sqrt{\left[d_{\widehat{m}} +\left(2\kappa_1^{-1}(n)\right)^2\right]\frac{2d_{\widehat{m}}}{\log n} }+8\kappa_1^{-1}(n)\frac{d_{\widehat{m}}}{n\log n }\ .\\ 
\end{eqnarray*}
Gathering these six upper bounds, we are able to upper bound $A_{\widehat{m}}$ and $B_{\widehat{m}}$,
\begin{eqnarray*}
A_{\widehat{m}} & \leq & \kappa_1(n) +L_1\sqrt{\frac{d_{\widehat{m}}}{n\log n}} +\frac{d_{\widehat{m}}}{n}\left[-1 +L_2\sqrt{\frac{d_{\widehat{m}}}{(n-d_{\widehat{m}})\log n}}+ \kappa_2(n) \frac{1+L_3/\sqrt{\log(n)} }{\left[1-(1+\sqrt{\frac{4}{\log n} })\sqrt{\frac{d_{\widehat{m}}}{n})}\right]^2}\right]\ ,\\
B_{\widehat{m}} & \leq & \frac{d_{\widehat{m}}}{n}\left[-1+L_1\sqrt{\frac{d_{\widehat{m}}}{(n-d_{\widehat{m}})\log n}}+ \kappa_2(n) \frac{1+L_2/\sqrt{\log(n)}}{\left[1-(1+\sqrt{\frac{4}{\log n}})\sqrt{\frac{d_{\widehat{m}}}{n})}\right]^2}\right]\\& +&  L_3\frac{d_{\widehat{m}}}{n}\left[\frac{\kappa_1^{-1}(n)}{d_{\widehat{m}}}+\frac{\kappa_1^{-1}(n)}{\log n}+\frac{1}{\sqrt{\log(n)}}+\frac{\kappa_1^{-1}(n)}{\sqrt{\log(n)d_{\widehat{m}}}}\right]\ .
\end{eqnarray*}
Conditionally to the event $\Omega_1$, the dimension of $\widehat{m}$ is moderate. Setting $\kappa_1$ to $\frac{1}{\log n}$, we get
\begin{eqnarray*}
A_{\widehat{m}} & \leq & \frac{L_1}{\log n} +\frac{d_{\widehat{m}}}{n}\left[-1 +\frac{L_2}{\log n}+ \kappa_2(n) \frac{1+\frac{L_3}{\sqrt{\log n}}}{\left[1-\frac{L_4}{\sqrt{\log(n)}}\right]^2}\right]\ ,\\
B_{\widehat{m}} & \leq & \frac{d_{\widehat{m}}}{n}\left[-1 +\frac{L_1}{\log n}+ \kappa_2(n) \frac{1+\frac{L_2}{\sqrt{\log n}}}{\left[1-\frac{L_3}{\sqrt{\log(n)}}\right]^2} + \frac{L_4}{\sqrt{\log n}}\right]\ .
\end{eqnarray*}
Hence, there exists a sequence $\kappa_2(n)$ converging to one such that conditionally on $\Omega_1\cap \Omega_2$, $B_{\widehat{m}}$ is non-positive and $A_{\widehat{m}}$ is bounded by $\frac{L}{\log n}$ when $n$ is large enough. Coming back to the inequality (\ref{decomposition_generale_optimale}) yields 
\begin{eqnarray*}
 \left[-\overline{\gamma}_n(\widetilde{\theta}) -  \gamma_n(\widetilde{\theta})pen(\widehat{m}) -\sigma^2 +
\|\boldsymbol{\epsilon}\|^2_n\right]\mathbf{1}_{\Omega_1\cap\Omega_2} \leq  l(\widetilde{\theta},\theta)\left[\frac{L}{\log n}\vee \left(1-\kappa_2(n)\right)\right]\ ,
\end{eqnarray*}
which concludes the proof.
\end{proof}

\begin{proof}[Proof of Lemma \ref{controle_m*}]
We follow a similar approach to the previous proof.
\begin{eqnarray}
 \overline{\gamma}_n(\widehat{\theta}_{m_*}) + \gamma_n(\widehat{\theta}_{m_*})pen(m_*) +\sigma^2 - 
\|\boldsymbol{\epsilon}\|^2_n & \leq &  C_{m_*}l(\theta_{m_*},\theta) + D_{m_*}\sigma^2 + \kappa_2(n) l(\widehat{\theta}_{m_*},\theta_{m_*}),
\end{eqnarray}
where for any model $m'\in \mathcal{M}_{\lfloor n/2 \rfloor}$, $C_{m'}$ and $D_{m'}$ are respectively defined as 
\begin{eqnarray*}
C_{m'} & = &  \kappa_1(n) + \frac{\|\Pi_{m'}^\perp \boldsymbol{\epsilon}_{m'}\|_n^2}{l(\theta_{m'},\theta)}-1 + 2\frac{d_{m'}}{n-d_{m'}} \frac{\|\Pi_{m'}^\perp (\boldsymbol{\epsilon}+ \boldsymbol{\epsilon}_{m'})\|_n^2}{l(\theta_{m'},\theta)+\sigma^2}\\ &-& (1+\kappa_2(n))\frac{n}{\varphi_{\text{max}}\left( {\bf Z}^*_{m'}{\bf Z}_{m'}\right)} \frac{\|\Pi_m (\boldsymbol{\epsilon}
  +\boldsymbol{\epsilon}_{{m'}})\|^2_n}{l(\theta_{m'},\theta)+\sigma^2}\\
D_{m'} & = & \kappa_1^{-1}(n)\frac{\langle\Pi^{\perp}_{m'}\boldsymbol{\epsilon},\Pi^{\perp}_{m'}\boldsymbol{\epsilon}_{{m'}}
  \rangle^2_n}{\sigma^2l(\theta_{m'},\theta)} - \frac{\|\Pi_{{m'}}
  \epsilon\|^2_n}{\sigma^2} \\ &-& (1+\kappa_2(n)) \frac{n}
 {\varphi_{\text{max}}\left(  {\bf Z}^*_{{m'}}{\bf Z}_{{m'}}\right)}\frac{\|\Pi_{{m'}} (\boldsymbol{\epsilon}
   +\boldsymbol{\epsilon}_{{m'}})\|^2_n}{l(\theta_{m'},\theta)+\sigma^2}\nonumber+ 2\frac{d_{{m'}}}{n-d_{{m'}}}\frac{\|\Pi^{\perp}_{m'}( \boldsymbol{\epsilon}  +\boldsymbol{\epsilon}_{{m'}})\|^2_n}{l(\theta_{m'},\theta)+\sigma^2}\ .
\end{eqnarray*}
We fix $\kappa_1(n)=1/\log n$ whereas $\kappa_2(n)$ will be fixed later. Arguing as in the proof of Lemma \ref{controle_mhat}, there exists an event $\Omega_3$ of large probability
\begin{eqnarray*}
\mathbb{P}(\Omega_3^c)\leq \exp\left[-n/8\right]+5\sum_{d=\log^2(n)}^{\frac{n}{\log n}}\exp\left[-\frac{2d}{\log n}\right] \leq \exp\left[-n/8\right]+\frac{5\log n}{2n^2(1-1/\log(n))}\ ,
\end{eqnarray*}
such that conditionally on $\Omega_1\cap \Omega_3$, the two following bounds hold:
\begin{eqnarray*}
C_{m_*} & \leq & \frac{L_1}{\log n} + \frac{d_{m_*}}{n}\left[1 + \frac{L_2}{\log n} - (1+\kappa_2(n))\frac{1+L_3\sqrt{\frac{2}{\log n}}}{\left[1+\frac{L_4}{\sqrt{\log n}}\right]^2}\right]\ ,\\
D_{m_*} & \leq &  \frac{d_{m_*}}{n}\left[1+\frac{L_1}{\log n}+\frac{L_2}{\sqrt{\log n}}- (1+\kappa_2(n))\frac{1+L_3\sqrt{\frac{2}{\log n}}}{\left[1+\frac{L_4}{\sqrt{\log n}}\right]^2}\right]\ ,
\end{eqnarray*}
if $n$ is large. The main difference with the proof of Lemma \ref{controle_mhat} lies in the fact that we now control the largest eigenvalue of ${\bf Z}^*_m{\bf Z}_m$ thanks to the second result of Lemma \ref{concentration_vp_wishart}. There exists a sequence $\kappa_2(n)$  converging to $0$ such that conditionally on 
$\Omega_1\cap\Omega_3$, $D_{m_*}$ is non-positive and $C_{m_*}$ is bounded by $\frac{L}{\log n}$ when $n$ is large.
Coming back to  (\ref{decomposition_generale_optimale}) yields 
\begin{eqnarray*}
 \left[\overline{\gamma}_n(\widehat{\theta}_{m_*}) + pen(m_*) +\sigma^2 -
\|\epsilon\|^2_n\right]\mathbf{1}_{\Omega_1\cup\Omega_3} \leq  l(\widehat{\theta}_{m_*},\theta)\left[\frac{L}{\log n}\vee \kappa_2(n)\right]\ ,
\end{eqnarray*}
which concludes the proof.
\end{proof}

\subsection{Proof of Proposition \ref{Pente_ordonnee}}

\begin{proof}[Proof of Proposition \ref{Pente_ordonnee}]
The approach is similar to the proof of Proposition 1 in \cite{massart_pente}. For any model $m\in\mathcal{M}_{\lfloor n/2 \rfloor}$, let us define
\begin{eqnarray*}
 \Delta(m,m_{\lfloor n/2\rfloor}):= \gamma_n\left(\widehat{\theta}_{m_{\lfloor n/2\rfloor}}\right)[1+pen(m_{\lfloor n/2\rfloor})]-\gamma_n\left(\widehat{\theta}_m\right)[1+pen(m)]\ .
\end{eqnarray*}	
We shall prove that with large probability the quantity $\Delta(m,m_{\lfloor n/2\rfloor})$ is negative for any model $m$ of dimension smaller than $n/4$. Hence, with large probability $d_{\widehat{m}}$ will be larger than $n/4$. Let us fix a model $m$ of dimension smaller than $n/4$.\\

First, we use Expression (\ref{perte_empirique_estimateur}) to lower bound $\gamma_n(\widehat{\theta}_m)$.
\begin{eqnarray*}
\gamma_n\left(\widehat{\theta}_m\right)& = & \|\Pi^\perp_{m}\left(\boldsymbol{\epsilon}+\boldsymbol{\epsilon}_{m_{\lfloor n/2\rfloor}}\right)\|_n^2 +  
\|\Pi^\perp_{m}\left(\boldsymbol{\epsilon}_m-\boldsymbol{\epsilon}_{m_{\lfloor n/2\rfloor}}\right)\|_n^2+2\langle\Pi^\perp_{m}\left(\boldsymbol{\epsilon}+\boldsymbol{\epsilon}_{m_{\lfloor n/2\rfloor}}\right), \Pi^\perp_{m}\left(\boldsymbol{\epsilon}_m-\boldsymbol{\epsilon}_{m_{\lfloor n/2\rfloor}}\right)\rangle_n\\
& \geq & \|\Pi^\perp_{m}\left(\boldsymbol{\epsilon}+\boldsymbol{\epsilon}_{m_{\lfloor n/2\rfloor}}\right)\|_n^2 -\left\langle\Pi^\perp_{m}\left(\boldsymbol{\epsilon}+\boldsymbol{\epsilon}_{m_{\lfloor n/2\rfloor}}\right), \frac{\Pi^\perp_{m}\left(\boldsymbol{\epsilon}_m-\boldsymbol{\epsilon}_{m_{\lfloor n/2\rfloor}}\right)}{\|\Pi^\perp_{m}\left(\boldsymbol{\epsilon}_m-\boldsymbol{\epsilon}_{m_{\lfloor n/2\rfloor}}\right)\|_n}\right\rangle_n^2\ , 
\end{eqnarray*}
since $2ab\geq -a^2-b^2$ for any number $a$ and $b$. Hence, we may upper bound $\Delta(m,m_{\lfloor n/2\rfloor})$ by
\begin{eqnarray}\nonumber
 \Delta(m,m_{\lfloor n/2\rfloor})& \leq & \|\Pi^\perp_{m_{\lfloor n/2\rfloor}}\left(\boldsymbol{\epsilon}+\boldsymbol{\epsilon}_{m_{\lfloor n/2\rfloor}}\right)\|_n^2\left[pen(m_{\lfloor n/2\rfloor})-pen(m)\right] \\ & - & \left\|[\Pi^\perp_{m}-\Pi^\perp_{m_{\lfloor n/2\rfloor}}]\left(\boldsymbol{\epsilon}+\boldsymbol{\epsilon}_{m_{\lfloor n/2\rfloor}}\right)\right\|_n^2\left[1+pen(m)\right] \nonumber\\
& + &\left\langle\Pi^\perp_{m}\left(\boldsymbol{\epsilon}+\boldsymbol{\epsilon}_{m_{\lfloor n/2\rfloor}}\right), \frac{\Pi^\perp_{m}\left(\boldsymbol{\epsilon}_m-\boldsymbol{\epsilon}_{m_{\lfloor n/2\rfloor}}\right)}{\|\Pi^\perp_{m}\left(\boldsymbol{\epsilon}_m-\boldsymbol{\epsilon}_{m_{\lfloor n/2\rfloor}}\right)\|_n}\right\rangle_n^2\left[1+pen(m)\right]\ . \label{variable_complique}
\end{eqnarray}

Arguing as the proof of Lemma \ref{prte_basique}, we observe that $\|\Pi^\perp_{m_{\lfloor n/2\rfloor}}\left(\boldsymbol{\epsilon}+\boldsymbol{\epsilon}_{m_{\lfloor n/2\rfloor}}\right)\|_n^2*n/[\sigma^2+l(\theta_{m_{\lfloor n/2\rfloor}})]$ follows a $\chi^2$ distribution with $n-\lfloor n/2\rfloor$ degrees of freedom. Analogously, the random variable $\|[\Pi^\perp_{m}-\Pi^\perp_{m_{\lfloor n/2\rfloor}}]\left(\boldsymbol{\epsilon}+\boldsymbol{\epsilon}_{m_{\lfloor n/2\rfloor}}\right)\|_n^2*n/[\sigma^2+l(\theta_{m_{\lfloor n/2\rfloor})}]$ follows a $\chi^2$ distribution with $(d_{m_{\lfloor n/2 \rfloor }}-d_m)$ degrees of freedom. Let us turn to the distribution of the third term. Coming back to the definition of $\epsilon_m$, we observe that $$\epsilon_m- \epsilon_{m_{\lfloor n/2\rfloor}}=  Y-X\theta_m-(Y-X\theta_{m_{\lfloor n/2\rfloor}})= X(\theta_m-\theta_{m_{\lfloor n/2\rfloor}})\ .$$
Hence, $\epsilon_m- \epsilon_{m_{\lfloor n/2\rfloor}}$ is both independent of $X_m$ and of $\epsilon+\epsilon_{m_{\lfloor n/2\rfloor}}$. Consequently, by conditioning and unconditioning, we conclude that the random variable defined in (\ref{variable_complique}) follows (up to a $[\sigma^2+l(\theta_{m_{\lfloor n/2\rfloor}})]/n$ factor) a $\chi^2$ distribution with $1$ degree of freedom.

Once again, we apply Lemma \ref{concentration_chi2} and the classical deviation bound $\mathbb{P}\left(|\mathcal{N}(0,1)|\geq \sqrt{2x}\right)\leq 2e^{-x}$. Let $x$ be some positive number smaller than one that we shall fix later. There exists an event $\Omega_x$ of probability larger than $1-\exp(-nx/2)-3\exp(-(n/4-1)x)\frac{1}{1-e^{-x}}$ such for any model of dimension smaller than $n/4$,
\begin{eqnarray*}
\frac{\Delta(m,m_{\lfloor n/2\rfloor})}{\sigma^2+l(\theta_{m_{\lfloor n/2\rfloor}})}&\leq & \left(\frac{n-\lfloor n/2\rfloor}{n}\right)\left(1+2\sqrt{x}+2x\right)\left(pen(m_{\lfloor n/2\rfloor})-pen(m)\right)\\ & -& \frac{\lfloor n/2\rfloor -d_m}{n}(1-2\sqrt{x}-2x)(1+pen(m)) \ .
\end{eqnarray*}
We now replace the penalty terms by their values thanks to Assumption (\ref{condition_petite_penalite}). Conditionally to $\Omega_x$, we obtain that
\begin{eqnarray*}
 \frac{\Delta(m,m_{\lfloor n/2\rfloor})}{\sigma^2+l(\theta_{m_{\lfloor n/2\rfloor}})}&\leq &\frac{\lfloor n/2\rfloor -d_m}{n}\left\{4(1-\nu)(\sqrt{x}+x)\left[1+\frac{d_m}{n-d_m}\right]-\nu(1-2\sqrt{x}-2x)\right\} \ .
\end{eqnarray*}
Since the dimension of the model $m$ is smaller than $n/4$, $\frac{d_m}{n-d_m}$ is smaller than $1/3$. Hence, the last upper bound becomes 
\begin{eqnarray*}
 \frac{\Delta(m,m_{\lfloor n/2\rfloor})}{\sigma^2+l(\theta_{m_{\lfloor n/2\rfloor}})}&\leq &\frac{\lfloor n/2\rfloor -d_m}{n}\left\{\frac{16}{3}(1-\nu)(\sqrt{x}+x)-\nu(1-2\sqrt{x}-2x)\right\} \ .
\end{eqnarray*}
There exists some $x(\nu)$ such that conditionally on $\Omega_{x(\nu)}$, $\Delta(m,m_{\lfloor n/2\rfloor})$ is negative for any model $m$ of dimension smaller than $n/4$. Since $\mathbb{P}(\Omega_{x(\nu)}^c)$ goes exponentially fast with $\nu$ to 0, there exists some $n_0(\nu,\delta)$ such that for any $n$ larger than $n_0(\nu,\delta)$, $\mathbb{P}(\Omega_{x(\nu)}^c)$ is smaller than $\delta$. We have proved that with probability larger than $1-\delta$, the dimension of $\widehat{m}$ is larger than $n/4$.\\

Let us simultaneously lower bound the loss $l(\widehat{\theta}_m,\theta_m)$ for every model $m\in\mathcal{M}$ of dimension larger than $n/4$. In the sequel, $\succeq$  means "stochastically larger than". Thanks to (\ref{perte_biais}), we stochastically lower bound $l(\widehat{\theta}_m,\theta_m)$
\begin{eqnarray*}
l(\widehat{\theta}_m,\theta_m) & \geq & n\varphi_{\text{max}}\left( {\bf Z}_m^*{\bf Z}_m\right)^{-1}\|\Pi_{m}(\boldsymbol{\epsilon}+\boldsymbol{\epsilon}_{m})\|_n^2 \\ & \succeq & \varphi_{\text{max}}\left( n{\bf Z}_m^*{\bf Z}_m\right)^{-1}\|\Pi_{m}\boldsymbol{\epsilon}\|_n^2 ,
\end{eqnarray*}
where ${\bf Z}^*_m{\bf Z}_m$ follows a standard Wishart distribution with parameters $(n,d_m)$. Applying Lemma \ref{concentration_chi2} and Lemma \ref{concentration_vp_wishart} in order to simultaneously lower bound the loss $l(\widehat{\theta}_m,\theta_m)$, we find an event $\Omega'$ of probability larger than $1-\frac{2\exp(-n/4)}{1-e^{-1/16}}$, such that 
\begin{eqnarray*}
 l(\widehat{\theta}_m,\theta_m)\mathbf{1}_{\Omega'} & \geq & \left(1+\sqrt{\frac{d_m}{n}}+\sqrt{\frac{2d_m}{16n}}\right)^{-2}\frac{d_m}{2n}\sigma^2 \geq  \frac{d_m}{8n}\sigma^2\ ,
\end{eqnarray*}
for any model $m\in\mathcal{M}$ of dimension larger than $n/4$. On the event $\Omega_{x(\nu)}$, the dimension $d_{\widehat{m}}$ is larger than $n/4$. As a consequence, $l(\widetilde{\theta},\theta_{\widehat{m}})\mathbf{1}_{\Omega'\cap \Omega_{x(\nu)}}\geq \frac{\sigma^2}{32}$. All in all, we obtain
\begin{eqnarray*}
 \mathbb{E}\left[l(\widetilde{\theta},\theta)\right]& \geq &  l(\theta_{m_{\lfloor n/2\rfloor}},\theta)+ \mathbb{E}\left[\mathbf{1}_{\Omega'\cap \Omega_{x(\nu)}}l(\widetilde{\theta},\theta_{\widehat{m}})\right]\\
& \geq & l(\theta_{m_{\lfloor n/2\rfloor}},\theta)+\left[1-\mathbb{P}(\Omega^c_{x(\nu)})-\mathbb{P}(\Omega'^c)\right]\frac{\sigma^2}{32}\\
& \geq & l(\theta_{m_{\lfloor n/2\rfloor}},\theta)+L(\delta,\nu)\sigma^2\ ,
\end{eqnarray*}
if $n$ is larger than some $n_0(\nu,\delta)$.

\end{proof}

\subsection{Proofs of the minimax lower bounds}

All these minimax lower bounds are based on Birg\'e's version of Fano's Lemma \cite{birgelemma}.

\begin{lemma}\label{lemmebirge}\emph{\bf (Birg\'e's Lemma)} Let $(\Theta,d)$ be some pseudo-metric space and 
  $\{\mathbb{P}_\theta,\theta\in \Theta\}$ be some statistical model. Let $\kappa$ denote some absolute constant smaller than one. Then for any estimator $\widehat{\theta}$ and any finite subset $\Theta_1$ of $\Theta$,
  setting $\delta=min_{\theta,\theta'\in \Theta_1,\theta\neq \theta'}d(\theta,\theta')$, provided that
  $\max_{\theta,\theta'\in \Theta_1}\mathcal{K}(\mathbb{P}_{\theta},\mathbb{P}_{\theta'})\leq \kappa \log|\Theta_1|$, the
  following lower bound holds for every $p\geq 1$,
\begin{eqnarray*}
\sup_{\theta\in \Theta_1}\mathbb{E}_{\theta}[d^p(\widehat{\theta},\theta)]\geq 2^{-p}\delta^p (1-\kappa)\ .
\end{eqnarray*}
\end{lemma}

First, we compute the Kullback-Leibler divergence between the distribution $\mathbb{P}_{\theta}$ and $\mathbb{P}_{\theta'}$. 
\begin{eqnarray*}
 \mathcal{K}\left(\mathbb{P}_{\theta};\mathbb{P}_{\theta'}\right) = \mathcal{K}\left(\mathbb{P}_{\theta}(X);\mathbb{P}_{\theta'}(X)\right)+\mathbb{E}_{\theta}\left[\mathcal{K}\left(\mathbb{P}_{\theta}(Y|X);\mathbb{P}_{\theta'}(Y|X)\right)\left.\right|X\right]
\end{eqnarray*}
The two marginal distributions $\mathbb{P}_{\theta}(X)$ and $\mathbb{P}_{\theta'}(X)$ are equal. The conditional distributions $\mathbb{P}_{\theta}(Y|X)$ and $\mathbb{P}_{\theta'}(Y|X)$ are Gaussian with variance $\sigma^2$ and with mean respectively equal to $X\theta$ and $X\theta'$. Hence, the conditional Kullback-Leibler divergence  equals
$$\mathcal{K}\left(\mathbb{P}_{\theta}(Y|X);\mathbb{P}_{\theta'}(Y|X)\right)= \frac{\left[X(\theta-\theta')\right]^2}{2\sigma^2}\ .$$
Reintegrating with respect to $X$ yields 
\begin{eqnarray}\label{comparaison_kullback}
\mathcal{K}\left(\mathbb{P}_{\theta};\mathbb{P}_{\theta'}\right)=\frac{l(\theta',\theta)}{2\sigma^2} \text{ and }\  \mathcal{K}\left(\mathbb{P}^{\otimes n}_{\theta};\mathbb{P}^{\otimes n}_{\theta'}\right)=n\frac{l(\theta',\theta)}{2\sigma^2} \ .
\end{eqnarray}

\begin{proof}[Proof of Proposition \ref{minoration_ellipsoides}]

First, we need a lower bound of the minimax rate of estimation on a subspace of dimension $D$.

\begin{lemma}\label{lemme_minoration_dimension}
Let $D$ be some positive number smaller than $p$ and $r$ be some arbitrary positive number. Let $S_D$ be the set of vectors in $\mathbb{R}^p$ whose support in included in $\{1,\ldots,D\}$. Then, for any estimator $\widehat{\theta}$ of $\theta$,  
\begin{eqnarray}
\sup_{\theta\in S_D,\, l(0_p,\theta)\leq Dr^2}\mathbb{E}_{\theta}\left[l(\widehat{\theta},\theta)\right]\geq L D\left[r^2\wedge \frac{\sigma^2}{n}\right] \ .
\end{eqnarray}
\end{lemma}

Let us fix some $D\in \{1,\ldots,p\}$. Consider the set $\Theta_D:=\left\{\theta\in S_D, l(0_p,\theta)\leq a_D^2R^2\right\}$. Since the $a_j$'s are 
non increasing, it holds that $$\sum_{i=1}^p\frac{l(\theta_{m_{i-1}},\theta_{m_{i}})}{a_i^2}\leq \sum_{i=1}^D\frac{l(\theta_{m_{i-1}},\theta_{m_{i}})}{a_D^2}\leq \frac{l(0_p,\theta)}{a_D^2}\leq R^2\ ,$$
for any $\theta\in \Theta_D$. Hence $\Theta_D$ is included in $\mathcal{E}_a(R)$. Applying Lemma \ref{lemme_minoration_dimension}, we get
\begin{eqnarray*}
 \inf_{\widehat{\theta}}\sup_{\theta\in\mathcal{E}_a(R)}& \geq &LD \left[ \frac{a_D^2R^2}{D}\wedge\frac{\sigma^2}{n}\right]\\
& \geq & L \left[a_D^2R^2\wedge\frac{D\sigma^2}{n}\right] \ .
\end{eqnarray*}
Taking the supremum over $D$ in $\{1,\ldots, p\}$ enables to conclude.

\end{proof}

\begin{proof}[Proof of Lemma \ref{lemme_minoration_dimension}]
Let us assume first that $\Sigma=I_p$. Consider the hypercube $\mathcal{C}_D(r):=\{0,r\}^D\times \{0\}^{p-D}$. Thanks to (\ref{comparaison_kullback}), we upper bound the Kullback-Leibler divergence between the distributions $\mathbb{P}_\theta$ and $\mathbb{P}_{\theta'}$
\begin{eqnarray*}
  \mathcal{K}\left(\mathbb{P}^{\otimes n}_{\theta};\mathbb{P}^{\otimes n}_{\theta'}\right) \leq \frac{nDr^2}{2\sigma^2} \ ,
\end{eqnarray*}
where $\theta$ and $\theta'$ belong to $\mathcal{C}_D(r)$.
Then, we apply Varshamov-Gilbert's lemma (e.g. Lemma 4.7 in \cite{massartflour}) to the set $\mathcal{C}_D(r)$.
\begin{lemma}[Varshamov-Gilbert's lemma]
Let $\{0,1\}^D$ be equipped with Hamming distance $d_H$. There exists some subset $\Theta$ of $\{0,1\}^D$ with the following properties
$$d_H(\theta,\theta')>D/4\text{ for every }(\theta,\theta')\in \Theta^2\text{ with }\theta\neq \theta'\text{ and } \log |\Theta|\geq D/8\ .$$ 
\end{lemma}
Combining Lemma \ref{lemmebirge} with the set $\Theta$ defined in the last lemma yields
\begin{eqnarray*}
 \inf_{\widehat{\theta}}\sup_{\theta\in \mathcal{C}_D(r)}\mathbb{E}_{\theta}\left[d_H(\widehat{\theta},\theta)\right]\geq \frac{D}{16}\ ,
\end{eqnarray*}
provided that $\frac{nDr^2}{2\sigma^2}\leq D/16$. Coming back to the loss function $l(.,.)$ yields
\begin{eqnarray*}
 \inf_{\widehat{\theta}}\sup_{\theta\in \mathcal{C}_D(r)}\mathbb{E}_{\theta}\left[l(\widehat{\theta},\theta)\right]\geq LDr^2\ ,
\end{eqnarray*}
if $r^2\leq L\frac{\sigma^2}{n}$. Finally, we get
\begin{eqnarray*}
 \inf_{\widehat{\theta}}\sup_{\theta\in S_D,\ l(0_p,\theta)\leq Dr^2}\mathbb{E}_{\theta}\left[l(\widehat{\theta},\theta)\right]\geq LD\left[r^2\wedge \frac{\sigma^2}{n}\right]\ .
\end{eqnarray*}

If we no longer assume that the covariance matrix $\Sigma$ is the identity, we orthogonalize the sequence $X_i$ thanks to Gram-Schmidt process. Applying the previous argument to this new sequence of covariates allows to conclude.
\end{proof}

\begin{proof}[Proof of Corollary \ref{adaptation_ellipsoide}]
This result follows from the upper bound on the risk of $\widetilde{\theta}$ in Theorem \ref{ordered_selection} and the minimax lower bound of Proposition \ref{minoration_ellipsoides}. Let $\mathcal{E}_a(R)$ an ellipsoid satisfying $\frac{\sigma^2}{n}\leq R^2\leq \sigma^2n^\beta$, then $l(0_p,\theta)$ is smaller than $\sigma^2n^\beta$. By Theorem \ref{ordered_selection}, the estimator $\widetilde{\theta}$ defined with the collection $\mathcal{M}_{\lfloor n/2\rfloor\wedge p}$ and $pen(m)=K\frac{d_m}{n-d_m}$ satisfies
\begin{eqnarray*}
\mathbb{E}_{\theta}\left[l(\widetilde{\theta},\theta)\right]& \leq &
 L(K)\inf_{1
\leq i\leq \lfloor n/2\rfloor \wedge p}\left\{l(\theta_{m_i},\theta)+K\frac{i}{n-i}[\sigma^2+l(\theta_{m_i},\theta)]\right\}
 + L(K,\beta)\frac{\sigma^2}{n}\\
& \leq &L(K,\beta)\inf_{1
\leq i\leq \lfloor n/2\rfloor \wedge p}\left[l(\theta_{m_i},\theta)+\frac{i}{n}\sigma^2\right] \ .
\end{eqnarray*}
If $\theta$ belongs to $\mathcal{E}_a(R)$, then
\begin{eqnarray*}
 l(\theta_{m_i},\theta)\leq a_{i+1}^2\sum_{j=i+1}^p\frac{l(\theta_{m_j},\theta_{m_{j-1}})}{a_j^2}\leq R^2a_{i+1}^2\ , 
\end{eqnarray*}
since the $(a_i)$'s are increasing. It follows that 
\begin{eqnarray}\label{majoration_ellipsoide}
\mathbb{E}_{\theta}\left[l(\widetilde{\theta},\theta)\right]& \leq &
L(K,\beta)\inf_{1\leq i\leq \lfloor n/2\rfloor \wedge p}\left[R^2a_{i+1}^2+\frac{i}{n}\sigma^2\right] \ .
\end{eqnarray}
Let us define $i^*:=\sup\left\{1\leq i\leq p\, ,\,R^2a_i^2\geq \frac{\sigma^2i}{n}\right\}$,
with the convention $\sup \varnothing =0$. Since $R^2\geq \sigma^2/n$, $i^*$ is
larger or equal to one. By Proposition $\ref{minoration_ellipsoides}$, the minimax rates of estimation is lower bounded as follows
\begin{eqnarray*}
 \inf_{\widehat{\theta}}\sup_{\theta\in \mathcal{E}_a(R)}\mathbb{E}_{\theta}\left[l(\widehat{\theta},\theta)\right]\geq L\left[a_{i^*+1}^2R^2 \vee \frac{\sigma^2 i^*}{n}\right]\geq L\left[a_{i^*+1}^2R^2+\frac{\sigma^2i^*}{n}\right]\ .
\end{eqnarray*}
 If either $p\leq 2n$ or $a_{\lfloor n/2\rfloor +1}^2R^2\leq \sigma^2/2$, then $i^*$ is smaller or equal to $\lfloor n/2\rfloor \wedge p$ and we obtain thanks to (\ref{majoration_ellipsoide}) that
\begin{eqnarray*}
 \mathbb{E}_{\theta}\left[l(\widetilde{\theta},\theta)\right]& \leq &L(K,\beta)\left[a_{i^*+1}^2R^2+\frac{\sigma^2i^*}{n}\right]\\
& \leq &L(K,\beta)\inf_{\widehat{\theta}}\sup_{\theta\in \mathcal{E}_a(R)}\mathbb{E}\left[l(\widehat{\theta},\theta)\right] \ .
\end{eqnarray*}

\end{proof}

\begin{proof}[Proof of Proposition \ref{minoration_minimax}]
First, we use (\ref{comparaison_kullback}) to upper bound the Kullback-Leibler divergence between the distributions corresponding to parameters $\theta$ and $\theta'$ in the set $\Theta[k,p](r)$
\begin{eqnarray*}
  \mathcal{K}\left(\mathbb{P}^{\otimes n}_{\theta};\mathbb{P}^{\otimes n}_{\theta'}\right) \leq \frac{nkr^2}{2\sigma^2}\ ,
\end{eqnarray*}
since the covariates are i.i.d standard Gaussian variables. Let us state a combinatorial argument due to Birg\'e and Massart \cite{birge98}.
\begin{lemma}\label{lemmabirgemassart}
Let $\{0,1 \}^p$ be equipped with Hamming distance $d_H$ and given $1\leq k\leq p/4$, define $\{0,1\}_k^p:=\left\{x\in\{0,1\}^p:d_H(0,x)=k\right\}$. There exists some
subset $\Theta$ of $\{0,1 \}_k^p$ with the following properties
\begin{eqnarray*}
  d_H(\theta,\theta')>k/8 \text{ for every $(\theta,\theta') \in \Theta^2$
  with $\theta\neq\theta'$}\text{ and }
\log|\Theta| \geq   k/5\log \left(\frac{p}{k}\right)\ .
\end{eqnarray*}
\end{lemma}
Suppose that $k$ is smaller than $p/4$. Applying Lemma \ref{lemmebirge} with Hamming distance $d_H$ and the set $r\Theta$ introduced 
in Lemma \ref{lemmabirgemassart} yields
\begin{eqnarray}\label{minoration_hamming}
\inf_{\widehat{\theta}}\sup_{\theta \in \Theta[k,p](r) }\mathbb{E}_{\theta}\left[d_H\left(\widehat{\theta},\theta\right)\right]\geq  \frac{k}{16}\ ,\hspace{0.5cm} \text{ provided that }\hspace{0.5cm} \frac{nkr^2}{2\sigma^2} \leq  \frac{k}{10}\log\left(\frac{p}{k}\right)\ .
\end{eqnarray}
Since the covariates $X_i$ are independent and of variance 1, the lower bound (\ref{minoration_hamming}) is equivalent to
\begin{eqnarray*}
 \inf_{\widehat{\theta}}\sup_{\theta \in \Theta[k,p](r) }\mathbb{E}_{\theta}\left[l\left(\widehat{\theta},\theta\right)\right]\geq  \frac{kr^2}{16} .
\end{eqnarray*}
All in all, we obtain 
\begin{eqnarray*}
\inf_{\widehat{\theta}} \sup_{\theta \in \Theta[k,p](r) }\mathbb{E}_{\theta}\left[l\left(\widehat{\theta},\theta\right)\right]\geq Lk\left(r^2\wedge \frac{\log \left(\frac{p}{k}\right)}{n}\sigma^2\right)\ .
\end{eqnarray*}
Since $p/k$ is larger than $4$, we obtain the desired lower bound by changing the constant $L$:
\begin{eqnarray*}
\inf_{\widehat{\theta}} \sup_{\theta \in \Theta[k,p](r) }\mathbb{E}_{\theta}\left[l\left(\widehat{\theta},\theta\right)\right]\geq Lk\left(r^2\wedge \frac{1+\log \left(\frac{p}{k}\right)}{n}\sigma^2\right)\ .
\end{eqnarray*}
If $p/k$ is smaller than $4$, we know from the proof of Lemma \ref{lemme_minoration_dimension}, that 
\begin{eqnarray*}
 \inf_{\widehat{\theta}} \sup_{\theta \in \mathcal{C}_k(r) }\mathbb{E}_{\theta}\left[l\left(\widehat{\theta},\theta\right)\right]\geq Lk\left(r^2\wedge \frac{\sigma^2}{n}\right)\ .
\end{eqnarray*}
We conclude by observing that $\log(p/k)$ is smaller than $\log(4)$ and that $\mathcal{C}_k(r)$ is included in $\Theta[k,p](r)$.

\end{proof}

\begin{proof}[Proof of Proposition \ref{minoration_restricted_isometry}]
Assume first the covariates $(X_i)$ have a unit variance. If this is not the case, then one only has to rescale them.
By Condition (\ref{condition_restrited}), the Kullback-Leibler divergence between the distributions corresponding to parameters $\theta$ and $\theta'$ in the set $\Theta[k,p](r)$ satisfies
\begin{eqnarray*}
  \mathcal{K}\left(\mathbb{P}^{\otimes n}_{\theta};\mathbb{P}^{\otimes n}_{\theta'}\right) \leq (1+\delta)^2 \frac{nkr^2}{2\sigma^2} ,
\end{eqnarray*}
We recall that $\|.\|$ refers to the canonical norm in $\mathbb{R}^p$.
Arguing as in the proof of Proposition \ref{minoration_minimax}, we lower bound the risk of any estimator $\widehat{\theta}$ with the loss function $\|.\|$,
\begin{eqnarray*}
 \inf_{\widehat{\theta}}\sup_{\theta \in \Theta[k,p](r) }\mathbb{E}_{\theta}\left[\|\widehat{\theta}-\theta\|^2\right]\geq Lk\left(r^2\wedge \frac{1+\log \left(\frac{p}{k}\right)}{(1+\delta)^2n}\sigma^2\right)\ ,
\end{eqnarray*}
 Applying again Assumption (\ref{condition_restrited}) allows to obtain the desired lower bound on the risk
\begin{eqnarray*}
\inf_{\widehat{\theta}}\sup_{\theta \in \Theta[k,p](r) }\mathbb{E}_{\theta}\left[l(\widehat{\theta},\theta)\right]\geq Lk(1-\delta)^2\left(r^2\wedge \frac{1+\log \left(\frac{p}{k}\right)}{(1+\delta)^2n}\sigma^2\right)\ .
\end{eqnarray*}
\end{proof}

\begin{proof}[Proof of Proposition \ref{spatial}]
In short, we find a subset $\Phi\subset\{1,\ldots,p\}$ whose correlation matrix follows a $1/2$-Restricted Isometry Property of size $2k$. We then apply Proposition \ref{minoration_restricted_isometry} with the subset $\Phi$ of covariates.\\

We first consider the correlation matrix $\Psi_1(\omega)$. Let us pick a maximal subset $\Phi\subset\{1,\ldots p\}$ of points that are $\lceil \log(4k)/\omega \rceil$ spaced with respect to the toroidal distance. Hence, the cardinality of $\Phi$ is $\lfloor p\lceil \log(4k)/\omega \rceil^{-1}\rfloor$. Assume that $k$ is smaller than this quantity.
We call $C$ the correlation matrix of the points that belong to $\Phi$. Obviously, for any $(i,j)\in \Phi^2$, it holds that $|C(i,j)|\leq 1/(4k)$ if $i\neq j$. Hence, any submatrix of $C$ with size $2k$ is diagonally dominant and the sum of the absolute value of its non-diagonal elements is smaller than $1/2$. Hence, the eigenvalues of any submatrix of $C$ with size $2k$ lies between $1/2$ and $3/2$. The matrix $C$ therefore follows a $1/2$-Restricted Isometry Property of size $2k$. Consequently, we may apply Proposition \ref{minoration_restricted_isometry} with the subset of covariates $\Phi$ and the result follows. The second case is handled similarly.~\\ ~\\
{\bf Definition of the correlations}

Let us now justify why these correlations are well-defined when $p$ is an odd integer. We shall prove that the matrices $\Psi_1(\omega)$ and $\Psi_2(t)$ are non-negative. Observe that these two matrices are  symmetric and circulant. This means that there exists a family of numbers $(a_k)_{1\leq k \leq p}$ such that 
$$\Psi_1(\omega)[i,j]= a_{i-j \text{ mod }p}\ \text{ for any } 1\leq i,j\leq p\ .$$
Such matrices are known to be jointly diagonalizable in the same basis and their eigenvalues correspond to the discrete Fourier transform of $(a_k)$. More precisely, their eigenvalues $(\lambda_l)_{1\leq l\leq p}$ are expressed as
\begin{eqnarray}\label{definition_valeur_propre}
 \lambda_l:= \sum_{k=0}^{p-1}\exp\left(\frac{2i\pi kl}{p}\right) a_k\ .
\end{eqnarray}
We refer to \cite{rue} Sect. 2.6.2 for more details.
In the first example, $a_k$ equals $\exp(-\omega(k\wedge(p-k))$, whereas it equals $[1+(k\wedge (p-k))]^{-t}$ in the second example.~\\~\\
{\bf CASE 1:}
Using the expression (\ref{definition_valeur_propre}), one can compute $\lambda_l$.
\begin{eqnarray*}
 \lambda_l & =  &-1+ 2\sum_{k=0}^{(p-1)/2}\cos\left(\frac{2\pi kl }{p}\right) \exp(-k\omega) \\
& =&  -1 + 2\mathrm{Re}\left\{\sum_{k=0}^{(p-1)/2} \exp\left[k(i2\pi\frac{l}{p}-\omega\right]\right\}\\
&= & -1+ 2\mathrm{Re}\left\{\frac{1-e^{-\omega\frac{p+1}{2}}(-1)^le^{i2\pi\frac{l}{p}}}{1-e^{-\omega+i2\pi\frac{l}{p}}}\right\}\\
&= & -1 +2 \frac{1-e^{-\omega}\cos\left(\frac{2\pi l}{p}\right)+e^{-\omega(p+1)/2}(-1)^{l}\cos\left(\frac{\pi l}{p}\right)\left(e^{-\omega}-1\right)}{1+e^{-2\omega}-2e^{-\omega}\cos\left(\frac{2\pi l}{p}\right)}
\end{eqnarray*}
Hence, we obtain that
\begin{eqnarray*}
 \lambda_l\geq 0 \Leftrightarrow  1+ 2e^{-\omega(p+1)/2}(-1)^l\cos\left(\frac{\pi l}{p}\right)\left(e^{-\omega}-1\right)- e^{-2\omega}\geq 0\ .
\end{eqnarray*}
It is sufficient to prove that 
\begin{eqnarray*}
 1 - e^{-2\omega} +2e^{-\omega(p+3)/2}- 2e^{-\omega(p+1)/2}\geq 0\ .
\end{eqnarray*}
This last expression is non-negative if $\omega$ equals zero and is increasing with respect to $\omega$. We conclude that $\lambda_l$ is non-negative for any $1\leq l \leq p$. The matrix $\Psi_1(\omega)$ is therefore non-negative and defines a correlation.~\\~\\
{\bf CASE 2:}
Let us prove that the corresponding eigenvalues $\lambda_l$ are non-negative.
\begin{eqnarray*}
\lambda_l = -1+ 2\sum_{k=0}^{(p-1)/2}\cos\left(\frac{2\pi kl }{p}\right) (k+1)^{-t} 
\end{eqnarray*}
Using the following identity
$$(k+1)^{-t}= \frac{1}{\Gamma(t)}\int_{0}^{\infty}e^{-r(k+1)}r^{t-1}dr\ ,$$
we decompose $\lambda_l$ into a sum of integrals.
\begin{eqnarray*}
 \lambda_l =\frac{1}{\Gamma(t)}\left\{\int_{0}^{\infty}r^{t-1}e^{-r}\left[-1+2\sum_{k=0}^{(p-1)/2}\cos\left(\frac{2\pi kl}{p}\right)e^{-rk}\right]\right\}dr\ .
\end{eqnarray*}
The term inside the brackets corresponds to the eigenvalue for an exponential correlation with parameter $r$ (CASE 1). This expression is therefore non-negative for any $r\geq 0$. In conclusion, the matrix $\Psi_2(t)$ is non-negative and the correlation is defined.
\end{proof}

\section*{Appendix}

\begin{proof}[Proof of Lemma \ref{lemma_expression_perte_perte_empirique}]
We recall that $\gamma_n(\widehat{\theta}_m)=\|{\bf Y}-\Pi_{m}{\bf Y}\|_n^2$. Thanks to the definition (\ref{definition_epsilonm}) of $\epsilon$ and $\epsilon_m$ , we obtain the first result. Let us turn to the mean squared error $\gamma(\widehat{\theta}_m)$. In the following computation $\widehat{\theta}_m$ is considered as fixed and we only use that $\widehat{\theta}_m$ belongs to $S_m$. By definition,
\begin{eqnarray*}
 \gamma(\widehat{\theta}_m) & = &\mathbb{E}_{Y,X}\left[Y-X\widehat{\theta}_m\right]^2 = \sigma^2+ \mathbb{E}_{X}\left[X(\theta-\widehat{\theta}_m)\right]^2 \\
& = & \sigma^2 + l(\theta_m,\theta)+ l(\widehat{\theta}_m,\theta_m) \ ,
\end{eqnarray*}
since $\theta_m$ is the orthogonal projection of $\theta$ with respect to the inner product associated to the loss $l(.,.)$. We then derive that
\begin{eqnarray*}
l(\widehat{\theta}_m,\theta_m) & = & \mathbb{E}_{X_m}\left[X\left(\theta_m-\widehat{\theta}_m \right)\right]^2=\left(\theta_m-\widehat{\theta}_m \right)^*\Sigma \left(\theta_m-\widehat{\theta}_m \right)\ .
\end{eqnarray*}
Since $\widehat{\theta}_m$ is the least-squares estimator of $\theta_m$, it follows from (\ref{definition_epsilonm}) that 
\begin{eqnarray*}
 l(\widehat{\theta}_m,\theta_m) & = &(\boldsymbol{\epsilon}+\boldsymbol{\epsilon}_{m})^*{\bf X}_{m}({\bf X}^*_{m}{\bf X}_{m})^{-1}\Sigma_m({\bf X}^*_{m}{\bf X}_{m})^{-1}{\bf X}^*_{m}(\boldsymbol{\epsilon}+\boldsymbol{\epsilon}_{m})\ .
\end{eqnarray*}
We replace ${\bf X}_m$ by ${\bf Z}_m \sqrt{\Sigma_m}$ and therefore obtain
\begin{eqnarray*}
 l(\widehat{\theta}_m,\theta_m) & = &(\boldsymbol{\epsilon}+\boldsymbol{\epsilon}_{m})^*{\bf Z}_{m}({\bf Z}^*_{m}{\bf Z}_{m})^{-2}{\bf Z}^*_{m}(\boldsymbol{\epsilon}+\boldsymbol{\epsilon}_{m})\ .
\end{eqnarray*}
\end{proof}

\begin{proof}[Proof of Lemma \ref{prte_basique}]
 Thanks to Equation (\ref{perte_empirique_estimateur}), we know that $\gamma_n(\widehat{\theta}_m)= \|\Pi_m^{\perp}(\boldsymbol{\epsilon}+\boldsymbol{\epsilon}_m)\|_n^2$. The variance of $\epsilon+\epsilon_m$ is $\sigma^2 +l(\theta_m, \theta)$. Since $\boldsymbol{\epsilon}+\boldsymbol{\epsilon}_m$ is independent of ${\bf X}_m$, $\gamma_n(\widehat{\theta}_m)*n/[\sigma^2+l(\theta_m,\theta)]$  follows a $\chi^2$ distribution with $n-d_m$ degrees of freedom  and the result follows.

Let us turn to the expectation of $\gamma(\widehat{\theta}_m)$. By (\ref{perte_estimateur}), $\gamma(\widehat{\theta}_m)$ equals 
\begin{eqnarray*}
\gamma\left(\widehat{\theta}_m\right)& = &\sigma^2+l(\theta_m,\theta)+(\boldsymbol{\epsilon}+\boldsymbol{\epsilon}_{\widehat{m}})^*{\bf Z}_{\widehat{m}}({\bf Z}^*_{\widehat{m}}{\bf Z}_{\widehat{m}})^{-2}{\bf Z}^*_{\widehat{m}}(\boldsymbol{\epsilon}+\boldsymbol{\epsilon}_{\widehat{m}})\ ,
\end{eqnarray*}
following the arguments of the proof of Lemma \ref{lemma_expression_perte_perte_empirique}. Since $\epsilon + \epsilon_m$ and $X_m$ are independent, one may integrate with respect to $\boldsymbol{\epsilon}+\boldsymbol{\epsilon}_m$ 
\begin{eqnarray*}
\mathbb{E}\left[ \gamma(\widehat{\theta}_m)\right]= & \left[\sigma^2 + l(\theta_m,\theta)\right]\left\{1+ \mathbb{E}\left[tr\left( {\bf Z}^*_m{\bf Z}_m)^{-1}\right)\right]\right\}\ ,
\end{eqnarray*}
where the last term it the expectation of the trace of an inverse standard Wishart matrix of parameters $(n,d_m)$. Thanks to  \cite{rosen1988}, we know  that it equals $\frac{d_m}{n-d_m-1}$.
\end{proof}

\begin{proof}[Proof of Lemma \ref{concentration_chi2_fine}]
The random variable $\sqrt{\chi^2(d)}$ may be interpreted as a Lipschitz function with constant 1 on $\mathbb{R}^d$ equipped with the standard Gaussian measure. Hence, we may apply the Gaussian concentration theorem (see e.g. \cite{massartflour} Th. 3.4). For any $x>0$, 
\begin{eqnarray}\label{concentration_gaussienne_1}
 \mathbb{P}\left(\sqrt{\chi^2(d)}\leq \mathbb{E}\left[\sqrt{\chi^2(d)}\right]-\sqrt{2x}\right)\leq \exp(-x)\ .
\end{eqnarray}
In order to conclude, we need to lower bound $\mathbb{E}\left[\sqrt{\chi^2(d)}\right]$. Let us introduce the variable $Z:=1-\sqrt{\frac{\chi^2(d)}{d}}$. By definition, $Z$ is smaller or equal to one. Hence, we upper bound $\mathbb{E}(Z)$ as
\begin{eqnarray*}
 \mathbb{E}(Z) \leq \int_{0}^1\mathbb{P}(Z\geq t)dt \leq \int_{0}^{\sqrt{\frac{1}{8}}}\mathbb{P}(Z\geq t)dt + \mathbb{P}(Z \geq \sqrt{\frac{1}{8}})\ .
\end{eqnarray*}
Let us upper bound $\mathbb{P}(Z\geq t)$ for any $0\leq t\leq \sqrt{\frac{1}{8}}$ by applying Lemma \ref{concentration_chi2}
\begin{eqnarray*}
 \mathbb{P}\left(Z \geq t\right)& \leq & \mathbb{P}\left(\chi^2(d)\leq d\left[1-t\right]^2\right)\\
&\leq & \mathbb{P}\left(\chi^2(d)\leq d-2\sqrt{d}\sqrt{dt^2/2}\right)\leq \exp\left(-\frac{dt^2}{2}\right)\ ,
\end{eqnarray*}
since $t\leq 2-\sqrt{2}$. Gathering this upper bound with the previous inequality yields
\begin{eqnarray*}
 \mathbb{E}(Z) &\leq & \exp\left(-\frac{d}{16}\right) +\int_{0}^{+\infty}\exp\left(-\frac{dt^2}{2}\right)dt\\
&\leq &	\exp\left(-\frac{d}{16}\right) +\sqrt{\frac{\pi}{2d}}\ .
\end{eqnarray*}
Thus, we obtain $\mathbb{E}\left(\sqrt{\chi^2(d)}\right)\geq \sqrt{d}-\sqrt{d}\exp(-d/16)-\sqrt{\pi/2}$. Combining this lower bound with (\ref{concentration_gaussienne_1}) allows to conclude.
\end{proof}

\section*{Acknowledgements}
I gratefully thank Pascal Massart for many fruitful discussions. I also would like to thank the referee for his suggestions that led to an improvement of the paper.

\addcontentsline{toc}{section}{References}

\bibliographystyle{plain}

\bibliography{estimation}

\begin{thebibliography}{10}

\bibitem{akaike69}
H.~Akaike.
\newblock Statistical predictor identification.
\newblock {\em Ann. Inst. Statist. Math.}, 22:203--217, 1970.

\bibitem{akaike}
H.~Akaike.
\newblock A new look at the statistical model identification.
\newblock {\em IEEE Trans. Automatic Control}, AC-19:716--723, 1974.
\newblock System identification and time-series analysis.

\bibitem{arlot}
S.~Arlot.
\newblock Model selection by resampling penalization, 2008.
\newblock oai:hal.archives-ouvertes.fr:hal-00262478\_v1.

\bibitem{baraud08}
Y.~Baraud, C.~Giraud, and S.~Huet.
\newblock Gaussian model selection with an unknown variance.
\newblock {\em \textit{Ann. Statist.}}, 37(2):630--672, 2009.

\bibitem{bickeltsy08}
P.~Bickel, Y.~Ritov, and A.~Tsybakov.
\newblock Simultaneous analysis of {L}asso and {D}antzig selector.
\newblock {\em \textit{Annals of {S}tatistics (to appear)}}, 2009.

\bibitem{birgelemma}
L.~Birg{\'e}.
\newblock A new lower bound for multiple hypothesis testing.
\newblock {\em IEEE Trans. Inform. Theory}, 51(4):1611--1615, 2005.

\bibitem{birge98}
L.~Birg{\'e} and P.~Massart.
\newblock Minimum contrast estimators on sieves: exponential bounds and rates
  of convergence.
\newblock {\em Bernoulli}, 4(3):329--375, 1998.

\bibitem{birge2001}
L.~Birg{\'e} and P.~Massart.
\newblock Gaussian model selection.
\newblock {\em J. Eur. Math. Soc. (JEMS)}, 3(3):203--268, 2001.

\bibitem{massart_pente}
L.~Birg{\'e} and P.~Massart.
\newblock Minimal penalties for {G}aussian model selection.
\newblock {\em Probab. Theory Related Fields}, 138(1-2):33--73, 2007.

\bibitem{tsybakov_agregation_07}
F.~Bunea, A.~Tsybakov, and M.~Wegkamp.
\newblock Aggregation for {G}aussian regression.
\newblock {\em Ann. Statist.}, 35(4):1674--1697, 2007.

\bibitem{bunea07}
F.~Bunea, A.~Tsybakov, and M.~Wegkamp.
\newblock Sparsity oracle inequalities for the {L}asso.
\newblock {\em Electron. J. Stat.}, 1:169--194 (electronic), 2007.

\bibitem{CandesPlan07}
E.~Cand\`es and Y.~Plan.
\newblock Near-ideal model selection by $l_1$ minimization.
\newblock {\em \textit{Ann. Statist. (to appear)}}, 2009.

\bibitem{candes08}
E.~Candes and T.~Tao.
\newblock The {D}antzig selector: statistical estimation when {$p$} is much
  larger than {$n$}.
\newblock {\em Ann. Statist.}, 35(6):2313--2351, 2007.

\bibitem{candes2005}
E.~J. Candes and T.~Tao.
\newblock Decoding by linear programming.
\newblock {\em IEEE Trans. Inform. Theory}, 51(12):4203--4215, 2005.

\bibitem{cowell}
R.~G. Cowell, A.~P. Dawid, S.~L. Lauritzen, and D.~J. Spiegelhalter.
\newblock {\em Probabilistic networks and expert systems}.
\newblock Statistics for Engineering and Information Science. Springer-Verlag,
  New York, 1999.

\bibitem{cressie}
N.~A.~C. Cressie.
\newblock {\em Statistics for spatial data}.
\newblock Wiley Series in Probability and Mathematical Statistics: Applied
  Probability and Statistics. John Wiley \& Sons Inc., New York, 1993.
\newblock Revised reprint of the 1991 edition, A Wiley-Interscience
  Publication.

\bibitem{Davidson2001}
K.~R. Davidson and S.~J. Szarek.
\newblock Local operator theory, random matrices and {B}anach spaces.
\newblock In {\em Handbook of the geometry of Banach spaces, Vol. I}, pages
  317--366. North-Holland, Amsterdam, 2001.

\bibitem{pena}
V.~H. de~la Pe{\~n}a and E.~Gin{\'e}.
\newblock {\em Decoupling}.
\newblock Probability and its Applications (New York). Springer-Verlag, New
  York, 1999.
\newblock From dependence to independence, Randomly stopped processes.
  $U$-statistics and processes. Martingales and beyond.

\bibitem{giraud08}
C.~Giraud.
\newblock Estimation of {G}aussian graphs by model selection.
\newblock {\em \textit{Electron. {J}. {S}tat.}}, 2:542--563, 2008.

\bibitem{gneiting}
T.~Gneiting.
\newblock Power-law correlations, related models for long-range dependence and
  their simulation.
\newblock {\em J. Appl. Probab.}, 37(4):1104--1109, 2000.

\bibitem{buhlmann07}
M.~Kalisch and P.~B\"uhlmann.
\newblock Estimating high-dimensional directed acyclic graphs with the
  {PC}-algorithm.
\newblock {\em J. Mach. Learn. Res.}, 8:613--636, 2007.

\bibitem{Laurent98}
B.~Laurent and P.~Massart.
\newblock Adaptive estimation of a quadratic functional by model selection.
\newblock {\em Ann. Statist.}, 28(5):1302--1338, 2000.

\bibitem{lauritzen96}
S.~L. Lauritzen.
\newblock {\em Graphical models}, volume~17 of {\em Oxford Statistical Science
  Series}.
\newblock The Clarendon Press Oxford University Press, New York, 1996.
\newblock Oxford Science Publications.

\bibitem{mallows73}
C.L. Mallows.
\newblock Some comments on ${C}_p$.
\newblock {\em \textit{Technometrics}}, 15:661--675, 1973.

\bibitem{massartflour}
P.~Massart.
\newblock {\em Concentration inequalities and model selection}, volume 1896 of
  {\em Lecture Notes in Mathematics}.
\newblock Springer, Berlin, 2007.
\newblock Lectures from the 33rd Summer School on Probability Theory held in
  Saint-Flour, July 6--23, 2003, With a foreword by Jean Picard.

\bibitem{meinshausen06}
N.~Meinshausen and P.~B{\"u}hlmann.
\newblock High-dimensional graphs and variable selection with the lasso.
\newblock {\em Ann. Statist.}, 34(3):1436--1462, 2006.

\bibitem{rue}
H.~Rue and L.~Held.
\newblock {\em \textit{Gaussian {M}arkov Random Fields: {T}heory and
  Applications}}, volume 104 of {\em Monographs on Statistics and Applied
  Probability}.
\newblock Chapman \& Hall/CRC, London, 2005.

\bibitem{Sachs05}
K.~Sachs, O.~Perez, D.Pe{'}er, D.~A. Lauffenburger, and G.~P. Nolan.
\newblock Causal protein-signaling networks derived from multiparameter
  single-cell data.
\newblock {\em Science}, 308:523--529, 2005.

\bibitem{schafer05}
J.~Sch\"afer and K.~Strimmer.
\newblock An empirical bayes approach to inferring large-scale gene association
  network.
\newblock {\em \textit{Bioinformatics}}, {\bf 21}:754--764, 2005.

\bibitem{schwarz78}
G.~Schwarz.
\newblock Estimating the dimension of a model.
\newblock {\em Ann. Statist.}, 6(2):461--464, 1978.

\bibitem{shibata81}
R.~Shibata.
\newblock An optimal selection of regression variables.
\newblock {\em Biometrika}, 68(1):45--54, 1981.

\bibitem{stone}
C.~Stone.
\newblock An asymptotically optimal histogram selection rule.
\newblock In {\em Proceedings of the {B}erkeley conference in honor of {J}erzy
  {N}eyman and {J}ack {K}iefer, {V}ol.\ {II} ({B}erkeley, {C}alif., 1983)},
  Wadsworth Statist./Probab. Ser., pages 513--520, Belmont, CA, 1985.
  Wadsworth.

\bibitem{tib96}
R.~Tibshirani.
\newblock Regression shrinkage and selection via the lasso.
\newblock {\em J. Roy. Statist. Soc. Ser. B}, 58(1):267--288, 1996.

\bibitem{tibshirani96}
R.~Tibshirani.
\newblock Regression shrinkage and selection via the lasso.
\newblock {\em J. Roy. Statist. Soc. Ser. B}, 58(1):267--288, 1996.

\bibitem{tsybakov_minimax}
A.~Tsybakov.
\newblock Optimal rates of aggregation.
\newblock In {\em 16th {A}nnual {C}onference on {L}earning {T}heory}, volume
  2777, pages 303--313. Springer-{Verlag}, 2003.

\bibitem{villers1}
N.~Verzelen and F.~Villers.
\newblock Goodness-of-fit tests for high-dimensional gaussian linear models.
\newblock {\em Ann. Statist. (to appear)}, 2009.

\bibitem{rosen1988}
D.~von Rosen.
\newblock Moments for the inverted {W}ishart distribution.
\newblock {\em Scand. J. Statist.}, 15(2):97--109, 1988.

\bibitem{wainwright07}
M.~J. Wainwright.
\newblock Information-theoretic limits on sparsity recovery in the
  high-dimensional and noisy setting.
\newblock Technical Report 725, Department of {S}tatistics, {UC} {B}erkeley,
  2007.

\bibitem{wille04}
A.~Wille, P.~Zimmermann, E.~Vranova, A.~F\"{u}rholz, O.~Laule, S.~Bleuler,
  L.~Hennig, A.~Prelic, P.~von Rohr, L.~Thiele, E.~Zitzler, W.~Gruissem, and
  P.~B\"{u}hlmann.
\newblock Sparse graphical {G}aussian modelling of the isoprenoid gene network
  in \textit{arabidopsis thaliana}.
\newblock {\em \textit{{G}enome {B}iology}}, {\bf 5}(11), 2004.

\bibitem{Zhao06}
P.~Zhao and B.~Yu.
\newblock On model selection consistency of {L}asso.
\newblock {\em J. Mach. Learn. Res.}, 7:2541--2563, 2006.

\bibitem{zou_adaptive}
H.~Zou.
\newblock The adaptive lasso and its oracle properties.
\newblock {\em J. Amer. Statist. Assoc.}, 101(476):1418--1429, 2006.

\end{thebibliography}

\end{document}